\documentclass[11pt]{article}
\usepackage[utf8]{inputenc}
\usepackage{amstext, amsmath,latexsym,amsbsy,amssymb}
\usepackage{hyperref}
\usepackage{amsfonts}
\usepackage{mathtools}
\usepackage{graphicx}
\usepackage{amsfonts,graphicx}
\usepackage{amsfonts,amsmath,amssymb,amsthm}
\usepackage{epsf}
\usepackage{epstopdf}
\usepackage{graphics}
\usepackage{psfrag}
\usepackage{times}
\usepackage{epsfig}
\usepackage{subfig}
\usepackage{color}
\usepackage[titletoc]{appendix}
\usepackage{multirow}
\usepackage{hhline}
\usepackage{marg,ifthen}
\usepackage{bbm}
\usepackage{authblk}

\RequirePackage{natbib}
\long\def\comment#1{}
\renewcommand\vec[1]{\ensuremath\boldsymbol{#1}}

\newcommand{\Ocal}{\ensuremath{\mathcal{O}}}

\comment{
\setlength{\topmargin}{0 in}
\setlength{\textwidth}{5.5 in}
\setlength{\textheight}{8 in}
\setlength{\oddsidemargin}{0.5 in}
}

\theoremstyle{plain}

\newtheorem{theorem}{Theorem}

\numberwithin{theorem}{section}

\newtheorem{proposition}{Proposition}

\numberwithin{proposition}{section}

\newtheorem{lemma}{Lemma}

\numberwithin{lemma}{section}

\newtheorem{definition}{Definition}

\numberwithin{definition}{section}

\numberwithin{condition}{section}

\numberwithin{problem}{section}

\newtheorem{corollary}{Corollary}

\numberwithin{corollary}{section}

\numberwithin{assumption}{section}

\newtheorem{example}{Example}

\numberwithin{example}{section}

\numberwithin{conjecture}{section}

\theoremstyle{definition}

\numberwithin{remark}{section}

\newcommand{\eqsplit}[2][*]%
  {\ifthenelse{\equal{#1}{*} \or\equal{#1}{<???>}
                     \or\equal{#1}{nn} \or\equal{#1}{ } \or\equal{#1}{}}
    { \begin{align*}%
         #2 %
         \end{align*}%
        }
    {\begin{equation}\label{#1}\begin{split}\allowdisplaybreaks%
         #2%
         \end{split}\end{equation}
        }
  }

\renewenvironment{abstract}
 {\small
  \begin{center}
  \bfseries \abstractname\vspace{-.5em}\vspace{0pt}
  \end{center}
  \list{}{%
    \setlength{\leftmargin}{15mm}
    \setlength{\rightmargin}{\leftmargin}%
  }%
  \item\relax}
 {\endlist}
 
 \title{\textbf{Robust estimation of mixing measures in finite mixture models}}
\author[*]{Nhat Ho}
\author[**]{XuanLong Nguyen}
\author[**]{Ya'\kern-2pt{}acov Ritov}
\affil[*]{Department of EECS, University of California, Berkeley}
\affil[**]{Department of Statistics, University of Michigan, Ann Arbor}

\begin{document}
\maketitle




\begin{abstract}
In finite mixture models, apart from underlying mixing measure, true kernel density function
of each subpopulation in the data is, in many scenarios, unknown. Perhaps the most popular
approach is to choose some kernel functions that we empirically believe our data are
generated from and use these kernels to fit our models. Nevertheless, as long as the
chosen kernel and the true kernel are different, statistical inference of mixing measure
under this setting will be highly unstable. To overcome this challenge, we propose flexible
and efficient robust estimators of the mixing measure in these models, which are inspired
by the idea of minimum Hellinger distance estimator, model selection criteria, and
superefficiency phenomenon. We demonstrate that our estimators consistently recover the
true number of components and achieve the optimal convergence rates of parameter
estimation under both the well- and mis-specified kernel settings for any fixed bandwidth.
These desirable asymptotic properties are illustrated via careful simulation studies with
both synthetic and real data. 
\footnote{This research is supported in part by grants
NSF CCF-1115769, NSF CAREER DMS-1351362, and NSF CNS-1409303 to XN. YR gratefully acknowledges the partial support from NSF DMS-1712962.}


\end{abstract}
\section{Introduction}
\label{Section:introduction}
Finite mixture models have long been a popular modeling tool for making inference about the
heterogeneity in data, starting, at least, with the classical work of \cite{Pearson-1894} on
biometrical ratios on crabs. They have been used in numerous domains arising from biological,
physical, and social sciences. For a comprehensive introduction of statistical inference in
mixture models, we refer the readers to the books of \citep{Mclachlan-1988, Linsay-95,
Mclachlan-2000}.

In finite mixture models, we have our data $X_{1},X_{2}, \ldots, X_{n} \in \mathcal{X} \subset \mathbb{R}^{d}$ ($d \geq 1$) to be
i.i.d observations from a finite mixture with density function
\begin{eqnarray}
{\displaystyle
p_{G_{0},f_{0}}(x) := \int {f_{0}(x|\theta)}dG_{0}(\theta)} = \sum \limits_{i=1}^{k_{0}}{p_{i}^{0}f_{0}(x|\theta_{i}^{0})}, \nonumber
\end{eqnarray}
where $G_{0}=\sum_{i=1}^{k_{0}} {p_{i}^{0}\delta_{\theta_{i}^{0}}}$
is a true but \emph{unknown} mixing measure with exactly $k_{0}< \infty$ non-zero components 
and $\biggr\{f_{0}(x|\theta), \theta \in \Theta \subset \mathbb{R}^{d_{1}} \biggr\}$ is 
a true family of density functions, possibly \emph{partially unknown} where $d_{1} \geq 1$. There 
are essentially three principal challenges to the models that have attracted a great deal of 
attention from various researchers. They include estimating the true number of 
components $k_{0}$, understanding the behaviors of parameter estimation, i.e., the 
atoms and weights of true mixing measure $G_{0}$, and determining the underlying kernel 
density function $f_{0}$ of each subpopulation in the data. The first topic has been an 
intense area of research recently, see for example \citep{Roeder-1994, Escobar-1995, 
Gassiat-1997, Richardson-Green-1997, Gassiat-1999, Keribin-2000, James-2001,Chen-2012, Khalili-Chen-2012, Kasahara-2014}. However, the second and third topic 
have received much less attention due to their mathematical difficulty.

When the kernel density function $f_{0}$ is assumed to be known and $k_{0}$ is bounded 
by some fixed positive integer number, there have been considerable recent advances in the 
understanding of parameter estimation. In particular, when $k_{0}$ is known, i.e., the 
exact-fitted setting of finite mixtures, \cite{Ho-Nguyen-EJS-16} introduced a stronger 
version of classical parameter identifiability condition, which is first order identifiability 
notion, see Definition \ref{definition:strong_order_identifiability} below, to guarantee the 
standard convergence rate $n^{-1/2}$ of parameter estimation. When $k_{0}$ is 
unknown and bounded by a given number, i.e., the over-fitted setting of finite mixtures, 
\cite{Chen-95, Nguyen-13, Ho-Nguyen-EJS-16} utilized a notion of second order 
identifiability to establish convergence rate $n^{-1/4}$ of parameter 
estimation, which is achieved under some minimum distance based
estimator and the maximum likelihood estimator. Sharp minimax rates of parameter estimation for finite mixtures under strong identifiability
conditions in sufficiently high orders have been obtained by \cite{Jonas-2016}.
On the other hand, \cite{Ho-Nguyen-Ann-16, Ho-Nguyen-Ann-17} studied the 
singularity structure of finite mixture's parameter space and its impact on 
rates of parameter estimation when either the first or the second order 
identifiability condition fails to hold.  When the kernel density function $f_{0}$ is unknown, 
there have been some work utilizing the semiparametric approaches \citep{Bordes-2006, Hunter-2007}.  The salient feature of these work is to estimate $f_{0}$ from certain 
classes of functions with infinite dimension and achieve parameter estimation accordingly. 
However, it is usually very difficult to establish a strong guarantee for the identifiability of 
the parameters, even when the parameter space is simple \citep{Hunter-2007}. Therefore,  
semiparametric approaches for estimating true mixing measure $G_{0}$ are usually not 
reliable.

Perhaps, the most common approach to avoid the identifiability issue of $f_{0}$ is to choose some kernel function $f$ that we tactically believe the data are generated from, and utilize
that kernel function to fit the model to obtain an estimate of the true mixing measure $G_{0}$. In view of
its simplicity and prevalence, this is also the approach that we consider in this paper.
However, there is a fundamental challenge with that approach. It is likely that we are subject to a misspecified kernel setting, i.e., the chosen kernel $f$ and the true kernel $f_{0}$ are different. Hence, parameter estimation
under this approach will be potentially unstable. The robustness question is unavoidable. Our
principal goal in the paper therefore, is the construction of robust estimators of $G_{0}$
where the estimation of both its number of components and its parameters
is of interest. Moreover, these estimators should achieve the best possible convergence
rates under various assumptions of both the chosen kernel $f$ and the true kernel $f_{0}$. When the true number of
components $k_{0}$ is known, various robust methods had been proposed in the literature,
see for example \citep{Woodward-1984, Donoho-1988, Cutler-1996}. However, there are
scarce work for robust estimators when the true number of components $k_{0}$ is
unknown. Recently, \cite{Mijawoo-2006} proposed a robust estimator of the number of
components based on the idea of minimum Hellinger distance estimator \citep{Beran-1977, Lindsay-1994, He-2006, Karunamuni-2009}.
However, their work faced certain limitations. Firstly, their estimator greatly relied upon the
choice of kernel bandwidth. In particular, in order to achieve the consistency of the number of
components under the well-specified kernel setting, i.e., when $\left\{f\right\} = \left\{f_{0}\right\}$, the
bandwidth should vanish to 0 sufficiently slowly (cf. Theorem 3.1 in \citep{Mijawoo-2006}).
Secondly, the behaviors of parameter estimation from their estimator are difficult to interpret
due to the subtle choice of bandwidth. Last but not least, they also argued that their
method achieved the robust estimation of the number of components under the misspecified
kernel setting, i.e., when $\left\{f\right\} \neq \left\{f_{0}\right\}$. Not only did their statement lack
theoretical guarantee, their argument turned out to be also erroneous (see Section 5.3 in
\citep{Mijawoo-2006}). More specifically, they considered the chosen kernel $f$ to be Gaussian kernel while
the true kernel $f_{0}$ to be Student's t-kernel with a given fixed degree of freedom. The parameter space $
\Theta$ only consists of mean and scale parameter while the true number of components $k_{0}$ is $2$.
They demonstrated that their estimator still maintained the correct number of components
of $G_{0}$, i.e., $k_{0}=2$, under that setting of $f$ and $f_{0}$. Unfortunately, their
argument is not clear as their estimator should maintain the number of components of a
mixing measure $G_{*}$ which minimizes the appropriate Hellinger distance to the true
model. Of course, establishing the consistency of their parameter estimation under the
misspecified kernel setting is also a non-trivial problem.

Inspired by the idea of minimum Hellinger distance estimator, we propose flexible and 
efficient robust estimators of mixing measure $G_{0}$ that address all the limitations from 
the estimator in \citep{Mijawoo-2006}. Not only our estimators are computationally 
feasible and robust but they also possess various desirable properties, such as the flexible 
choice of bandwidth, the consistency of the number of components, and the best possible 
convergence rates of parameter estimation. In particular, the main contributions in this 
paper can be summarized as follows
\begin{itemize}
\item[(i)] We treat the well-specified kernel setting, i.e., $\left\{f\right\}=\left\{f_{0}
\right\}$, and the misspecified kernel setting, i.e., $\left\{f\right\} \neq \left\{f_{0}\right\}$, separately. Under both settings, we achieve the consistency of our estimators 
regarding the true number of components for any fixed bandwidth. Furthermore, when the 
bandwidth vanishes to 0 at an appropriate rate, the consistency of estimating the true 
number of components is also guaranteed.
\item[(ii)] For any fixed bandwidth, when $f_{0}$ is identifiable in the first order, the 
optimal convergence rate $n^{-1/2}$ of parameter estimation is established under the 
well-specified kernel setting. Additionally, when $f_{0}$ is not identifiable in the first order, 
we also demonstrate that our estimators still achieve the best possible convergence rates 
of parameter estimation.
\item[(iii)] Under the misspecified kernel setting, we prove that our estimators converge to 
a mixing measure $G_{*}$ that is closest to the true model under the Hellinger metric for 
any fixed bandwidth. When $f$ is first order identifiable and $G_{*}$ has finite number of 
components, the optimal convergence rate $n^{-1/2}$ is also established under mild 
conditions of both kernels $f$ and $f_{0}$. Furthermore, when $G_{*}$ has infinite 
number of components, some analyses about the consistency of our estimators are also 
discussed.
\end{itemize}
Finally, our argument, so far, has mostly focused on the setting when the true mixing 
measure $G_{0}$ is fixed with the sample size $n$. However, we note in passing that in a 
proper asymptotic model, $G_{0}$ may also vary with $n$ and converge to some 
probability distribution in the limit. Under the well-specified kernel setting, we verify that 
our estimators also achieve the minimax convergence rate of estimating $G_{0}$ under 
sufficiently strong condition on the identifiability of kernel density function $f_{0}$.
\paragraph{Paper organization} The rest of the paper is organized as follows. Section 
\ref{Section:prelim} provides preliminary backgrounds and facts. Section 
\ref{Section:minimum_Hellinger} presents an algorithm to construct a robust estimator of 
mixing measure based on minimum Hellinger distance estimator idea and model selection 
perspective. Theoretical results regarding that estimator are treated separately under both 
the well- and mis-specified kernel setting. Section \ref{Section:Another_approach} 
introduces another algorithm to construct a robust estimator of mixing measure based on 
superefficiency idea. Section \ref{Section:minimum_Hellinger_singular} addresses the 
performance of our estimators developed in previous sections under non-standard setting 
of our models. The theoretical results are illustrated via careful simulation studies with both 
synthetic and real data in Section \ref{Section:simulation}. Discussions regarding possible 
future work are presented in Section \ref{Section:discussion} while self-contained proofs of 
key results are given in Section \ref{Section:proof} and proofs of the remaining results are given in the Appendices.
 \paragraph{Notation} Given two densities $p, q$ (with respect to the Lebesgue measure $\mu$), the total variation distance is given by $V(p,q)={\displaystyle \dfrac{1}{2}\int {|p(x)-q(x)|}d\mu(x)}$. Additionally, the square Hellinger distance is given by $h^{2}
(p,q)={\displaystyle \dfrac{1}{2}\int {(\sqrt{p(x)}-\sqrt{q(x)})^{2}}d\mu(x)}$.

For any $\kappa=(\kappa_{1},\ldots,\kappa_{d_{1}}) \in \mathbb{N}^{d_{1}}$, we 
denote $\dfrac{\partial^{|\kappa|}{f}}{\partial{\theta^{\kappa}}}(x|\theta) =
\dfrac{\partial^{|\kappa|}{f}}{\partial{\theta_{1}^{\kappa_{1}}\ldots
\partial{\theta_{d_{1}}^{\kappa_{d_{1}}}}}}(x|\theta)$ where $\theta=(\theta_{1},
\ldots,\theta_{d_{1}})$. Additionally, the expression $a_{n} \gtrsim b_{n}$ will be used 
to denote the inequality up to a constant multiple where the value of the constant is 
independent of $n$. We also denote $a_{n} \asymp b_{n}$ if both $a_{n} \gtrsim b_{n}$ 
and $a_{n} \lesssim b_{n}$ hold. Finally, for any $a, b \in \mathbb{R}$, we denote $a 
\vee b = \max \left\{a,b\right\}$ and $a \wedge b = \min \left\{a,b\right\}$.

\section{Background}
\label{Section:prelim}
Throughout the paper, we assume that the parameter space $\Theta$ is a compact subset of $
\mathbb{R}^{d_{1}}$. For any kernel density function $f$ and mixing measure $G$, we define
\begin{eqnarray}
p_{G,f}(x) := \int {f(x|\theta)}dG(\theta). \nonumber
\end{eqnarray}
Additionally, we denote $\mathcal{E}_{k_{0}}: = \mathcal{E}_{k_{0}}(\Theta)$ the space of discrete mixing measures with exactly $k_{0}$ distinct support points on $
\Theta$ and $\mathcal{O}_{k} := \mathcal{O}_{k}(\Theta)$ the space of discrete mixing measures with at most $k$ distinct
support points on $\Theta$. Additionally, denote $\mathcal{G}:=\mathcal{G}(\Theta) =
\mathop {\cup } \limits_{k \in \mathbb{N}_{+}}{\mathcal{E}_{k}}$ the set of all discrete
measures with finite supports on $\Theta$. Finally, $\overline{\mathcal{G}}$ denotes the space of all discrete measures (including those with countably infinite supports) on $\Theta$.

As described in the introduction, a principal goal of our paper is to construct robust estimators that maintain the consistency of the number of components and the best possible convergence rates of parameter estimation. As in~\cite{Nguyen-13}, our tool-kit for analyzing the identifiability and convergence
of parameter estimation in mixture models is based on Wasserstein distance, which
can be defined as the optimal cost of moving masses transforming one probability
measure to another~\citep{Villani-09}. In particular, consider a mixing measure $G=\mathop {\sum }\limits_{i=1}^{k}{p_{i}\delta_{\theta_{i}}}$, where $\textbf{p}=(p_{1},p_{2},\ldots,p_{k})$ denotes the proportion vector.
Likewise, let
$G' = \sum_{i=1}^{k'}p'_i \delta_{\theta'_i}$. A coupling between $\vec{p}$ and
$\vec{p'}$ is a joint distribution $\vec{q}$ on $[1\ldots,k]\times [1,\ldots, k']$, which
is expressed as a matrix
$\vec{q}=(q_{ij})_{1 \leq i \leq k,1\leq j \leq k'} \in [0,1]^{k \times k'}$
with margins
$\mathop {\sum }\limits_{m=1}^{k}{q_{mj}}=p_{j}'$ and $\mathop {\sum  }\limits_{m=1}^{k'}{q_{im}}=p_{i}$ for any $i=1,2,\ldots,k$ and $j=1,2,\ldots,k'$.
We 
use $\mathcal{Q}(\vec{p},\vec{p'})$ to denote the space of all such couplings. For any $r \geq 1$, the $r$-th order Wasserstein distance between
$G$ and $G'$ is given by
\begin{eqnarray}
W_{r}(G,G') & = &
 \inf_{\vec{q} \in \mathcal{Q}(\vec{p},\vec{p'})}\biggr ({\mathop {\sum }\limits_{i,j}{q_{ij}(\|\theta_{i}-\theta_{j}'\|)^{r}}}
\biggr )^{1/r}, \nonumber
\end{eqnarray}
where $\|\cdot\|$  denotes the $l_{2}$ norm for
elements in $\mathbb{R}^{d_{1}}$. It is simple to argue that if a sequence of probability measures $G_{n} \in \Ocal_{k_{0}}$
converges to $G_{0} \in \mathcal{E}_{k_{0}}$ under the $W_{r}$ metric
at a rate $\omega_{n} = o(1)$ then there exists a subsequence of $G_{n}$ such that the set of
atoms of $G_{n}$ converges to the $k_{0}$ atoms of $G_{0}$, up to a permutation of the atoms, at the same rate $\omega_{n}$.

We recall now the following key definitions that are used to analyze the behavior of mixing measures in finite mixture models (cf. \citep{Ho-Nguyen-Ann-17,Jonas-2016}). We start with
\begin{definition} \label{definition:uniform_Lipschitz_condition}
We say the family of densities $\left\{f(x|\theta), \theta \in \Theta \right\}$ is \emph{uniformly Lipschitz} up to
the order $r$, for some $r \geq 1$, if $f$ as a function of $\theta$ is differentiable up to the order $r$ and its partial derivatives with respect to $\theta$
satisfy the following inequality
\begin{eqnarray}
\sum_{|\kappa| = r} \biggr | \bigg (\frac{\partial^{|\kappa|} f}{\partial \theta^\kappa}
(x|\theta_1) -
 \frac{\partial^{|\kappa|} f}{\partial \theta^\kappa}(x|\theta_2) \biggr ) \gamma^\kappa\biggr |
\leq C \|\theta_1-\theta_2\|_r^\delta \|\gamma\|_r^r \nonumber
\end{eqnarray}
for any $\gamma \in \mathbb{R}^{d_{1}}$ and for some positive constants $\delta$ and $C$ independent of $x$ and $\theta_{1},\theta_{2} \in \Theta$. Here,  $\gamma^\kappa=\prod \limits_{i=1}^{d_{1}}{\gamma_{i}^{\kappa_{i}}}$ where $\kappa=(\kappa_{1},\ldots,\kappa_{d_{1}})$.
\end{definition}
We can verify that many popular classes of density functions, including Gaussian, Student's t, and skewnormal family, satisfy the uniform Lipschitz condition up to any order $r \geq 1$.

The classical identifiability condition entails that the family of density function $\left\{f(x|\theta), \theta \in \Theta \right\}$ is identifiable if for any $G_{1}, G_{2} \in \mathcal{G}$, $p_{G_{1},f}(x)=p_{G_{2},f}(x)$ almost surely implies that $G_{1} \equiv G_{2}$ \citep{Teicher-1961}. To be able to establish convergence rates of parameters, we have to utilize the following stronger notion of identifiability:
\begin{definition} \label{definition:strong_order_identifiability}
For any $r \geq 1$, we say that the family $\left\{f(x|\theta), \theta \in \Theta \right\}$ (or in short, $f$) is
\emph{identifiable in the $r$-th order} if $f(x|\theta)$ is differentiable up to the $r$-th order in $\theta$
and the following holds
\begin{itemize}
\item[A1.] For any $k \geq 1$, given $k$ different elements
$\theta_{1}, \ldots, \theta_{k} \in \Theta$.
If we have $\alpha_{\eta}^{(i)}$ such that for almost all $x$
\begin{eqnarray}
\sum \limits_{l=0}^{r}{\sum \limits_{|\eta|=l}{\sum \limits_{i=1}^{k}{\alpha_{\eta}^{(i)}\dfrac{\partial^{|\eta|}{f}}{\partial{\theta^{\eta}}}(x|\theta_{i})}}}=0 \nonumber
\end{eqnarray}
then $\alpha_{\eta}^{(i)}=0$ for all $1 \leq i \leq k$ and $|\eta| \leq r$.
\end{itemize}
\end{definition}
\paragraph{Rationale of the first order identifiability:} Throughout the paper, we denote
$I(G,f) : =E(l_{G}l_{G}^{T})$ the Fisher information matrix of the  kernel density $f$ at
a given mixing measure $G$. Here, $l_{G} : =\dfrac{\partial}{\partial{G}}\log p_{G,f}
(x)$ is the score function, where $\dfrac{\partial}{\partial{G}}$ denotes the derivatives
with respect to all the components and masses of $G$. The first order identifiability of
$f$ is an equivalent way to say that the Fisher information matrix $I(G,f) $ is non-singular for any $G$. Now, under the first order identifiability and the first order uniform
Lipschitz condition on $f$, a careful
investigation of Theorem 3.1 and Corollary 3.1 in \cite{Ho-Nguyen-EJS-16} yields the
following result:
\begin{proposition} \label{proposition:first_order_identifiability}
Suppose that the density family $\left\{f(x|\theta,\theta \in \Theta \right\}$ is
identifiable in the first order and uniformly Lipschitz up to the first order. Then, there is a
positive constant $C_{0}$ depending on $G_{0}$, $\Theta$, and $f$ such that as long as $G \in
\mathcal{O}_{k_{0}}$ we have
\begin{eqnarray}
h(p_{G,f},p_{G_{0},f}) \geq C_{0}W_{1}(G,G_{0}). \nonumber
\end{eqnarray}
\end{proposition}
Note that, the result of Proposition \ref{proposition:first_order_identifiability} is slightly
stronger than that of Theorem 3.1 and Corollary 3.1 in \cite{Ho-Nguyen-EJS-16} as it
holds for any $G \in \mathcal{O}_{k_{0}}$ instead of only for any $G \in \mathcal{E}
_{k_{0}}$ as in these later results. The
first order identifiability
property of kernel density function $f$ implies that any estimation method that yields the
convergence rate $n^{-1/2}$ for $p_{G_{0},f}$ under the Hellinger distance, the induced
rate of convergence for the mixing measure $G_{0}$ is $n^{-1/2}$ under $W_{1}$
distance.

\section{Minimum Hellinger distance estimator with non-singular Fisher information matrix}
\label{Section:minimum_Hellinger}
Throughout this section, we assume that two families of density functions $\left\{f_0(x|\theta),
\theta \in \Theta \right\}$ and  $\left\{f(x|\theta), \right. \\ \left. \theta \in \Theta \right\}$ are
identifiable in the first order and admit the uniform Lipschitz condition up to the first order.
Now, let $K$ be any fixed multivariate density function and $K_{\sigma}(x)=\dfrac{1}
{\sigma^{d}}K\left(\dfrac{x}{\sigma}\right)$ for any $\sigma>0$. We define
\begin{eqnarray}
f*K_{\sigma}(x|\theta) : = \int {f(x-y|\theta)K_{\sigma}(y)}dy \nonumber
\end{eqnarray}
for any $\theta \in \Theta$. The notation $f*K_{\sigma}$ can be thought as the 
convolution of the density family $\left\{f(x|\theta), \theta \in \Theta \right\}$ with the kernel function $K_{\sigma}$. From that definition, we further define
\begin{eqnarray}
p_{G,f}*K_{\sigma}(x) := \sum \limits_{i=1}^{k}{p_{i}f*K_{\sigma}(x|\theta_{i})}=\sum \limits_{i=1}^{k}{p_{i}\int {f(x-y|\theta_{i})K_{\sigma}(y)}dy} \nonumber
\end{eqnarray}
for any discrete mixing measure $G = \sum \limits_{i=1}^{k}{p_{i}\delta_{\theta_{i}}}$ in 
$\overline{\mathcal{G}}$. For the convenience of our argument later, we also denote that 
$p_{G,f}*K_{\sigma} := p_{G,f}$ as long as $\sigma=0$. Now, our approach to define a 
robust estimator of $G_{0}$ is inspired by the minimum Hellinger distance estimator 
\citep{Beran-1977} and the model selection criteria. Indeed, we have the following algorithm
\paragraph{Algorithm 1:} Let $C_{n}n^{-1/2} \to 0$ and $C_{n}n^{1/2} \to \infty$ as $n \to \infty$.
\begin{itemize}
\item[•] Step 1: Determine $\widehat{G}_{n,m}=\mathop {\arg \min} \limits_{G \in \mathcal{O}_{m}}{h(p_{G,f}*K_{\sigma_{1}},P_{n}*K_{\sigma_{0}})}$ for any $m \geq 1$.
\item[•] Step 2: Choose
\begin{eqnarray}
\widehat{m}_{n}=\mathop {\inf } \biggr\{m \geq 1: h(p_{\widehat{G}_{n,m},f}*K_{\sigma_{1}},P_{n}*K_{\sigma_{0}}) \leq h(p_{\widehat{G}_{n,m+1},f}*K_{\sigma_{1}},P_{n}*K_{\sigma_{0}}) \nonumber \\
+C_{n}n^{-1/2}\biggr\}, \nonumber
\end{eqnarray}

\item[•] Step 3: Let $\widehat{G}_{n}=\widehat{G}_{n,\widehat{m}_{n}}
$ for each $n$.
\end{itemize}
Note that, $\sigma_{1} \geq 0$ and $\sigma_{0}>0$ are two chosen bandwidths that 
control the amount of smoothness that we would like to add to $f$ and $f_{0}$ 
respectively. The choice of $C_{n}$ in Algorithm 1 is to guarantee that $\widehat{m}_{n}
$ is finite. Additionally, it can be chosen based on certain model selection criteria. For 
instance, if we use BIC, then $C_{n} = \sqrt{(d_{1}+1)\text{log} n/2}$ where $d_{1}$ is 
the dimension of parameter space. Algorithm 1 is in fact the generalization of the
algorithm considered in \cite{Mijawoo-2006} when $\sigma_{1}=0$ and $\sigma_{0} > 
0$. In particular, with the adaptation of notations as those in our paper, the algorithm in \cite{Mijawoo-2006} can be stated as follows.
\paragraph{Woo-Sriram (WS) Algorithm:}
\begin{itemize}
\item[•] Step 1: Determine $\overline{G}_{n,m}=\mathop {\arg \min} \limits_{G \in \mathcal{O}_{m}}{h(p_{G,f},P_{n}*K_{\sigma_{0}})}$ for any $n,m \geq 1$.
\item[•] Step 2: Choose
\begin{eqnarray}
\overline{m}_{n}=\mathop {\inf }{\biggr\{m \geq 1: h(p_{\overline{G}_{n,m},f},P_{n}*K_{\sigma_{0}}) \leq h(p_{\overline{G}_{n,m+1},f},P_{n}*K_{\sigma_{0}}) +C_{n}'n^{-1/2}\biggr\}}, \nonumber
\end{eqnarray}
where $C_{n}'n^{-1/2} \to 0$.
\item[•] Step 3: Let $\overline{G}_{n}=\overline{G}_{n,\overline{m}_{n}}
$ for each $n$.
\end{itemize}
The main distinction between our estimator and Woo-Sriram's (WS) estimator is that we also allow the convolution of mixture density $p_{G,f}$ with $K_{\sigma_{1}}$. This double convolution trick in Algorithm 1 was also considered in
\cite{James-2001} to construct the consistent estimation of mixture complexity. However,
their work was based on the Kullback-Leibler (KL) divergence rather than the Hellinger distance and was restricted to only the choice that $\sigma_{1}=\sigma_{0}$. Under the
misspecified kernel setting, i.e., $\left\{f\right\} \neq \left\{f_{0}\right\}$, the estimation of mixing measure $G_{0}$ from KL divergence
can be highly unstable. Additionally, \cite{James-2001} only worked with the Gaussian case of true kernel function $f_{0}$, while in many
applications, it is not realistic to expect that $f_{0}$ is Gaussian. To demonstrate the 
advantages of our proposed estimator $\widehat{G}_{n}$ over WS estimator $
\overline{G}_{n}$, we will provide careful theoretical studies of these estimators in the 
paper. For readers' convenience, we provide now a brief summary of our analyses of the 
convergence behaviors of $\widehat{G}_{n}$ and $\overline{G}_{n}$.

Under the well-specified setting, i.e., $\left\{f\right\}=\left\{f_{0}\right\}$, the optimal 
choice of $\sigma_{1}$ and $\sigma_{0}$ in Algorithm 1 is $\sigma_{1}=\sigma_{0}>0$, 
which guarantees that $G_{0}$ is the exact mixing measure that we seek for. Now, the 
double convolution trick in Algorithm 1 is sufficient to yield
the optimal convergence rate $n^{-1/2}$ of $\widehat{G}_{n}$ to $G_{0}$ for any fixed 
bandwidth $\sigma_{0}>0$ (cf. Theorem 
\ref{theorem:convergence_rate_mixing_measure}). The core idea of this result comes from the fact that $P_{n}*K_{\sigma_{0}}(x)
$ is an unbiased estimator of $p_{G_{0},f_{0}}*K_{\sigma_{0}}(x)$ for all $x \in 
\mathcal{X}$. It guarantees that $h(P_{n}*K_{\sigma_{0}},p_{G_{0},f_{0}}*K_{\sigma_{0}})=O_{p}(n^{-1/2})
$ under suitable conditions of $f_{0}$ when the bandwidth $\sigma_{0}$ is fixed. 
However, it is not the case for WS Algorithm. Indeed, we demonstrate later in Section 
\ref{Section:MS_algorithm} that for any fixed bandwidth $\sigma_{0}>0$, $\overline{G}
_{n}$ converges to $\overline{G}_{0}$ where $\overline{G}_{0}=\mathop {\arg \min} 
\limits_{G \in \overline{\mathcal{G}}}{h(p_{G,f_{0}},p_{G_{0},f_{0}}*K_{\sigma_{0}})}$ 
under certain conditions of $f_{0}$, $K$, and $\overline{G}_{0}$. Unfortunately, $
\overline{G}_{0}$ can be very different from $G_{0}$ even if they may have the same 
number of components. Therefore, even though we may still be able to recover the true 
number of components with WS Algorithm, we hardly can obtain exact estimation of true 
parameters. It shows that Algorithm 1 is more appealing than WS Algorithm under the well-specified kernel setting with fixed bandwidth $\sigma_{0}>0$.

When we allow the bandwidth $\sigma_{0}$ to vanish to 0 as $n \to
\infty$ under the well-specified kernel setting with $\sigma_{1}=\sigma_{0}$, we are able 
to guarantee that $\widehat{m}_{n} \to k_{0}$ almost surely when $n\sigma_{0}^{d} 
\to 0$ (cf. Proposition \ref{proposition:consistency_components_well_specified}). This 
result is also consistent with the result $\overline{m}_{n} \to k_{0}$ almost surely from 
Theorem 1 in \cite{Mijawoo-2006} under the same assumptions of $\sigma_{0}$. 
Moreover, under these conditions of bandwidth $\sigma_{0}$, both the estimators $
\widehat{G}_{n}$ and $\overline{G}_{n}$ converge to $G_{0}$ as $n \to \infty$. 
However, instead of obtaining the exact convergence rate $n^{-1/2}$ of $\widehat{G}
_{n}$ to $G_{0}$, we are only able to achieve its convergence rate to be $n^{-1/2}$ up 
to some logarithmic factor when the bandwidth $\sigma_{0}$ vanishes to 0 sufficiently 
slowly. It is mainly due to the fact that our current technique is based on the evaluation of 
the term $h(P_{n}*K_{\sigma_{0}},p_{G_{0},f_{0}}*K_{\sigma_{0}})$, which may not 
converge to 0 at the exact rate $n^{-1/2}$ when $\sigma_{0} \to 0$. The situation is 
even worse for the convergence rate of $\overline{G}_{n}$ to $G_{0}$ as it relies not only 
on the evaluation of $h(P_{n}*K_{\sigma_{0}},p_{G_{0},f_{0}}*K_{\sigma_{0}})$ but 
also on the convergence rate of $\overline{G}_{0}$ to $G_{0}$, which depends strictly on 
the vanishing rate of $\sigma_{0}$ to 0. Therefore, the convergence rate of $\overline{G}
_{n}$ in WS Algorithm may be much slower than $n^{-1/2}$. As a consequence, our 
estimator in Algorithm 1 may be also more efficient than that in WS Algorithm when the 
bandwidth $\sigma_{0}$ is allowed to vanish to 0.

Under the misspecified kernel setting, i.e., $\left\{f\right\} \neq \left\{f_{0}\right\}$, the 
double convolution technique in Algorithm 1 continues to be useful for studying the 
convergence rate of $\widehat{G}_{n}$ to $G_{*}$ where $G_{*}=\mathop {\arg \min}
\limits_{G \in \overline{\mathcal{G}}}{h(p_{G,f}*K_{\sigma_{1}},p_{G_{0},f_{0}}
*K_{\sigma_{0}})}$. Unlike the well-specified kernel setting, we allow $\sigma_{1}$ and $
\sigma_{0}$ to be different under the misspecified kernel setting. It is particularly useful if 
we can choose $\sigma_{1}$ and $\sigma_{0}$ such that two families $\left\{f*K_{\sigma_{1}}\right\}$ and 
$\left\{f_{0}*K_{\sigma_{0}}\right\}$ are identical under Hellinger distance. The consequence is that 
$G_{*}$ and $G_{0}$ will be identical under Wasserstein distance, which means that our 
estimator is still able to recover true mixing measure even though we choose the wrong 
kernel to fit our data. Granted, the misspecified setting \emph{means} that we are usually not in
such a fortunate situation, but our theory entails a good performance for our estimate when $f*K_{\sigma_1} \approx
f_0*K_{\sigma_{0}}$.
Now, for the general choice of $\sigma_{1}$ and $\sigma_{0}$, as 
long as $G_{*}$ has finite number of components, we are able to establish the 
convergence rate $n^{-1/2}$ of $\widehat{G}_{n}$ to $G_{*}$ under sufficient 
conditions on $f,f_{0}$, and $K$ (cf. Theorem 
\ref{theorem:convergence_rate_mixing_measure_misspecified_strong}). However, when 
the number of components of $G_{*}$ is infinite, we are only able to achieve the 
consistency of the number of components of $\widehat{G}_{n}$ (cf. Proposition 
\ref{proposition:infinite_mixture_complexity_misspecified_kernel}). Even though we do 
not have specific result regarding the convergence rate of $\widehat{G}_{n}$ to $G_{*}$ 
under that setting of $G_{*}$, we also provide important insights regarding that 
convergence in Section \ref{Section:infinite_component}.


\subsection{Well-specified kernel setting} \label{Section:well_specified_kernel}
In this section, we consider the setting that $f_{0}$ is known, i.e., $\{f\}=\{f_{0}\}$. Under that setting, the optimal choice of $\sigma_{1}$ and $\sigma_{0}$ is $\sigma_{1}=\sigma_{0}>0$ to guarantee that $G_{0}$ is the exact mixing measure that we estimate.
As we have seen from the discussion in Section \ref{Section:prelim}, the first order
identifiability condition plays an important
role to obtain the convergence rate $n^{-1/2}$ of parameter estimation. Since
Algorithm 1 relies on investigating the variation around kernel function $f_{0}
*K_{\sigma_{0}}$ in the limit, we
would like to guarantee that $f_{0}*K_{\sigma_{0}}$ is identifiable in the first order for any
$\sigma_{0}>0$. It appears that we have a mild
condition of $K$ such that the first order identifiability of $f_{0}*K_{\sigma_{0}}$ is
maintained.
\begin{lemma} \label{lemma:first_order_convolution}
Assume that $\widehat{K}(t) \neq 0$ for almost all $t \in \mathbb{R}^{d}$ where $
\widehat{K}(t)$ is the Fourier transform of kernel function $K$. Then, as long as $f_{0}$ is
identifiable in the first order, we obtain that $f_{0}*K_{\sigma_{0}}$ is identifiable in the first
order for any $\sigma_{0}>0$.
\end{lemma}
The assumption $\widehat{K}(t) \neq 0$ is very mild. Indeed, popular choices of $K$ to 
satisfy that assumption include the Gaussian and Student's t kernel. Inspired by the result 
of Lemma \ref{lemma:first_order_convolution}, we have the following result establishing 
the convergence rate of $\widehat{G}_{n}$ to $G_{0}$ under $W_{1}$ distance for any fixed bandwidth $\sigma_{0}>0$.
\begin{theorem} \label{theorem:convergence_rate_mixing_measure} Let $\sigma_{0}>0$ be given.
\begin{itemize}
\item[(i)] If $f_{0}*K_{\sigma_{0}}$ is identifiable, then $\widehat{m}_{n} \to k_{0}$ almost surely.
\item[(ii)] Assume further the following conditions
\begin{itemize}
\item[(P.1)] The kernel function $K$ is chosen such that $f_{0} * K_{\sigma_{0}}$ is
also identifiable in the first order and admits a uniform Lipschitz property up to the first
order.
\item[(P.2)] $\Psi(G_{0},\sigma_{0}) :={\displaystyle \int{\dfrac{g(x|G_{0},\sigma_{0})}
{p_{G_{0},f_{0}}*K_{\sigma_{0}}(x)}}\textrm{d}x < \infty}$ where we have that \\ $g(x|G_{0},\sigma_{0}) := {\displaystyle \int{K_{\sigma_{0}}^{2}(x-y)p_{G_{0},f_{0}}(y)}dy}$.
\end{itemize}
Then, we obtain
\begin{eqnarray}
W_{1}(\widehat{G}_{n},G_{0})=O_{p}\biggr(\sqrt{\dfrac{\Psi(G_{0},\sigma_{0})}{C_{1}^{2}(\sigma_{0})
}}n^{-1/2}\biggr) \nonumber
\end{eqnarray}
where $C_{1}(\sigma_{0}) :=\inf \limits_{G \in \mathcal{O}_{k_{0}}}
{\dfrac{h(p_{G,f_{0}}*K_{\sigma_{0}},p_{G_{0},f_{0}}*K_{\sigma_{0}})} {W_{1}(G,G_{0})}}$.
\end{itemize}
\end{theorem}
\paragraph{Remarks:}
\begin{itemize}
\item[(i)] Condition (P.1) is satisfied by many kernel functions $K$ according to Lemma \ref{lemma:first_order_convolution}. By assumption (P.1) and Proposition \ref{proposition:first_order_identifiability}, we obtain the following bound
\begin{eqnarray}
h(p_{G,f_{0}}*K_{\sigma_{0}},p_{G_{0},f_{0}}*K_{\sigma_{0}}) \gtrsim W_{1}(G,G_{0}) \label{eqn:comment_convergence_rate_wellspecify} \nonumber
\end{eqnarray}
for any $G \in \mathcal{O}_{k_{0}}$, i.e., $C_{1}(\sigma_{0})>0$.
\item[(ii)] Condition (P.2) is mild. One easy example for such setting is when $f_{0}$ and
$K$ are both Gaussian kernels. In fact, when $\left\{f_{0}(x|\eta,\tau), (\eta,\tau) \in 
\Theta \right\}$ is a family of univariate Gaussian distributions where $\eta$ and $\tau$ are location and scale parameter respectively and $K$ is a standard univariate
Gaussian kernel, we achieve
\begin{eqnarray}
\Psi(G_{0},\sigma_{0}) & = & \sum \limits_{i=1}^{k_{0}}{\int {\dfrac{p_{i}^{0}\int {K_{\sigma_{0}}^{2}(x-y)f_{0}(y|\eta_{i}^{0},\tau_{i}^{0})}dy}{p_{G_{0},f_{0}}*K_{\sigma_{0}}(x)}}}\textrm{d}x \nonumber \\
& < & \sum \limits_{i=1}^{k_{0}}{\int {\dfrac{\int {K_{\sigma_{0}}^{2}(x-y)f_{0}(y|\eta_{i}^{0},\tau_{i}^{0})}dy}{f_{0}*K_{\sigma_{0}}(x|\eta_{i}^{0},\tau_{i}^{0})}}}\textrm{d}x \nonumber \\
& \propto & \sum \limits_{i=1}^{k_{0}}{\left((\tau_{i}^{0})^{2}+\sigma_{0}^{2}\right)/\sigma_{0}^{2}} < \infty. \nonumber
\end{eqnarray}
Another specific example is when $f_{0}$ and $K$ are both Cauchy kernels or generally 
Student's t kernels with odd degree of freedom. However, assumption (P.2) may fail when 
$K$ has much shorter tails than $f_0$. For example, if $f_{0}$ is Laplacian kernel and $K$ 
is Gaussian kernel, then $\Psi(G_{0},\sigma_{0}) = \infty$.
\end{itemize}
\paragraph{Comments on $\widehat{G}_{n}$ as $\sigma_{0} \to 0$:} To avoid the 
ambiguity, we now denote $\left\{\sigma_{0,n}\right\}$ as the sequence of varied 
bandwidths $\sigma_{0}$. The following result shows the consistency of $\widehat{m}
_{n}$ under specific conditions on $\sigma_{0,n} \to 0$.
\begin{proposition}\label{proposition:consistency_components_well_specified}
Given a sequence of bandwidths $\left\{\sigma_{0,n}\right\}$ such that $\sigma_{0,n} \to 0$ and $n\sigma_{0,n}^{d} \to \infty$ as $n \to \infty$. If $f_{0}$ is identifiable, then $\widehat{m}_{n} \to k_{0}$ almost surely.
\end{proposition}
Our previous result with Theorem \ref{theorem:convergence_rate_mixing_measure}
shows that the parametric $n^{-1/2}$ rate of convergence of $\widehat{G}_{n}$ to
$G_{0}$ is achieved for any fixed $\sigma_{0}>0$. It would be more elegant to argue that
this rate is achieved for some sequence $\sigma_{0,n} \to 0$. However, this cannot be
done with the current technique employed in the proof of Theorem
\ref{theorem:convergence_rate_mixing_measure}. In particular, even though we still can
guarantee that $\lim \limits_{\sigma_{0,n} \to 0} C_{1}(\sigma_{0,n})>0$ (cf. Lemma
\ref{lemma:bandwith_vary_hellinger_bound} in Appendix B), the technical difficulty is
that $\Psi(G_{0},\sigma_{0,n})=O(\sigma_{0,n}^{-\beta(d)})$ for some $\beta(d)>0$
depending on
$d$ as $\sigma_{0,n} \to 0$. As a consequence, whatever the sequence of bandwidths $
\sigma_{0,n} \to 0$ we choose, we will be only able to obtain the convergence rate
$n^{-1/2}$ up to the logarithmic term of $\widehat{G}_{n}$ to $G_{0}$. It can be
thought as the limitation of the elegant technique employed in Theorem
\ref{theorem:convergence_rate_mixing_measure}. We leave the exact convergence rate
$n^{-1/2}$ of $\widehat{G}_{n}$ to $G_{0}$ under the setting $\sigma_{0,n} \to 0$ for
the future work.

\subsection{Misspecified kernel setting} \label{Section:adaptiveminimumdistancemisspecified}
In the previous section, we assume the well-specified kernel setting, i.e., $\left\{f\right\}=
\left\{f_{0}\right\}$, and achieve the convergence rate $n^{-1/2}$ of $\widehat{G}_{n}
$ to $G_{0}$ under mild conditions on $f_{0}$ and $K$ and the choice that $\sigma_{1}
=\sigma_{0}$ for any fixed bandwidth $\sigma_{0}>0$. However, the well-specified kernel 
assumption is often violated in practice, i.e., the chosen kernel $f$ may be different from 
the true kernel $f_{0}$. Motivated by this challenge, in this section we consider the setting 
when $\left\{f\right\} \neq \left\{f_{0}\right\}$. Additionally, we also take into account 
the case when the chosen bandwidths $\sigma_{1}$ and $\sigma_{0}$ may be different. We will demonstrate 
that the convergence rate of $\widehat{G}_{n}$ is still desirable under certain 
assumptions on $f, f_{0}$, and $K$. Furthermore, we also argue that the choice that $
\sigma_{1}$ and $\sigma_{0}$ are different can be very useful under the case when two 
families of density functions $\left\{f*K_{\sigma_{1}}(x|\theta), \theta \in \Theta \right
\}$ and $\left\{f_{0}*K_{\sigma_{0}}(x|\theta), \theta \in \Theta \right\}$ are identical. 
Due to the complex nature of misspecified kernel setting, we will only study the behavior of 
$\widehat{G}_{n}$ when the bandwidth $\sigma_{1} \geq 0$ and $\sigma_{0}>0$ are 
fixed in this section. Now, for fixed bandwidths $\sigma_{1}$, $\sigma_{0}$ assume that 
there exists a discrete mixing measure $G_{*}$ that minimizes the Hellinger distance 
between $p_{G,f}*K_{\sigma_{1}}$ and $p_{G_{0},f_{0}}*K_{\sigma_{0}}$, i.e.,
\begin{eqnarray}
G_{*} := \mathop {\arg \min}\limits_{G \in \overline{\mathcal{G}}}{h(p_{G,f}*K_{\sigma_{1}},p_{G_{0},f_{0}}*K_{\sigma_{0}})}. \nonumber
\end{eqnarray}
As $G_{*}$ may not be unique, we denote
\begin{eqnarray}
\mathcal{M} := \left\{G_{*} \in \overline{\mathcal{G}}: G_{*}\; \textrm{is a minimizer of} \; {h(p_{G,f}*K_{\sigma_{1}},p_{G_{0},f_{0}}*K_{\sigma_{0}})}\right\}. \nonumber
\end{eqnarray}
When $f*K_{\sigma_{1}}=f_{0}*K_{\sigma_{0}}$, 
it is clear that $G_{0}$ is an element of $\mathcal{M}$ such that it has the minimum 
number of components among all the elements in $\mathcal{M}$. To further investigate $
\mathcal{M}$ under general setting of $f, f_{0}, \sigma_{1}, \sigma_{0}$, and $K$, we 
start with the following key property of elements $G_{*}$ in $\mathcal{M}$:
\begin{lemma}\label{lemma:key_inequality_misspecified_setting}
For any $G \in \overline{\mathcal{G}}$ and $G_{*} \in \mathcal{M}$, there holds
\begin{eqnarray}
\int {p_{G,f}*K_{\sigma_{1}}(x)\sqrt{\dfrac{p_{G_{0},f_{0}}*K_{\sigma_{0}}(x)}{p_{G_{*},f}*K_{\sigma_{1}}(x)}}}\textrm{d}x \leq \int {\sqrt{p_{G_{*,f}}*K_{\sigma_{1}}(x)}
\sqrt{p_{G_{0},f_{0}}*K_{\sigma_{0}}(x)}}\textrm{d}x.
 \label{eqn:property_modified_hellinger_distance}
\end{eqnarray}
\end{lemma}
Equipped with this bound, we have the following important property of $\mathcal{M}$.
\begin{lemma} \label{lemma:key_property_non_unique}
For any two elements $G_{1,*},G_{2,*} \in \mathcal{M}$, we obtain $p_{G_{1,*},f}*K_{\sigma_{1}}(x)=p_{G_{2,*},f}*K_{\sigma_{1}}(x)$ for almost surely $x \in \mathcal{X}$.
\end{lemma}
Now, we consider the partition of $\mathcal{M}$ into the union of $\mathcal{M}_{k}=\left
\{G_{*} \in \mathcal{M}: \ G_{*} \ \text{has} \ k \ \text{elements}\right\}$ where $k \in 
[1,\infty]$. Let $k_{*}:=k_{*}(\mathcal{M})$ be the minimum number $k \in [1,\infty]$ 
such that $\mathcal{M}_{k}$ is non-empty. We divide our argument into two distinct 
settings of $k_{*}$: $k_{*}$ is finite and $k_{*}$ is infinite.

\subsubsection{Finite $k_{*}$:} By Lemma \ref{lemma:key_property_non_unique}, $
\mathcal{M}_{k_{*}}$ will have exactly one element $G_{*}$ provided that $f*K_{\sigma}$ is 
identifiable. Furthermore, $\mathcal{M}_{k}$ is empty for all $k_{*}<k<\infty$. However, 
it is possible that $\mathcal{M}_{\infty}$ still contains various elements . Due to the parsimonious nature of Algorithm 1 and the result of Theorem 
\ref{theorem:convergence_rate_mixing_measure_misspecified_strong}, we will be able to 
demonstrate that $\widehat{G}_{n}$ still converges to the unique element $G_{*} \in 
\mathcal{M}_{k_{*}}$ at the optimal rate $n^{-1/2}$ regardless of the behavior of $
\mathcal{M}_{\infty}$.

For the simplicity of our later argument under that setting of $k_{*}$, we denote by $G_{*}$ 
the unique element in $\mathcal{M}_{k_{*}}$. As we mentioned earlier, one simple 
example for $k_{*}<\infty$ is when $\left\{f*K_{\sigma_{1}}\right\}=\left\{f_{0}
*K_{\sigma_{0}}\right\}$. Another example is when $f$ is a location-scale family and 
$f_{0}$ is a finite mixture of $f$ while $\sigma_{1}=\sigma_{0}>0$. In particular,  $f
(x|\eta,\tau)=\dfrac{1}{\tau}f\bigr((x-\eta)/\tau\bigr)$ where $\eta$ and $\tau$ 
are location and scale parameters respectively. 
Additionally, $f_{0}(x)=\sum_{i=1}^{m}
{p_{i}^{*}f(x|\eta_{i}^{*},\tau_{i}^{*})}$ for some fixed positive integer $m$ and fixed 
pairwise distinct components $(p_{i}^{*},\eta_{i}^{*},\tau_{i}^{*})$ where $1 \leq i \leq 
m$. Under that setting, if we choose $\sigma_{1}=\sigma_{0}$, then we can check 
that $k_{*} \leq mk_{0}$ and $p_{G_{*},f}(x)=p_{G_{0},f_{0}}(x)$ almost surely. The 
explicit formulation of $G_{*}$, therefore, can be found from the combinations of $G_{0}$ 
and $(p_{i}^{*},\eta_{i}^{*},\tau_{i}^{*})$ where $1 \leq i \leq m$.

From inequality \eqref{eqn:property_modified_hellinger_distance} in Lemma 
\ref{lemma:key_inequality_misspecified_setting}, we have the following well-defined weighted version of Hellinger distance.
\begin{definition} Given $
\sigma_{1}>0$. For any two mixing measures $G_{1}, G_{2} \in \overline{\mathcal{G}}$, we define the weighted Hellinger distance $h^*(p_{G_{1},f}*K_{\sigma_{1}}, p_{G_{2},f}*K_{\sigma_{1}})$ by
\begin{multline}
\biggr(h^{*}(p_{G_{1},f}*K_{\sigma_{1}},p_{G_{2},f}*K_{\sigma_{1}})\biggr)^{2} =  \\ \nonumber 
 \dfrac{1}
{2}\int{\biggr(\sqrt{p_{G_{1},f}*K_{\sigma_{1}}(x)}-\sqrt{p_{G_{2},f}*K_{\sigma_{1}}(x)}
\biggr)^{2}} 
\times \sqrt{\dfrac{p_{G_{0},f_{0}}*K_{\sigma_{0}}(x)}{p_{G_{*},f}
*K_{\sigma_{1}}(x)}}\textrm{d}x. \nonumber
\end{multline}
\end{definition}
The notable feature of $h^{*}$ is the presence of term $\sqrt{p_{G_{0},f_{0}}*K_{\sigma_{0}}(x)/p_{G_{*},f}
*K_{\sigma_{1}}(x)}$ in its formulation, which makes it different from the traditional 
Hellinger distance. As long as $\left\{f\right\} = \left\{f_{0}\right\}$ and $\sigma_{1}=
\sigma_{0}$, we obtain $h^{*}(p_{G_{1},f}*K_{\sigma_{1}},p_{G_{2},f}
*K_{\sigma_{1}}) \equiv h(p_{G_{1},f}*K_{\sigma_{1}},p_{G_{2},f}*K_{\sigma_{1}})$ 
for any $G_{1},G_{2} \in \overline{\mathcal{G}}$, i.e., the traditional Hellinger distance is 
a special case of $h^{*}$ under the well-specified kernel setting and the choice that $
\sigma_{1}=\sigma_{0}$. The weighted Hellinger distance $h^{*}$ is particularly useful 
for studying the convergence rate of $\widehat{G}_{n}$ to $G_{*}$ for any fixed $
\sigma_{1} \geq 0$ and $\sigma_{0} >0$.

Note that, in the context of the well-specified kernel setting in Section 
\ref{Section:well_specified_kernel}, the key step that we utilized to obtain the 
convergence rate $n^{-1/2}$ of $\widehat{G}_{n}$ to $G_{0}$ is based on the lower 
bound of the Hellinger distance and the first order Wasserstein distance in inequality 
\eqref{eqn:comment_convergence_rate_wellspecify}. With the modified Hellinger distance 
$h^{*}$, it turns out that we still have the similar kind of lower bound as long as $k_{*}<\infty$.
\begin{lemma} \label{lemma:lower_bound_modified_Hellinger_Wasserstein}
Assume that $f*K_{\sigma_{1}}$ is identifiable in the first order and admits uniform Lipschitz property up to the first order. If $k_{*}<\infty$, then for any $G \in \mathcal{O}_{k_{*}}$ there holds
\begin{eqnarray}
h^{*}(p_{G,f}*K_{\sigma_{1}},p_{G_{*},f}*K_{\sigma_{1}}) \gtrsim W_{1}(G,G_{*}). \nonumber
\end{eqnarray}
\end{lemma}
Equipped with the above inequality, we have the following result regarding the
convergence rate of $\widehat{G}_{n}$ to $G_{*}$:
\begin{theorem}\label{theorem:convergence_rate_mixing_measure_misspecified_strong}
Assume $k_{*}< \infty$ for some $\sigma_{1} \geq 0$ and $\sigma_{0}>0$.
\begin{itemize}
\item[(i)] If $f*K_{\sigma_{1}}$ is identifiable, then $\widehat{m}_{n} \to k_{*}$ almost surely.
\item[(ii)] Assume further that condition (P.2) in Theorem \ref{theorem:convergence_rate_mixing_measure} holds, i.e., $\Psi(G_{0},\sigma_{0})<\infty$ and the following conditions hold:
\begin{itemize}
\item[(M.1)] The kernel $K$ is chosen such that $f * K_{\sigma_{1}}$ is identifiable
in the first order and admits the uniform Lipschitz property up to the first order.
\item[(M.2)] ${\displaystyle \sup \limits_{\theta \in \Theta}{\int {\sqrt{f*K_{\sigma_{1}}(x|\theta)}}\textrm{d}x}} \leq M_{1}(\sigma_{1})$ for some positive constant $M_{1}(\sigma_{1})$.
\item[(M.3)] $\sup \limits_{\theta \in \Theta}{\Bigr\|\dfrac{\partial{f*K_{\sigma_{1}}}}{\partial{\theta}}(x|\theta) /\bigl(f*K_{\sigma_{1}}(x|\theta)\bigr)^{3/4}\Bigr\|_{\infty}} \leq M_{2}(\sigma_{1})$ for some positive constant $M_{2}(\sigma_{1})$.
\end{itemize}
Then, we have
\begin{eqnarray}
W_{1}(\widehat{G}_{n},G_{*})=  O_{p}\biggr(\sqrt{\dfrac{M^{2}(\sigma_{1})\Psi(G_{0},\sigma_{0})}{C_{*,1}^{4}
(\sigma_{1})}}n^{-1/2}
\biggr) \nonumber
\end{eqnarray}
where $C_{*,1}(\sigma_{1}) := \inf \limits_{G \in \mathcal{O}_{k_{*}}}
{\dfrac{h^{*}(p_{G,f}*K_{\sigma_{1}},p_{G_{*},f}*K_{\sigma_{1}})}{W_{1}
(G,G_{*})}}$ and $M(\sigma_{1})$ is some positive constant.
\end{itemize}
\end{theorem}
\paragraph{Remarks:} 
\begin{itemize}
\item[(i)] As being mentioned in Lemma \ref{lemma:lower_bound_modified_Hellinger_Wasserstein}, condition (M.1) is sufficient to guarantee that $C_{*,1}(\sigma_{1})>0$.
\item[(ii)] Conditions (M.2) and (M.3) are mild. An easy example is when $f$ is Gaussian kernel and $K$ is standard Gaussian kernel.
\item[(iii)] When $f_{0}$ is indeed a finite mixture of $f$, a close investigation of the proof of Theorem \ref{theorem:convergence_rate_mixing_measure_misspecified_strong} reveals that we can relax condition (M.2) and (M.3) for the conclusion of this theorem to hold.
\item[(iv)] Under the setting that $\left\{f*K_{\sigma_{1}}\right\}=\left\{f_{0}
*K_{\sigma_{0}}\right\}$, i.e., $G_{*} \equiv G_{0}$, the result of Theorem 
\ref{theorem:convergence_rate_mixing_measure_misspecified_strong} implies that $
\widehat{G}_{n}$ converges to the true mixing measure $G_{0}$ at optimal rate $n^{-1/
2}$ even though we are under the misspecified kernel setting.
\end{itemize}

\subsubsection{Infinite $k_{*}$:} \label{Section:infinite_component}
So far, we have assumed that $k_{*}$ has finite number of support points and achieve the 
cherished convergence rate $n^{-1/2}$ of $\widehat{G}_{n}$ to unique element $G_{*} 
\in \mathcal{M}_{k_{*}}$ under certain conditions on $f, f_{0}$, and $K$. It is due to the 
fact that $\widehat{m}_{n} \to k_{*}<\infty$ almost surely, which is eventually a 
consequence of the identifibility of kernel density function $f*K_{\sigma_{1}}$. However, 
for the setting $k_{*}=\infty$, to establish the consistency of $\widehat{m}_{n}$, we 
need to resort to a slightly stronger version of identifiability, which is finitely identifiable 
condition. We adapt Definition 3 in \cite{Nguyen-13} as follows.
\begin{definition} \label{definition:finite_identifiable}
The family $\left\{f(x|\theta),\theta \in \Theta\right\}$ is finitely identifiable if for any 
$G_{1} \in \mathcal{G}$ and $G_{2} \in \overline{\mathcal{G}}$, $|p_{G_{1},f}(x)-
p_{G_{2},f}(x)|=0$ for almost all $x \in \mathcal{X}$ implies that $G_{1} \equiv 
G_{2}$.
\end{definition}
An example of finite identifiability is when $f$ is Gaussian kernel with both location and 
variance parameter. Now, a close investigation of the proof of Theorem 
\ref{theorem:convergence_rate_mixing_measure_misspecified_strong} quickly yields the 
following result.
\begin{proposition}\label{proposition:infinite_mixture_complexity_misspecified_kernel}
Given $\sigma_{1}>0$ such that $f*K_{\sigma_{1}}$ is finitely identifiable. If $k_{*}=\infty$, we achieve $\widehat{m}_{n} \to \infty$ almost surely.
\end{proposition}
Even though we achieve the consistency result of $\widehat{m}_{n}$ when $k_{*}=\infty
$, the convergence rate of $\widehat{G}_{n}$ to $G_{*}$ still remains an elusive problem. 
However, an important insight from Proposition \ref{proposition:infinite_mixture_complexity_misspecified_kernel} indicates that the 
convergence rate of $\widehat{G}_{n}$ to some element $G_{*} \in \mathcal{M}_{\infty}
$ may be much slower than $n^{-1/2}$ when $k_{*}=\infty$. It is due to the fact that 
both $\widehat{G}_{n}$ and $G_{*} \in \mathcal{M}_{\infty}$ have unbounded numbers 
of components in which the kind of bound in Lemma 
\ref{lemma:lower_bound_modified_Hellinger_Wasserstein} is no longer sufficient.
Instead, something akin to the bounds given in Theorem 2 of \cite{Nguyen-13} in the misspecified
setting is required. We leave 
the detailed analyses of $\widehat{G}_{n}$ under that setting of $k_{*}$ for the future 
work.
\subsection{Comparison to WS Algorithm} \label{Section:MS_algorithm}
In the previous sections, we have established a careful study regarding the behaviors of $
\widehat{G}_{n}$ in Algorithm 1, i.e., we achieved the consistency of the number of 
components as well as the convergence rates of parameter estimation under various 
settings of $f$ and $f_{0}$ when the bandwidths $\sigma_{1}$ and $\sigma_{0}$ are 
fixed. As we mentioned at the beginning of Section \ref{Section:minimum_Hellinger}, 
Algorithm 1 is the generalization of WS Algorithm when $\sigma_{1}=0$ and $\sigma_{0}
>0$. Therefore, the general results with estimator $\widehat{G}_{n}$ in Theorem 
\ref{theorem:convergence_rate_mixing_measure_misspecified_strong} are still applicable 
to $\overline{G}_{n}$ under that special case of $\sigma_{1}$ and $\sigma_{0}$. To 
rigorously demonstrate the flexibilities and advantages of our estimator $\widehat{G}_{n}
$ over WS estimator $\overline{G}_{n}$, we firstly discuss the behaviors of estimator $
\overline{G}_{n}$ from WS Algorithm under the well-specified kernel setting, i.e., $\left\{f
\right\}=\left\{f_{0}\right\}$, and the fixed bandwidth setting of $\sigma_{0}$. 
Remember that $f_{0}$ is assumed to be identifiable in the first order and to have uniform 
Lipschitz property up to the first order. Assume now we can find
\begin{eqnarray}
\overline{G}_{0} := \mathop {\arg \min} \limits_{G \in \overline{\mathcal{G}}}{h(p_{G,f_{0}},p_{G_{0},f_{0}}*K_{\sigma_{0}})}, \nonumber
\end{eqnarray}
i.e., $\overline{G}_{0}$ is the discrete mixing measure that minimizes the Hellinger
distance between $p_{G,f_{0}}$ and $p_{G_{0},f_{0}}*K_{\sigma_{0}}$. Note that, $
\overline{G}_{0}$ is a special case of $G_{*}$ when $\left\{f\right\} = \left\{f_{0}\right\}$ and $\sigma_{1}=0$. The form of $\overline{G}_{0}$ can be determined explicitly 
under various settings of $f_{0}$ and $K$. For instance, assume that $f_{0}$ are either 
univariate Gaussian kernel or Cauchy kernel with parameters $\theta=(\eta,\tau)$ where $
\eta$ and $\tau$ are location and variance parameter and $K$ are either standard 
univariate Gaussian kernel or Cauchy kernel respectively. Then, a simple calculation shows 
that $\overline{G}_{0}=\sum \limits_{i=1}^{k_{0}}{p_{i}^{0}\delta_{(\theta_{i}^{0},
\overline{\tau}_{i}^{0})}}$ where $\overline{\tau}_{i}^{0}=\sqrt{(\tau_{i}^{0})^{2}+
\sigma_{0}^{2}}$ for any $1 \leq i \leq k_{0}$ and $\sigma_{0}>0$.

As being argued in Section \ref{Section:adaptiveminimumdistancemisspecified}, $
\overline{G}_{0}$ may have infinite number of components in general; however, for the 
sake of simplicity, we assume that there exists $\overline{G}_{0}$ having finite number of 
components, which is also unique according to the argument in Section 
\ref{Section:adaptiveminimumdistancemisspecified}. Under the assumptions of Theorem 
\ref{theorem:convergence_rate_mixing_measure_misspecified_strong} when $
\sigma_{1}=0$, we eventually achieve that
\begin{eqnarray}
W_{1}(\overline{G}_{n},\overline{G}_{0})=  O_{p}\biggr(\sqrt{\dfrac{\overline{M}^{2}\Psi(G_{0},\sigma_{0})}{[\overline{C}]^{4}}}n^{-1/2}
\biggr) \nonumber
\end{eqnarray}
where $\overline{C} := \inf \limits_{G \in \mathcal{O}_{\overline{k}_{0}}}
{\dfrac{h^{*}(p_{G,f_{0}},p_{\overline{G}_{0},f_{0}})}{W_{1}(G,\overline{G}_{0})}}$ 
and $\overline{M}$ is some positive constant. The above result implies that the estimator 
$\overline{G}_{n}$ from WS Algorithm will not converge to the true mixing measure 
$G_{0}$ for any fixed bandwith $\sigma_{0}$. It demonstrates that Algorithm 1 is more 
appealing than WS Algorithm under the well-specified kernel setting with fixed bandwidth $
\sigma_{0}>0$. For the setting when the bandwidth $\sigma_{0}$ is allowed to vanish to 
0, our result indicates that the convergence rate of $\overline{G}_{n}$ to $G_{0}$ will 
depend not only on the vanishing rate of the term $\Psi(G_{0},\sigma_{0})$ to 0 but also 
on the convergence rate of $\overline{G}_{0}$ to $G_{0}$. Intuitively, to ensure that the 
convergence of $\overline{G}_{n}$ to $G_{0}$ is $n^{-1/2}$, we also need to achieve 
that of $\overline{G}_{0}$ to $G_{0}$ to be $n^{-1/2}$. Under the specific case that 
$f_{0}$ and $K$ are univariate Gaussian kernels, the convergence rate of $\overline{G}
_{0}$ to $G_{0}$ is $n^{-1/2}$ only when $\sigma_{0}$ goes to 0 at the same rate 
$n^{-1/2}$. However, it will lead to a strong convergence of $\Psi(G_{0},\sigma_{0})$ to $
\infty$, which makes the convergence rate of $\overline{G}_{n} \to G_{0}$ become much 
slower than $n^{-1/2}$. Therefore, it is possible that the convergence rate of WS 
estimator $\overline{G}_{n}$ to $G_{0}$ may be much slower than $n^{-1/2}$ regardless 
of the choice of bandwidth $\sigma_{0}$. As a consequence, our estimator in Algorithm 1 
may also be more efficient than WS estimator under that regime of vanishing bandwidth $\sigma_{0}$.

Under the misspecified kernel setting, we would like to emphasize that our estimator $
\widehat{G}_{n}$ is also more flexible than WS estimator $\overline{G}_{n}$ as we 
provide more freedom with the choice of bandwidth $\sigma_{1}$ in Algorithm 1, instead 
of specifically fixing $\sigma_{1}=0$ as that in WS Algorithm. If there exists $\sigma_{1}
>0$ such that $\left\{f*K_{\sigma_{1}}\right\}=\left\{f_{0}*K_{\sigma_{0}}\right\}$, 
then our estimator $\widehat{G}_{n}$ will converge to $G_{0}$ while WS estimator $
\overline{G}_{n}$ will converge to $\overline{G}_{0}$ that can be very different from 
$G_{0}$. Therefore, the performance of our estimator is also better than that of WS 
estimator under that specific misspecified kernel setting.

\section{Different approach with minimum Hellinger distance estimator}
\label{Section:Another_approach}
Thus far, we have developed a robust estimator of mixing measure $G_{0}$ 
based on the idea of minimum Hellinger distance estimator and model selection criteria. 
That estimator is shown to attain various desirable properties, including the consistency of 
number of components $\widehat{m}_{n}$ and the optimal convergence rates of $
\widehat{G}_{n}$. In this section, we take a rather different approach of constructing such 
robust estimator. In fact, we have the following algorithm:
\paragraph{Algorithm 2:}
 \begin{itemize}
\item[•] Step 1: Determine $\widehat{G}_{n,m}=\mathop {\arg \min} \limits_{G \in \mathcal{O}_{m}}{h(p_{G,f}*K_{\sigma_{1}},P_{n}*K_{\sigma_{0}})}$ for any $n,m \geq 1$.
\item[•] Step 2: Choose
\begin{eqnarray}
\widetilde{m}_{n}=\mathop {\inf }{\biggr\{m \geq 1: h(p_{\widehat{G}_{n,m},f}*K_{\sigma_{1}},P_{n}*K_{\sigma_{0}})< \epsilon\biggr\}}, \nonumber
\end{eqnarray}
where $\epsilon>0$ is any given positive constant and $\sigma_{1}, \sigma_{0}$ are two chosen bandwidths.
\item[•] Step 3: Let $\widetilde{G}_{n}=\widehat{G}_{n,\widetilde{m}_{n}}
$ for each $n$.
\end{itemize}
 Unlike Step 2 in Algorithm 1 where we consider the difference between $h(p_{\widehat{G}_{n,m},f}*K_{\sigma_{1}},P_{n}*K_{\sigma_{0}})$ and $h(p_{\widehat{G}
_{n,m+1},f}*K_{\sigma_{1}},P_{n}*K_{\sigma_{0}})$, here we consider solely the 
evaluation of $h(p_{\widehat{G}_{n,m},f}*K_{\sigma_{1}},P_{n}*K_{\sigma_{0}})$ in 
Algorithm 2. The above robust estimator of mixing measure is based on the idea of 
minimum Hellinger distance estimator and superefficiency phenomenon. 
A related approach considered in the well-specified setting was taken by \cite{Jonas-2016}.
Their construction was based on minimizing supremum norm based distance,
without using the convolution kernels 
$K_{\sigma_{1}}$ and $K_{\sigma_{0}}$; moreover, the threshold $\epsilon$ was set to vanish
as $n \to \infty$. Although of theoretical interest, their estimator appears difficult to 
compute efficiently and may be unstable due to the use of the supremum norm.

Our focus with Algorithm 2 in this section will be mainly about its attractive theoretical 
performance. As we observe from Algorithm 2, the values of $f, f_{0}$, $K$, and $G_{0}$ 
along with the bandwidths $\sigma_{1}, \sigma_{0}$ play crucial roles in determining the 
convergence rate of $\widetilde{G}_{n}$ to $G_{0}$ for any given $\epsilon>0$. Similar 
to the argument of Theorem \ref{theorem:convergence_rate_mixing_measure} and 
Theorem \ref{theorem:convergence_rate_mixing_measure_misspecified_strong}, one of 
the key ingredients to fulfill that goal is to find the conditions of these factors such that we 
obtain the consistency of $\widetilde{m}_{n}$. The following
theorem yields the sufficient and necessary conditions to address the consistency question.
\begin{theorem} \label{theorem:sufficent_necessary_condition}
Given $\sigma_{1} \geq 0$ and $\sigma_{0}>0$. Then, we have
\begin{itemize}
\item [(i)] Under the well-specified kernel setting and the case that $\sigma_{1}=\sigma_{0}$, $\widetilde{m}_{n} \to k_{0}$ almost surely if and only if
\begin{eqnarray}
\epsilon< h(p_{G_{0,k_{0}-1},f_{0}}*K_{\sigma_{0}},p_{G_{0},f_{0}}*K_{\sigma_{0}}) \label{eqn:another_wellspecified_kernel}
\end{eqnarray}
where $G_{0,k_{0}-1}=\mathop {\arg \min}\limits_{G \in \mathcal{E}_{k_{0}-1}}{h(p_{G,f_{0}}*K_{\sigma_{0}},p_{G_{0},f_{0}}*K_{\sigma_{0}})}$.
\item [(ii)] Under the misspecified kernel setting, if $k_{*}<\infty$, then $\widetilde{m}_{n} \to k_{*}$ almost surely if and only if
\begin{eqnarray}
h(p_{G_{*},f}*K_{\sigma_{1}},p_{G_{0},f_{0}}*K_{\sigma_{0}}) \leq \epsilon< h(p_{G_{*,k_{*}-1},f}*K_{\sigma_{1}},p_{G_{0},f_{0}}*K_{\sigma_{0}}) \label{eqn:another_misspeficied_kernel}
\end{eqnarray}
where $G_{*,k_{*}-1}=\mathop {\arg \min}\limits_{G \in \mathcal{E}_{k_{*}-1}}{h(p_{G,f}*K_{\sigma_{1}},p_{G_{0},f_{0}}*K_{\sigma_{0}})}$ and $G_{*} \in \mathcal{M}$ with exactly $k_{*}$ components.
\end{itemize}
\end{theorem}
If we allow $\epsilon\to 0$ in Algorithm 2, we achieve the inconsistency of $\widetilde{m}
_{n}$ under the misspecified kernel setting when $k_{*}<\infty$. Hence, the choice of 
threshold $\epsilon$ from \cite{Jonas-2016} is not optimal regarding the misspecified 
kernel setting. Unfortunately, conditions \eqref{eqn:another_wellspecified_kernel} and
\eqref{eqn:another_misspeficied_kernel} are rather cryptic as in general, it is hard to
determine the exact formulation of $G_{0,k_{0}-1}$, $G_{*,k_{*}-1}$, and $G_{*}$.
It would be of interest to find relatively simple sufficient conditions on 
$f, f_{0}$, $K$, $G_{0}$, $\sigma_{1}$, and $\sigma_{0}$ according to which
either  \eqref{eqn:another_wellspecified_kernel} or
\eqref{eqn:another_misspeficied_kernel}  holds.
Unfortunately, this seems to be a difficult task in the mis-specified setting. Under
the well-specified kernel setting, a sufficient condition for \eqref{eqn:another_wellspecified_kernel} can be reformulated as 
a condition regarding the lower bound on the
smallest mass of $G_{0}$, the minimal distance between its point masses, and the lower bound between the Hellinger distance and Wasserstein distance:
\begin{proposition} \label{proposition:sufficient_condition_wellspecified_setting} {\bf(Well-speficied kernel setting)}
For any given $\sigma_{0}>0$, assume that $f_{0}*K_{\sigma_{0}}$ admits uniform 
Lipschitz property up to the first oder and is identifiable. If we have
\begin{eqnarray}
\mathop {\inf }\limits_{G \in \mathcal{E}_{k_{0}-1}}{\dfrac{h(p_{G,f_{0}}*K_{\sigma_{0}},p_{G_{0},f_{0}}*K_{\sigma_{0}})}{W_{1}(G,G_{0})}} \mathop {\min }\limits_{1 \leq i \leq k_{0}}{p_{i}^{0}} \mathop {\min }\limits_{1 \leq i \neq j \leq k_{0}} \|\theta_{i}^{0}-\theta_{j}^{0}\| \geq \epsilon, \nonumber 
\end{eqnarray}
then we obtain the inequality in \eqref{eqn:another_wellspecified_kernel}.
\end{proposition}
\section{Non-standard settings}
\label{Section:minimum_Hellinger_singular}
In this section, we briefly demonstrate that our robust estimator in Algorithm 1 (similarly 
Algorithm 2) also achieves desirable convergence rates under non-standard 
settings. In particular, in the first setting, either $f_{0}$ or $f$ may 
not be identifiable in the first order. In the second setting, the true mixing measure $G_{0}
$ changes with the sample size $n$ and converges to some discrete distribution $
\widetilde{G}_{0}$ under $W_{1}$ distance.

\subsection{Singular Fisher information matrix}
The results in the previous sections are under the assumption that both the true kernel 
$f_{0}$ and the chosen kernel $f$ are identifiable in the first order. This is equivalent to 
the non-singularity of the Fisher information matrix of $p_{G_{0},f_{0}}$  and 
$p_{G_{*},f}$ when $G_{*} \in \mathcal{M}$, i.e., both $I(G_{0},f_{0})$ and $I(G_{*},f)$ 
are non-singular. Therefore, we achieve the cherished convergence rate $n^{-1/2}$ of $
\widehat{G}_{n}$. Unfortunately, these assumptions do not always hold. For instance, 
both the Gamma and skewnormal kernel are not identifiable in the first order \citep{Ho-Nguyen-Ann-16, Ho-Nguyen-Ann-17}. According to \cite{Azzalini-1996, Wiper-01}, these 
kernels are particularly useful for modelling various kinds of data: the Gamma kernel is used 
for modeling non-negative valued data and the skewnormal kernel is used for modeling 
asymmetric data. Therefore, it is worth considering the performance of our estimator in 
Algorithm 1 under the nonidentifiability in the first order of both kernels $f_{0}$ and $f$. 
Throughout this section, for the simplicity of the argument we consider only the well-specified kernel setting and the setting that $f_{0}$ may not be identifiable in the first 
order. Additionally, we also choose $\sigma_{1}=\sigma_{0}>0$. The argument for the 
misspecified kernel setting, the non-identifiability in the first order setting of either $f$ or $f_{0}$, 
and the general choices of $\sigma_{1}, \sigma_{0}$ can be argued in the similar fashion.

The non-identifiability in the first order of $f_{0}$ implies that the Fisher information 
matrix $I(G_{0},f_{0})$ of $p_{G_{0},f_{0}}$ is singular at particular values of $G_{0}$. 
Therefore, the convergence rate of $\widehat{G}_{n}$ to $G_{0}$ will be much slower 
than the standard convergence rate $n^{-1/2}$. In order to precisely determine the 
convergence rates of parameter estimation under the singular Fisher information matrix 
setting, \cite{Ho-Nguyen-Ann-17} introduced a notion of \textit{singularity level} of the 
mixing measure $G_0$ relative to the mixture model class; alternatively we say the
singularity level of Fisher information matrix $I(G_{0},f_{0})$ 
(cf. Definition 3.1 and Definition 3.3). Here, we 
briefly summarize the high level idea of singularity level according to the notations in our 
paper for the convenience of readers. In particular, we say that $I(G_{0},f_{0})$ admits $r
$-th level of singularity  relative to the ambient space $\mathcal{O}_{k_{0}}$ for $0 
\leq r < \infty$ if we have:
\begin{eqnarray}
&&\inf \limits_{G \in \mathcal{O}_{k_{0}}}{V(p_{G,f_{0}},p_{G_{0},f_{0}})/W_{s}^{s}(G,G_{0})}=0,\quad s=1,\dots,r. \nonumber
\\
&&V(p_{G,f_{0}},p_{G_{0},f_{0}}) \gtrsim W_{r+1}^{r+1}(G,G_{0}),\quad\text{for all } G \in \mathcal{O}_{k_{0}}.\label{eqn:inequality_singular_Fisher}
\end{eqnarray}
The infinite singularity level of the Fisher information matrix $I(G_{0},f_{0})$ implies that inequality \eqref{eqn:inequality_singular_Fisher} will not hold for any $r \geq 0$.
(Actually, these are consequences, not the original definition of singularity level
in \cite{Ho-Nguyen-Ann-17}, but this is sufficient for our purpose.)

When $f_{0}$ is identifiable in the first order, $I(G_{0},f_{0})$ will only have  
singularity level zero for all $G_{0} \in \mathcal{E}_{k_{0}}$, i.e., $r=0$ in 
\eqref{eqn:inequality_singular_Fisher}. However, the singularity levels of the Fisher 
information matrix $I(G_{0},f_{0})$ are generally not uniform over $G_{0}$ when 
$I(G_{0},f_{0})$ is singular. For example, when $f_{0}$ is skewnormal kernel, 
$I(G_{0},f_{0})$ will admit any level of singularity, ranging from 0 to $\infty$ 
depending on the interaction of atoms and masses of $G_{0}$ \citep{Ho-Nguyen-Ann-17}. 
The notion of singularity level allows us to establish precisely the convergence rate of any 
estimator of $G_{0}$. In fact, if $r<\infty$ is the singularity level of 
$I(G_{0},f_{0})$, for any estimation method that yields the
convergence rate $n^{-1/2}$ for $p_{G_{0},f_{0}}$ under the Hellinger distance, the induced
best possible rate of convergence for the mixing measure $G_{0}$ is $n^{-1/2(r+1)}$ under $W_{r+1}$
distance. If $r=\infty$ is the singularity level of $I(G_{0},f_{0})$, all the estimation 
methods will yield a non-polynomial convergence rate of $G_{0}$, one that is slower than 
$n^{-1/2s}$ for any $s \geq 1$.

Now, by using the same line of argument as that of Theorem 
\ref{theorem:convergence_rate_mixing_measure} we have the following result regarding 
the convergence rate of $\widehat{G}_{n}$ to $G_{0}$ when the Fisher information 
matrix $I(G_{0},f_{0})$ has $r$-th singularity level for some $r<\infty$.
\begin{proposition} \label{proposition:singularity_Fisher_information}
Given the well-specified kernel setting, i.e., $\left\{f\right\}=\left\{f_{0}\right\}$, and the 
choice that $\sigma_{1}=\sigma_{0}>0$. Assume that the Fisher information 
$I(G_{0},f_{0})$ has $r$-th singularity level where $r<\infty$ and condition (P.2) in 
Theorem \ref{theorem:convergence_rate_mixing_measure} holds, i.e., $\Psi(G_{0},
\sigma_{0})<\infty$. Furthermore, the kernel $K$ is chosen such that the Fisher 
information matrix $I(G_{0},f_{0}*K_{\sigma_{0}})$ has $r$-th singularity level and 
$f_{0}*K_{\sigma_{0}}$ admits a uniform Lipschitz property up to the $r$-th order. Then, 
we have
\begin{eqnarray}
W_{r+1}(\widehat{G}_{n},G_{0})=O_{p}\biggr(\sqrt{\dfrac{\Psi(G_{0},\sigma_{0})}{C_{r}^{2}(\sigma_{0})}}n^{-1/2(r+1)}\biggr) \nonumber
\end{eqnarray}
where $C_{r}(\sigma_{0})=\inf \limits_{G \in \mathcal{O}_{k_{0}}}
{\dfrac{h(p_{G,f_{0}}*K_{\sigma_{0}},p_{G_{0},f_{0}}*K_{\sigma_{0}})} {W_{r+1}^{r+1}(G,G_{0})}}$.
\end{proposition}
\paragraph{Remarks:} 
\begin{itemize}
\item[(i)] A mild condition such that $I(G_{0},f_{0})$ and $I(G_{0},f_{0}*K_{\sigma})$ 
have the same singularity level is $\widehat{K}(t) \neq 0$ for all $t \in \mathbb{R}^{d}$ 
where $\widehat{K}(t)$ denotes the Fourier transformation of $K$ (cf. Lemma 
\ref{lemma:singularity_level} in Appendix B).
\item[(ii)] Examples of $f_{0}$ that are not identifiable in the first order and satisfy $
\Psi(G_{0},\sigma)<\infty$ are skewnormal and exponential kernel while $K$ is chosen to 
be Gaussian or exponential kernel respectively.
\item[(iii)] The result of Proposition \ref{proposition:singularity_Fisher_information} 
implies that under suitable choices of kernel $K$, our estimator in Algorithm 1 still achieves 
the best possible convergence rate for estimating $G_{0}$ even when the Fisher 
information matrix $I(G_{0},f_{0})$ is singular.
\end{itemize}

\subsection{Varying true parameters}

So far, our analysis has relied upon the assumption that $G_{0}$ is fixed as $n \to \infty$. 
However, there are situations such as in an asymptotic minimax analysis 
the true mixing measure $G_{0}$ is allowed to vary with $n$ and converge to some distribution $\widetilde{G}_{0}$ 
under $W_{1}$ distance as $n \to \infty$. In this section, we will demonstrate that our 
estimator in Algorithm 1 still achieves the optimal convergence rate 
under that setting of $G_{0}$.

Denote the number of components of $\widetilde{G}_{0}$  by $\widetilde{k}_{0}$. For 
the clarity of our argument we only work with the well-specified kernel setting and with the 
setting that $f_{0}$ is identifiable in the first order. As we have seen from the analysis of 
Section \ref{Section:well_specified_kernel}, when $G_{0}$ does not change with $n$, the 
key steps used to establish the standard convergence rate $n^{-1/2}$ of $\widehat{G}
_{n}$ to $G_{0}$ are through the combination of the convergence of $\widehat{m}_{n}$ 
to $k_{0}$ almost surely and, under the first order identifiability of $f_{0}
*K_{\sigma_{0}}$, the lower bound
\begin{eqnarray}
h(p_{G,f_{0}}*K_{\sigma_{0}},p_{G_{0},f_{0}}*K_{\sigma_{0}}) \gtrsim W_{1}(G,G_{0}) \label{eqn:inequality_fixed_G0}
\end{eqnarray}
holds for any $G \in \mathcal{O}_{k_{0}}$. Unfortunately, these two results no longer hold as 
$G_{0}$ varies with $n$. The varying $G_0$ is now  
denoted by $G_{0}^{n}$, the true mixing distribution when the sample size is $n$. 
Let $k_{0}^{n}$ be the number of components of $G_{0}^{n}$. Assume moreover that $
\mathop{\limsup }\limits_{n \to \infty}{k_{0}^{n}}=k<\infty$. We start with the following 
result regarding the convergence rate of $\widehat{m}_{n}$ under that setting of $G_{0}
^{n}$:

\begin{proposition} \label{proposition:varied_true_parameters}
Given $\sigma_{0}>0$, $\widehat{m}_n$ obtained by Algorithm 1. 
If $f_{0}*K_{\sigma_{0}}$ is identifiable, then $|\widehat{m}_{n}-k_{0}^{n}| \to 0$ almost surely as $n \to \infty$.
\end{proposition}
According to the above proposition, $\widehat{m}_{n}$ will not converge to $
\widetilde{k}_{0}$ almost surely when $k>\widetilde{k}_{0}$. Additionally, from that 
proposition, inequality \eqref{eqn:inequality_fixed_G0} no longer holds since both the 
number of components of $\widehat{G}_{n}$ and $G_{0}^{n}$ vary. To account for that 
problem, we need to impose a much stronger condition on the identifiability of $f_{0}
*K_{\sigma_{0}}$.

Throughout the rest of this section, we assume that $d=d_{1}=1$, i.e., we specifically 
work with the univariate setting of $f_{0}$, and $k>\widetilde{k_{0}}$. 
Using a bound of \cite{Jonas-2016}, we obtain the following:
\begin{proposition} \label{proposition:lower_bound_varying_G0}
Given $\sigma_{0}>0$. Let $K$ be chosen such that $f_{0}*K_{\sigma_{0}}$ is 
identifiable up to the $(2k-2\widetilde{k}_{0})$-order and admits a uniform Lipschitz 
condition up to $(2k-2\widetilde{k}_{0})$-order. Then, there exist $\epsilon_{0}>0$ and 
$N(\epsilon_{0}) \in \mathbb{N}$ such that
\begin{eqnarray}
h(p_{G,f_{0}}*K_{\sigma_{0}},p_{G_{0}^{n},f_{0}}*K_{\sigma_{0}}) \geq C_{v}(\sigma)W_{1}^{2k-2\widetilde{k}_{0}+1}(G,G_{0}^{n}) \label{eqn:inequality_lower_bound_varying_G0}
\end{eqnarray}
for any $n \geq N(\epsilon_{0})$ and for any $G \in \mathcal{O}_{k_{0}^{n}}$ such that 
$W_{1}(G,\widetilde{G}_{0}) \leq \epsilon_{0}$. Here, $C_{v}(\sigma)$ is some positive 
constant depending only on $\widetilde{G}_{0}$ and $\sigma$.
\end{proposition}
Similar to the argument of Lemma \ref{lemma:first_order_convolution}, a simple example 
of $K$ and $f_{0}$ for the assumptions of Proposition 
\ref{proposition:lower_bound_varying_G0} to hold is $\widehat{K}(t) \neq 0$ for all $t 
\in \mathbb{R}^{d}$ and $f_{0}$ is identifiable up to the $(2k-2\widetilde{k}
_{0})$-order. 
Now, a combination of Proposition 
\ref{proposition:varied_true_parameters} and Proposition 
\ref{proposition:lower_bound_varying_G0} yields the following result regarding the 
convergence rate of $\widehat{G}_{n}$ to $G_{0}^{n}$.
\begin{corollary} \label{corollary:convergence_rate_varying_G0}
Given the assumptions in Proposition \ref{proposition:lower_bound_varying_G0}. Assume that $\Psi(G_{0}^{n},\sigma_{0})<\infty$ for all $n \geq 1$. Then, we have
\begin{eqnarray}
W_{1}(\widehat{G}_{n},G_{0}^{n})=O_{p}\biggr(\sqrt{\dfrac{\Psi(G_{0}^{n},\sigma_{0})}{C_{v}^{2}(\sigma_{0})
}}n^{-1/(4k-4\widetilde{k}_{0}+2)}\biggr) \nonumber
\end{eqnarray}
where $C_{v}(\sigma_{0})$ is the constant in inequality \eqref{eqn:inequality_lower_bound_varying_G0}.
\end{corollary}
\paragraph{Remark:}
\begin{itemize}
\item[(i)] If $f_{0}$ and $K$ are univariate Gaussian kernels or Cauchy kernel respectively, 
then $\Psi(G_{0}^{n},\sigma_{0}) \to \Psi(\widetilde{G}_{0},\sigma_{0})$ as $n \to \infty$.
\item[(ii)] If $W_{1}(G_{0}^{n},\widetilde{G}_{0}) = O(n^{-1/(4k-4\overline{k}_{0}+2)+
\kappa})$ for some $\kappa>0$, then the convergence rate $n^{-1/(4k-4\overline{k}
_{0}+2)}$ of $\widehat{G}_{n}$ to $G_{0}^{n}$ is sharp in the sense of minimax (cf. Theorem 3.2 in 
\citep{Jonas-2016}). Therefore, our estimator in Algorithm 1 also achieves the minimax 
rate of convergence for estimating $G_{0}^{n}$. However, our 
estimator from Algorithm 1 may be more appealing than that from \cite{Jonas-2016}
for computational reasons.
We will illustrate the result of Corollary 
\ref{corollary:convergence_rate_varying_G0} via careful simulation studies in Section 
\ref{Section:simulation}.
\end{itemize}

\section{Empirical studies}
\label{Section:simulation}
We present in this section numerous numerical studies to validate our theoretical results in 
the previous sections. To find the mixing measure $\widehat{G}_{n,m}=\mathop {\arg 
\min} \limits_{G \in \mathcal{O}_{m}}{h(p_{G,f}*K_{\sigma_{1}},P_{n}
*K_{\sigma_{0}})}$, we utilize the HMIX algorithm developed in Section 4.1 of 
\citep{Cutler-1996}. This algorithm is essentially similar to the EM algorithm and ultimately 
gives us local solutions to the previous minimization problem.
\subsection{Synthetic Data}
We start with testing Algorithm 1 using synthetic data. The discussion is divided into separate enquiries of the well- and mis-specified kernel setups.

\noindent\textbf{Well-specified kernel setting} Under this setting, we assess the performance of our estimator in Algorithm 1 under two cases of $G_{0}$:
\paragraph{Case 1:} $G_{0}$ is fixed with the sample size. Under this case, we consider three choices of $f_{0}$: Gaussian and Cauchy kernel for satisfying first order identifiability condition, and skewnormal kernel 
for failing the first order identifiability condition.
\begin{itemize}
\item Case 1.1 - Gaussian family:
\begin{eqnarray*}
  f_{0}(x|\eta,\tau) &=& \dfrac{1}{\sqrt{2\pi}\tau}\exp\left(-\dfrac{(x-\eta)^{2}}{2\tau^{2}}\right) \\
  G_{0} &=& \dfrac{1}{2}\delta_{(0,\sqrt{10})}+\dfrac{1}{4}\delta_{(-0.3,\sqrt{0.05})} +\dfrac{1}{4}\delta_{(0.3,\sqrt{0.05})}.
\end{eqnarray*}
\item Case 1.2 - Cauchy family:
\begin{eqnarray*}
  f_{0}(x|\eta,\tau) &=& \dfrac{1}{\pi\tau(1+(x-\eta)^{2}/\tau^{2})} \\
  G_{0} &=& \dfrac{1}{2}\delta_{(0,\sqrt{10})}+\dfrac{1}{4}\delta_{(-0.3,\sqrt{0.05})}+\dfrac{1}{4}\delta_{(0.3,\sqrt{0.05})}.
\end{eqnarray*}
\item Case 1.3 - Skewnormal family:
\begin{eqnarray*}
  f_{0}(x|\eta,\tau,m) &=& \dfrac{2}{\sqrt{2\pi}\tau}\exp\left(-\dfrac{(x-\eta)^{2}} {2\tau^{2}}\right) \Phi\left(m(x-\eta)/\tau\right) \\
  G_{0} &=& \dfrac{1}{2}\delta_{(0,\sqrt{10},0)}+\dfrac{1}{4}\delta_{(-0.3,\sqrt{0.05},0)} +\dfrac{1}{4}\delta_{(0.3,\sqrt{0.05},0)}.
\end{eqnarray*} where $\Phi$ is the cumulative function of standard normal distribution.
\end{itemize}
For the Gaussian case and skewnormal case of $f_{0}$, we choose $K$ to be the 
standard Gaussian kernel while $K$ is chosen to be the standard Cauchy kernel for the 
Cauchy case of $f_{0}$. Note that, regarding skewnormal case it was shown that the 
Fisher information matrix $I(G_{0},f_{0})$ has second level singularity (cf. Theorem 5.3 in 
\citep{Ho-Nguyen-Ann-17}); therefore, from the result of Proposition 
\ref{proposition:singularity_Fisher_information}, the convergence rate of $\widehat{G}
_{n}$ to $G_{0}$ will be at most $n^{-1/6}$. Now for the bandwidths, we choose $
\sigma_{1} = \sigma_{0} = 1$. The sample sizes will be $n = 200*i$ where $1 \leq i \leq 
20$. The tuning parameter $C_{n}$ is chosen according to BIC criterion. More specifically, 
$C_{n} = \sqrt{3\log n}/\sqrt{2}$ for Gaussian and Cauchy family while $C_{n}=
\sqrt{2\log n}$ for skewnormal family. For each sample size $n$, we perform Algorithm 1 
exactly 100 times and then choose $\widehat{m}_{n}$ to be the estimated number of 
components with the highest probability of appearing. Afterwards, we take the average 
among all the replications with the estimated number of components $\widehat{m}_{n}$ 
to obtain $W_{1}(\widehat{G}_{n},G_{0})$. See Figure \ref{figure_well_specify_kernel} 
where the Wasserstein distances $W_{1}(\widehat{G}_{n},G_{0})$ and the percentage of 
time $\widehat{m}_{n}=3$ are plotted against increasing sample size $n$ along with the 
error bars. The simulation results regarding Gaussian and Cauchy family match well with the 
standard $n^{-1/2}$ convergence rate from Theorem 
\ref{theorem:convergence_rate_mixing_measure} while the simulation results regarding 
skewnormal family also fit with the best possible convergence rate $n^{-1/6}$ as we argued earlier.
\begin{figure*}[h]
\centering
\captionsetup{justification=centering}
\begin{minipage}[b]{.20\textwidth}
\includegraphics[width=35mm,height=40mm]{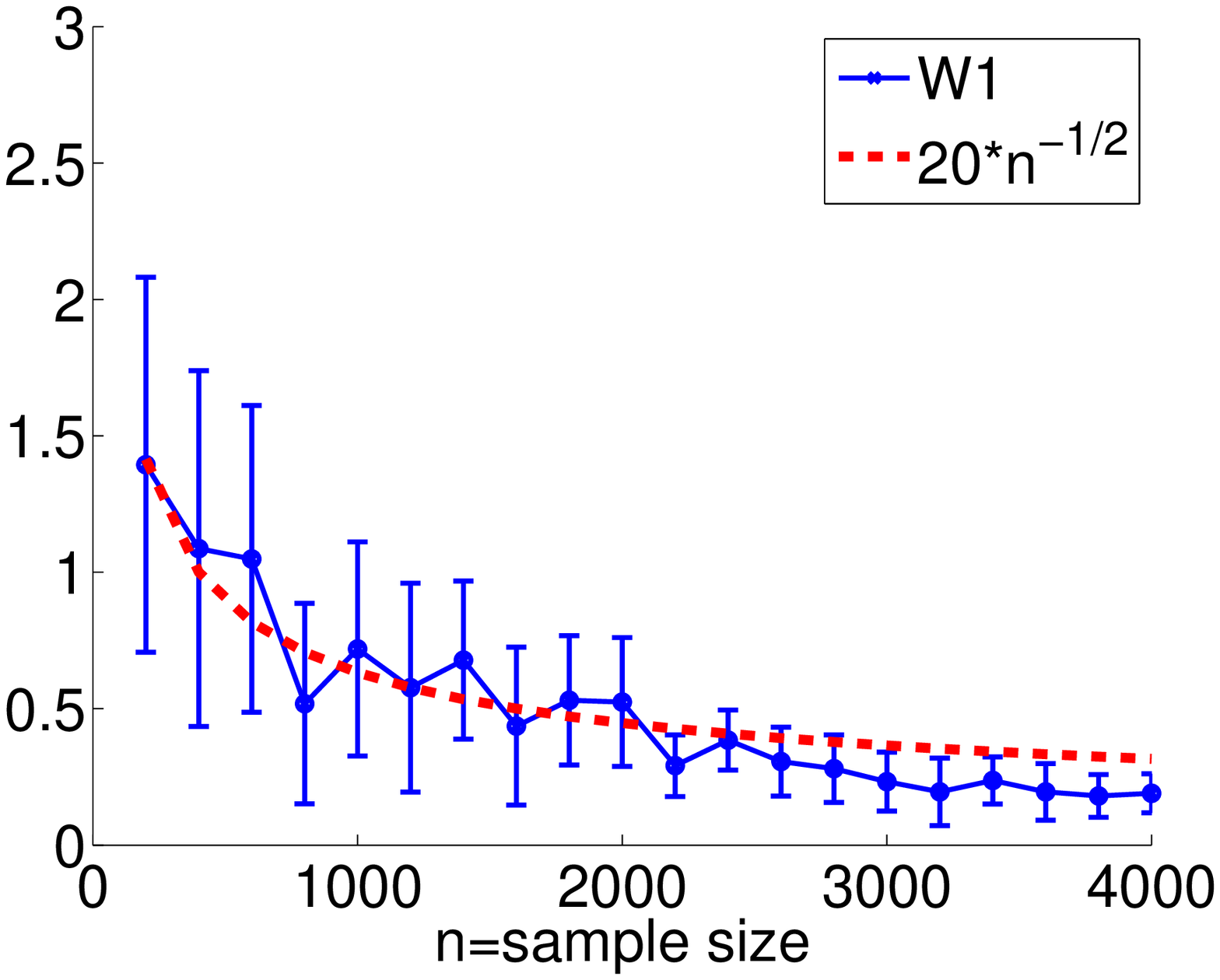}
\end{minipage}
\quad \quad
\begin{minipage}[b]{.20\textwidth}
\includegraphics[width=35mm,height=40mm]{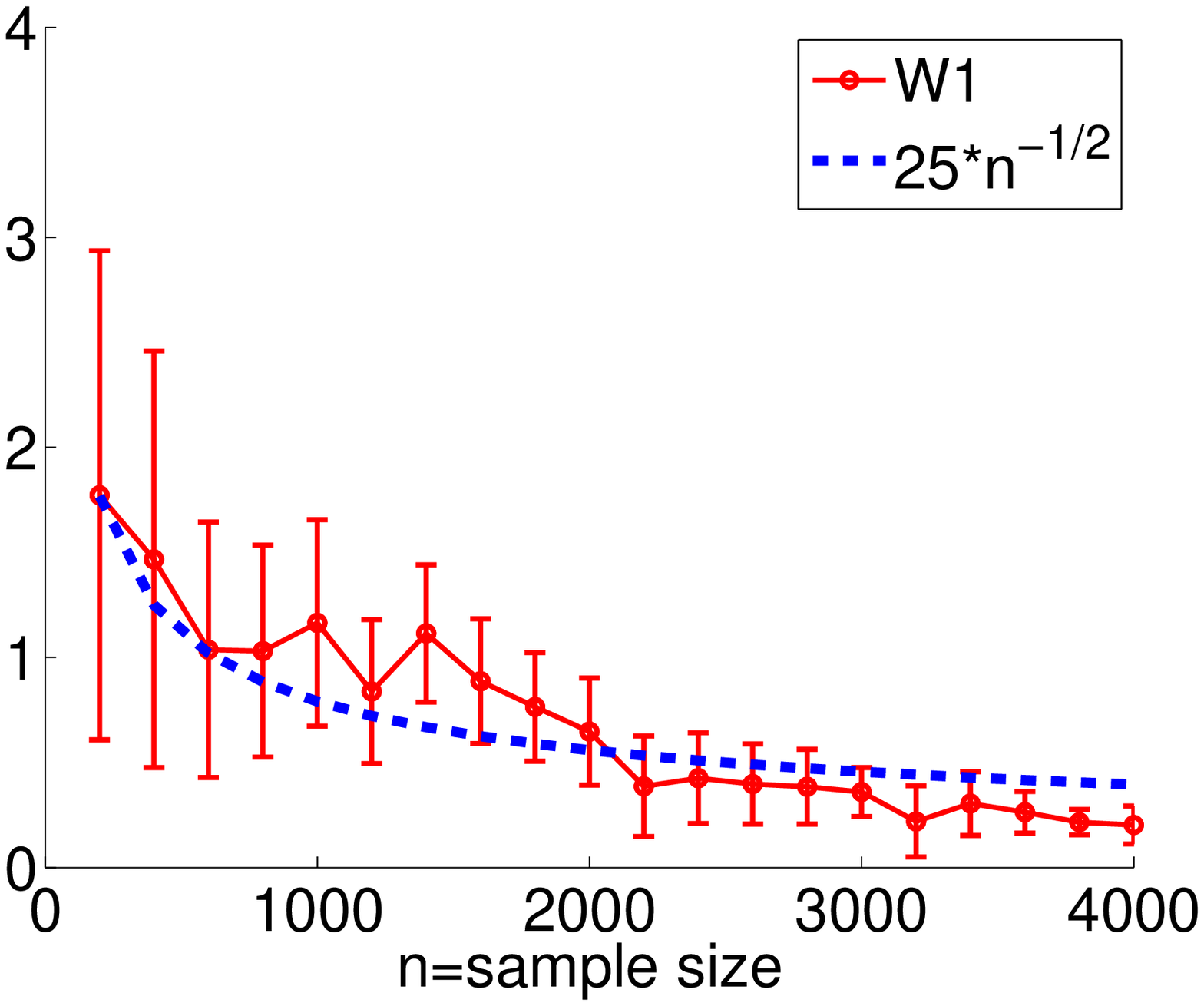}
\end{minipage}
\quad \quad
\begin{minipage}[b]{.20\textwidth}
\includegraphics[width=35mm,height=40mm]{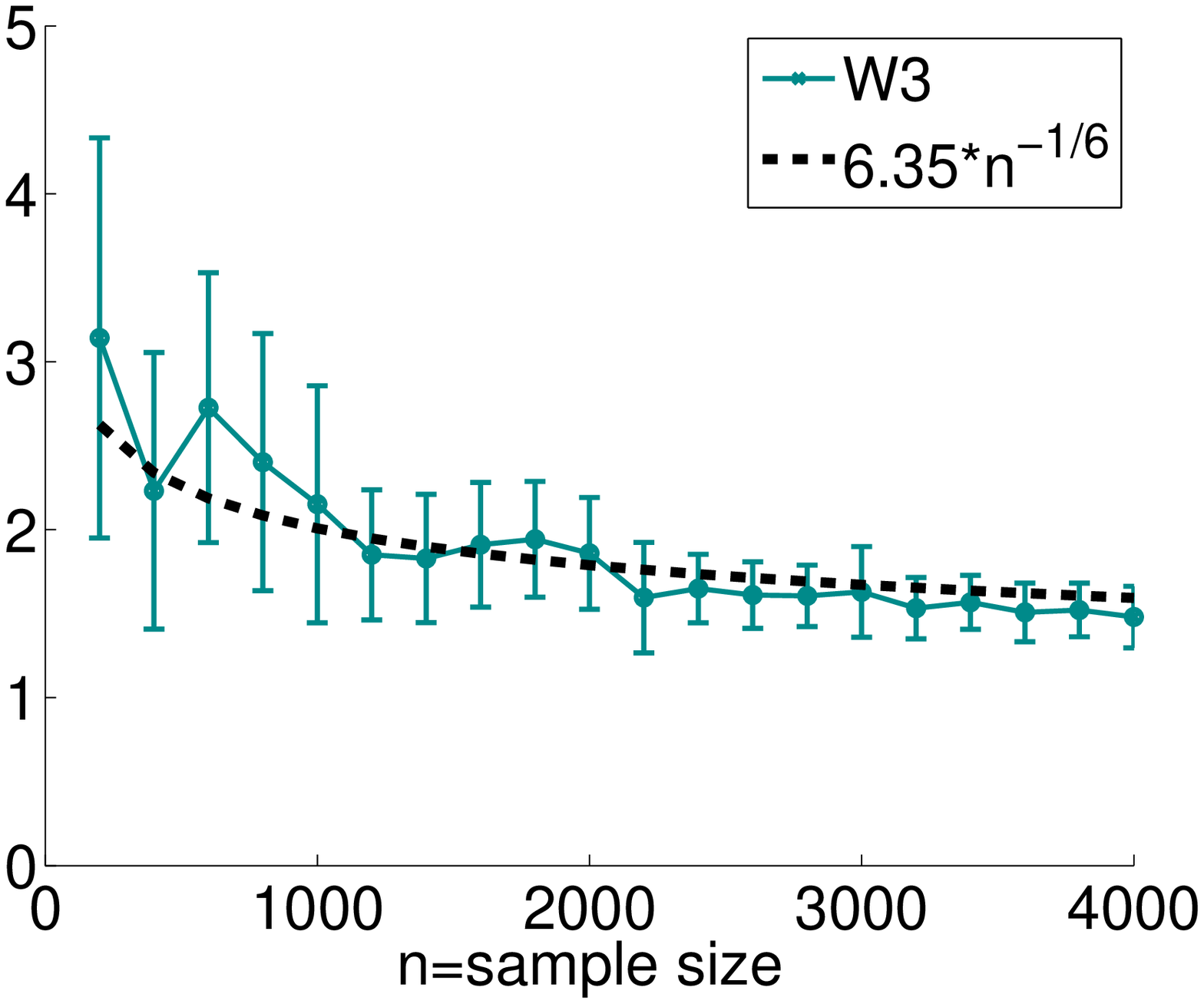}
\end{minipage}
\quad \quad
\begin{minipage}[b]{.20\textwidth}
\includegraphics[width=35mm,height=40mm]{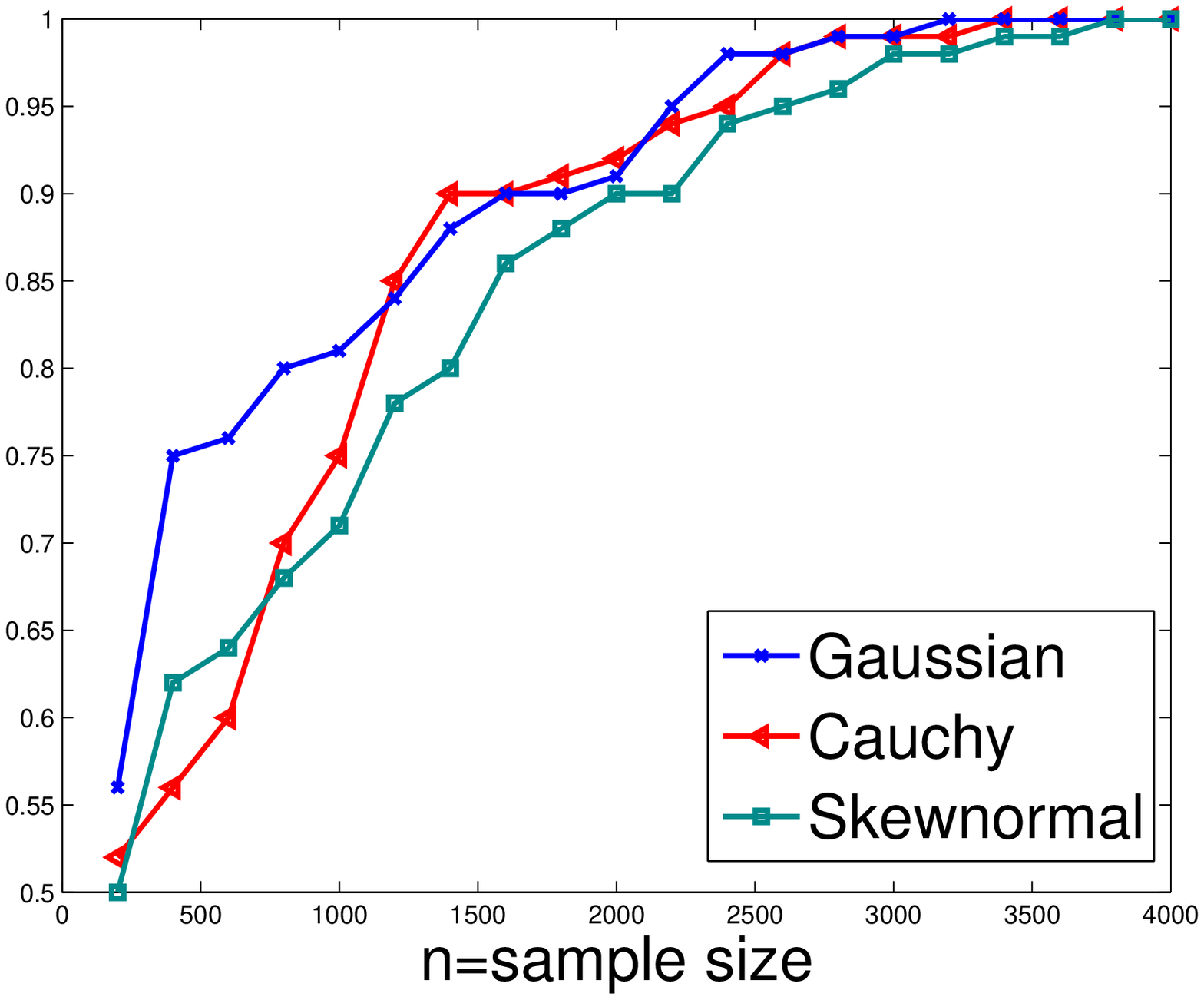}
\end{minipage}
\caption{\footnotesize{
Performance of $\widehat{G}_{n}$ in Algorithm 1 under the well-specified kernel setting and fixed $G_{0}$. Left to right:
(1) $W_{1}(\widehat{G}_{n},G_{0})$ under Gaussian case.
(2) $W_{1}(\widehat{G}_{n},G_{0})$ under Cauchy case. 
(3) $W_{1}(\widehat{G}_{n},G_{0})$ under Skewnormal case.
(4) Percentage of time $\widehat{m}_{n}=3$ obtained from 100 runs.
}}
\label{figure_well_specify_kernel}
\end{figure*}
\paragraph{Case 2:} $G_{0}$ is varied with the sample size. Under this case, we consider two choices of $f_{0}$: Gaussian and Cauchy kernel with only location parameter.
\begin{itemize}
\item Case 2.1 - Gaussian family:
\begin{eqnarray*}
  f_{0}(x|\eta) &=& \dfrac{1}{\sqrt{2\pi}}\exp\left(-\dfrac{(x-\eta)^{2}}{2}\right) \\
  G_{0} &=& \dfrac{1}{4}\delta_{1-1/n}+\dfrac{1}{4}\delta_{1+1/n}+\dfrac{1}{2}\delta_{2},
\end{eqnarray*}  where $n$ is the sample size.
\item Case 2.2 - Cauchy family:
\begin{eqnarray*}
  f_{0}(x|\eta) &=& \dfrac{1}{\pi(1+(x-\eta)^{2})} \\
  G_{0} &=& \dfrac{1}{4}\delta_{1-1/\sqrt{n}}+\dfrac{1}{4}\delta_{1+1/\sqrt{n}} +\dfrac{1}{2}\delta_{1+2/\sqrt{n}}.
\end{eqnarray*}
\end{itemize}
With these settings, we can verify that $\widetilde{G}_{0}=\dfrac{1}{2}\delta_{1}+
\dfrac{1}{2}\delta_{2}$ for the Gaussian case and $\widetilde{G}_{0}=\delta_{1}$ for 
the Cauchy case. Additionally, $W_{1}(G_{0},\widetilde{G}_{0}) \asymp 1/n$ for the 
Gaussian case 
and $W_{1}(G_{0},\widetilde{G}_{0}) \asymp 1/\sqrt{n}$ for the Cauchy 
case. According to the result of Corollary \ref{corollary:convergence_rate_varying_G0}, 
the convergence rate of $W_{1}(\widehat{G}_{n},G_{0})$ is $n^{-1/6}$ for the Gaussian 
case and is $n^{-1/10}$ for the Cauchy case, which are also minimax according to the 
values of $W_{1}(G_{0},\widetilde{G}_{0})$. The procedure for choosing $K, \sigma_{1}, 
\sigma_{0}, n$, and $\widehat{m}_{n}$ is similar to that of Case 1.  See Figure 
\ref{figure_well_specify_kernel_varied_true} where the Wasserstein distances $W_{1}
(\widehat{G}_{n},G_{0})$ and the percentage of time $\widehat{m}_{n}=3$ are plotted 
against increasing sample size $n$ along with the error bars. The simulation results for 
both Gaussian and Cauchy family agree with the convergence rates $n^{-1/6}$ and $n^{-1/10}$ respectively.
\begin{figure*}[h]
\centering
\captionsetup{justification=centering}
\begin{minipage}[b]{.25\textwidth}
\includegraphics[width=40mm,height=40mm]{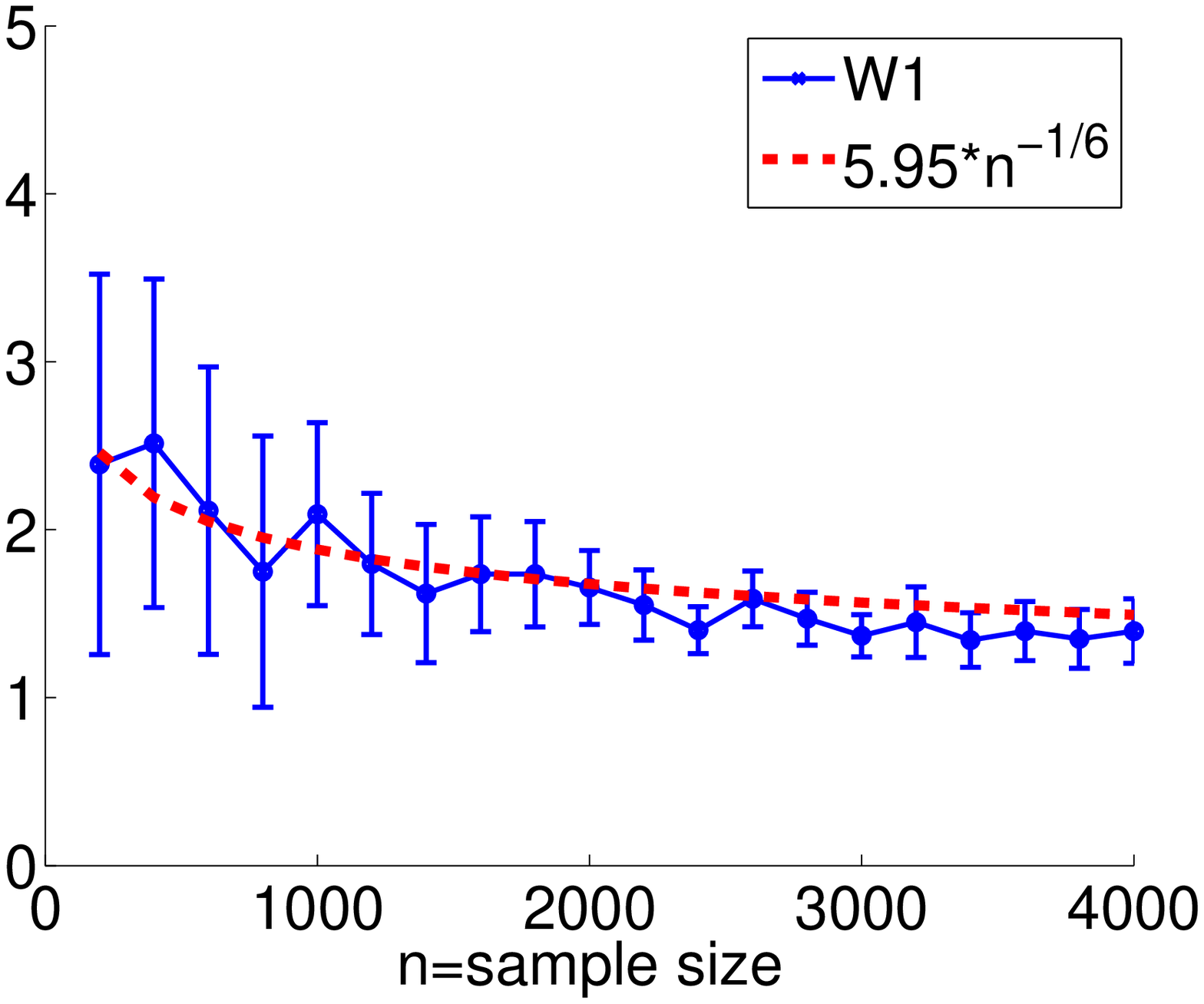}
\end{minipage}
\quad \quad
\begin{minipage}[b]{.25\textwidth}
\includegraphics[width=40mm,height=40mm]{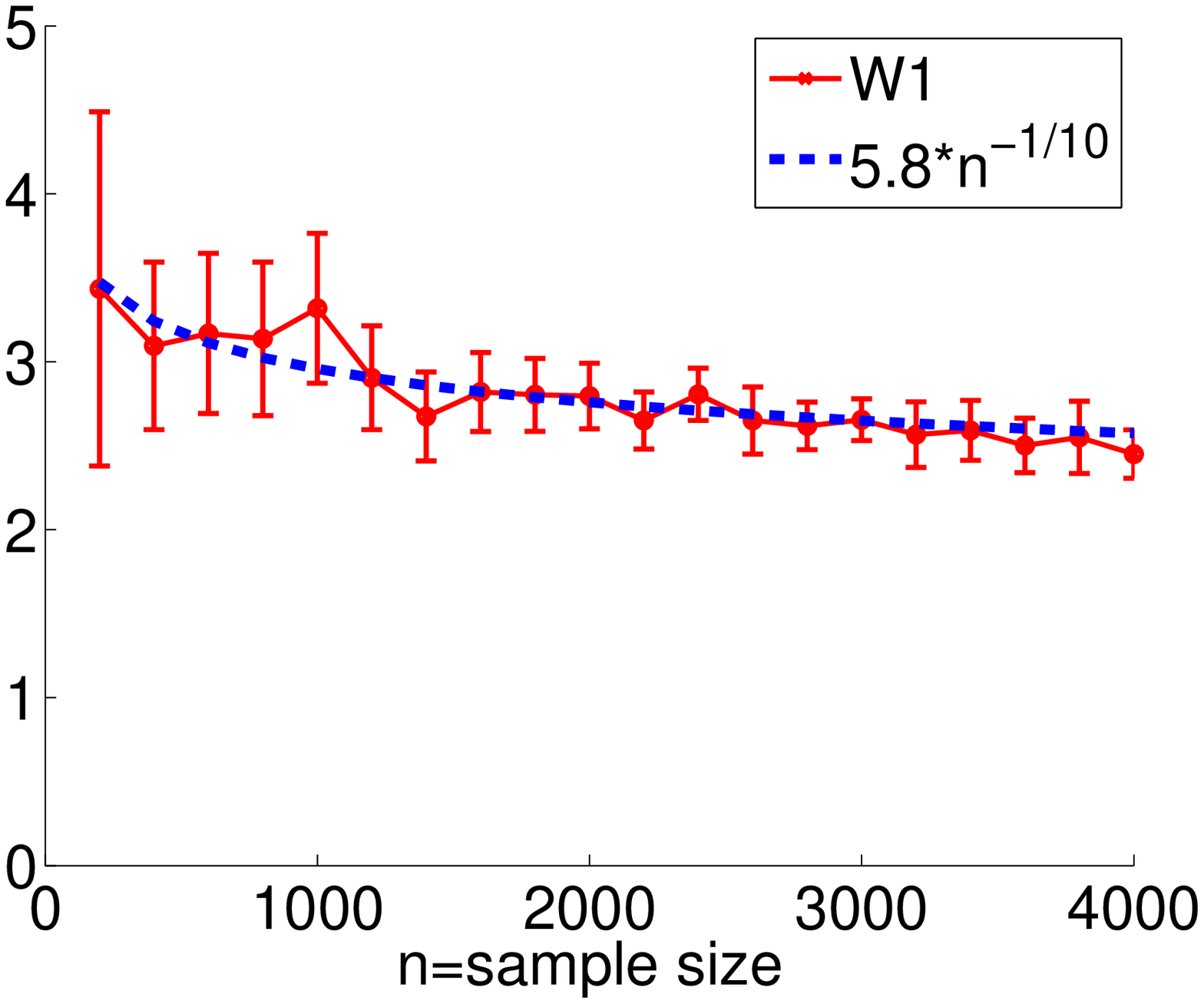}
\end{minipage}
\quad \quad
\begin{minipage}[b]{.25\textwidth}
\includegraphics[width=40mm,height=40mm]{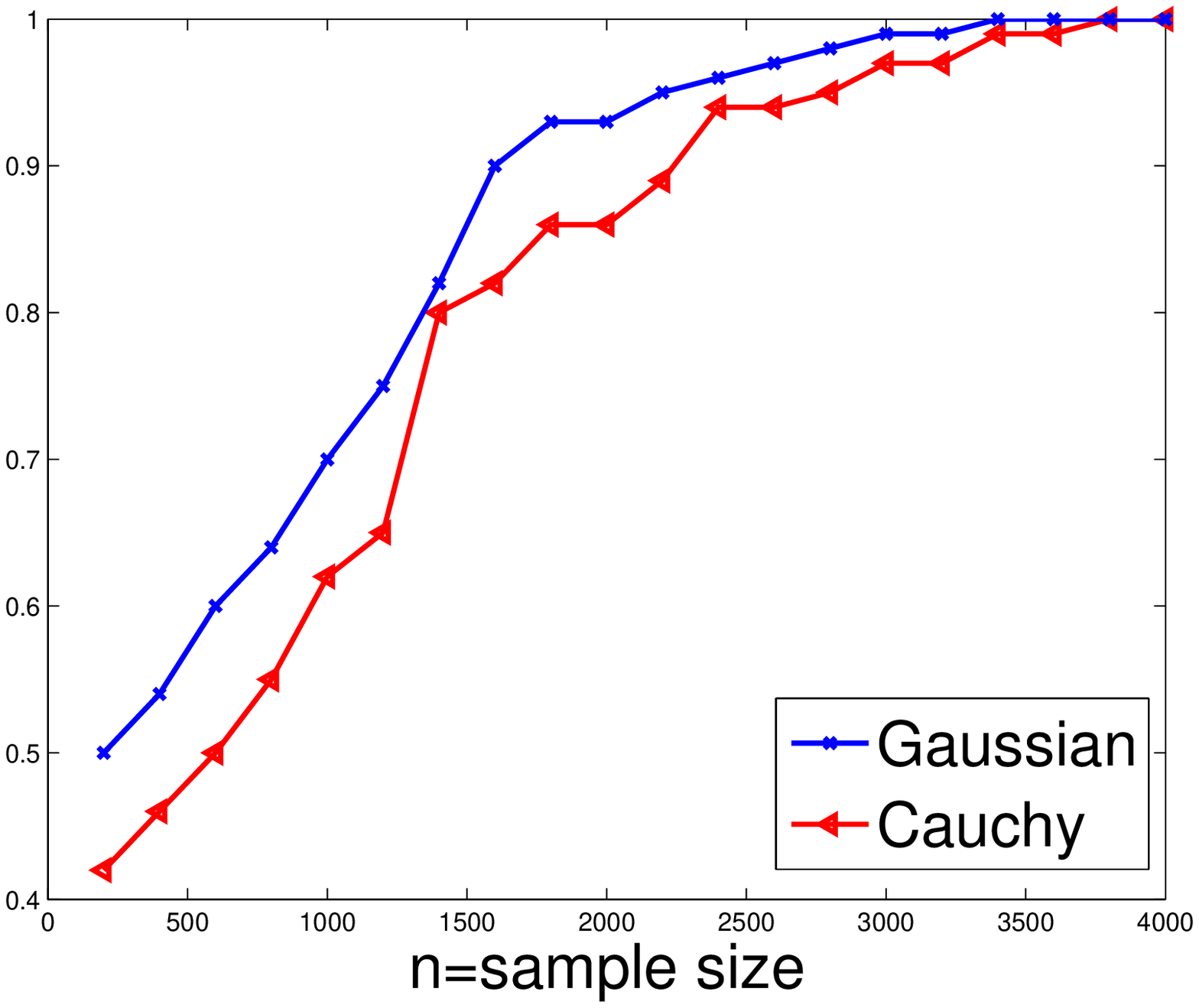}
\end{minipage}
\caption{\footnotesize{
Performance of $\widehat{G}_{n}$ in Algorithm 1 under the well-specified kernel setting and varied $G_{0}$. Left to right:
(1) $W_{1}(\widehat{G}_{n},G_{0})$ under Gaussian case.
(2) $W_{1}(\widehat{G}_{n},G_{0})$ under Cauchy case.
(3) Percentage of time $\widehat{m}_{n}=3$ obtained from 100 runs.
}}
\label{figure_well_specify_kernel_varied_true}
\end{figure*}
\paragraph{Misspecified kernel setting} Under that setting, we assess the performance of Algorithm 1 under two cases of $f, f_{0}$, $K$, $\sigma_{1}$, $\sigma_{0}$, and $G_{0}$.
\paragraph{Case 3:} $f_{0}$ is a finite mixture of $f$ and $\sigma_{1}=\sigma_{0}>0$. Under this case, we consider two choices of $f$: Gaussian and Cauchy kernel with both location and scale parameter.
\begin{itemize}
\item Case 3.1 - Gaussian distribution: $f$ is normal kernel,
\begin{eqnarray*}
  f_{0}(x|\eta,\tau) &=& \dfrac{1}{2}f(x-2|\eta,\tau)+\dfrac{1}{2}f(x+2|\eta,\tau) \\
  G_{0} &=& \dfrac{1}{3}\delta_{(0,2)}+\dfrac{2}{3}\delta_{(1,3)}.
\end{eqnarray*}
\item Case 3.2 - Cauchy distribution: $f$ is Cauchy kernel,
\begin{eqnarray*}
  f_{0}(x|\eta,\tau) &=& \dfrac{1}{2}f(x-2|\eta,\tau)+\dfrac{1}{2}f(x+2|\eta,\tau) \\
  G_{0} &=& \dfrac{1}{3}\delta_{(0,2)}+\dfrac{2}{3}\delta_{(1,3)}.
\end{eqnarray*}
\end{itemize}
With these settings of $f, f_{0}, G_{0}$, we can verify that $G_{*}=\dfrac{1}{6}
\delta_{(-2,2)}+\dfrac{1}{3}\delta_{(-1,3)}+\dfrac{1}{6}\delta_{(2,2)}+\dfrac{1}{3}
\delta_{(3,3)}$ for any $\sigma_{1}=\sigma_{0} >0$. The procedure for choosing $K, 
\sigma_{0}, n$, and $\widehat{m}_{n}$ is similar to that of Case 1 in the well-specified 
kernel setting. Figure \ref{figure_misspecify_kernel} illustrates the Wasserstein distances 
$W_{1}(\widehat{G}_{n},G_{*})$ and the percentage of time $\widehat{m}_{n}=4$ along 
with the increasing sample size $n$ and the error bars. The simulation results under that 
simple misspecified seting of both families suit with the standard $n^{-1/2}$ rate from 
Theorem \ref{theorem:convergence_rate_mixing_measure_misspecified_strong}.
\begin{figure*}[h]
\centering
\captionsetup{justification=centering}
\begin{minipage}[b]{.25\textwidth}
\includegraphics[width=40mm,height=40mm]{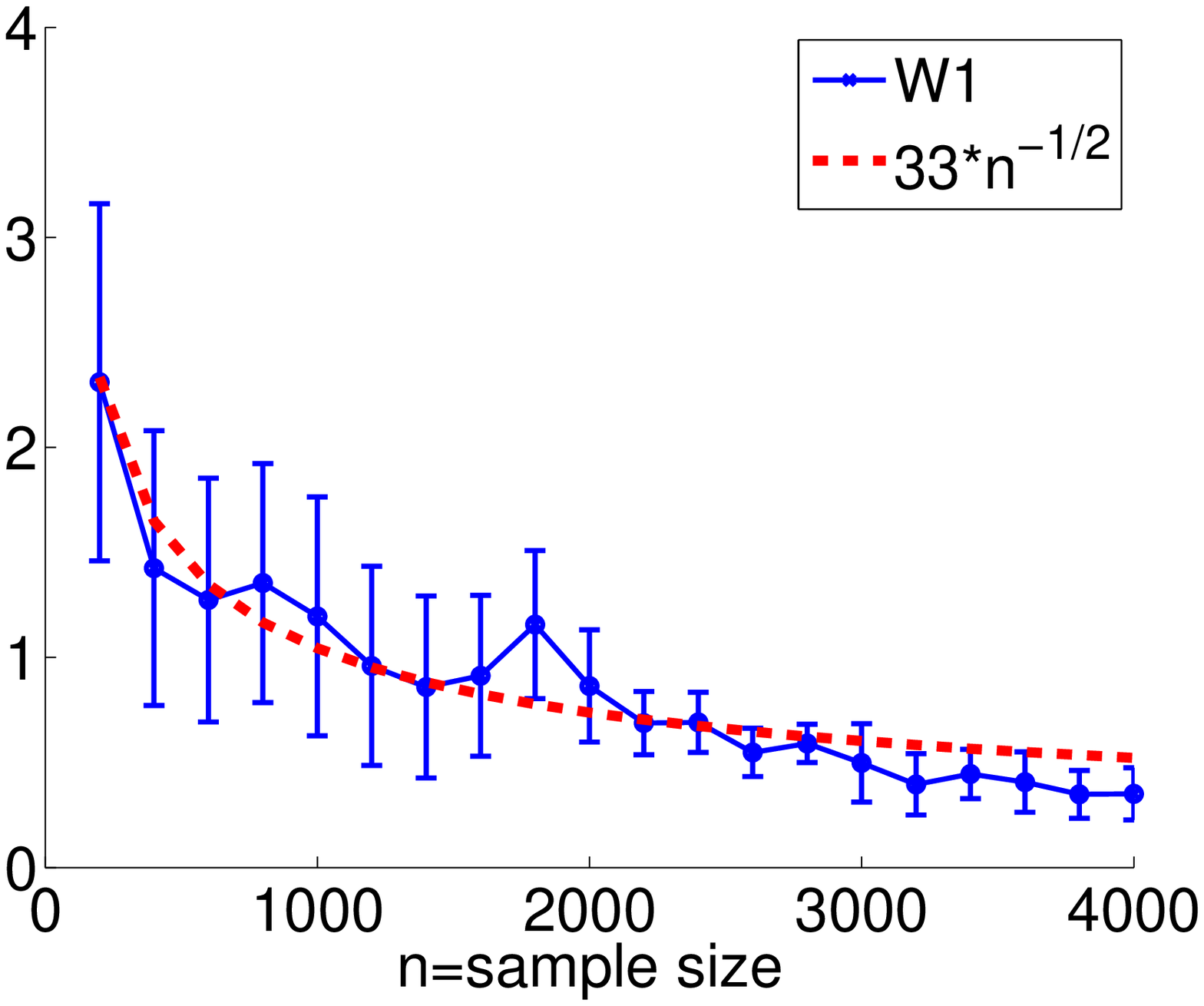}
\end{minipage}
\quad \quad
\begin{minipage}[b]{.25\textwidth}
\includegraphics[width=40mm,height=40mm]{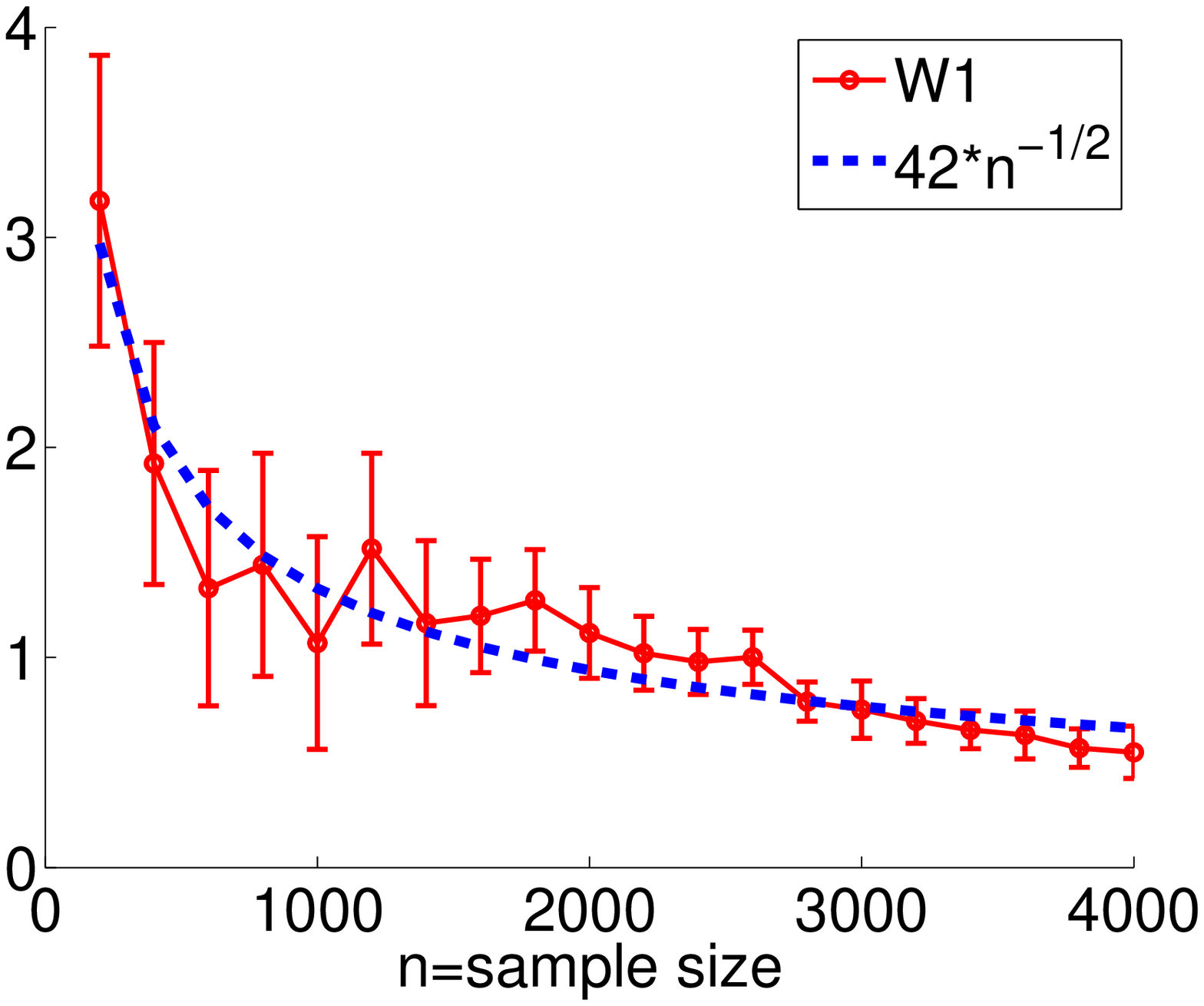}
\end{minipage}
\quad \quad
\begin{minipage}[b]{.25\textwidth}
\includegraphics[width=40mm,height=40mm]{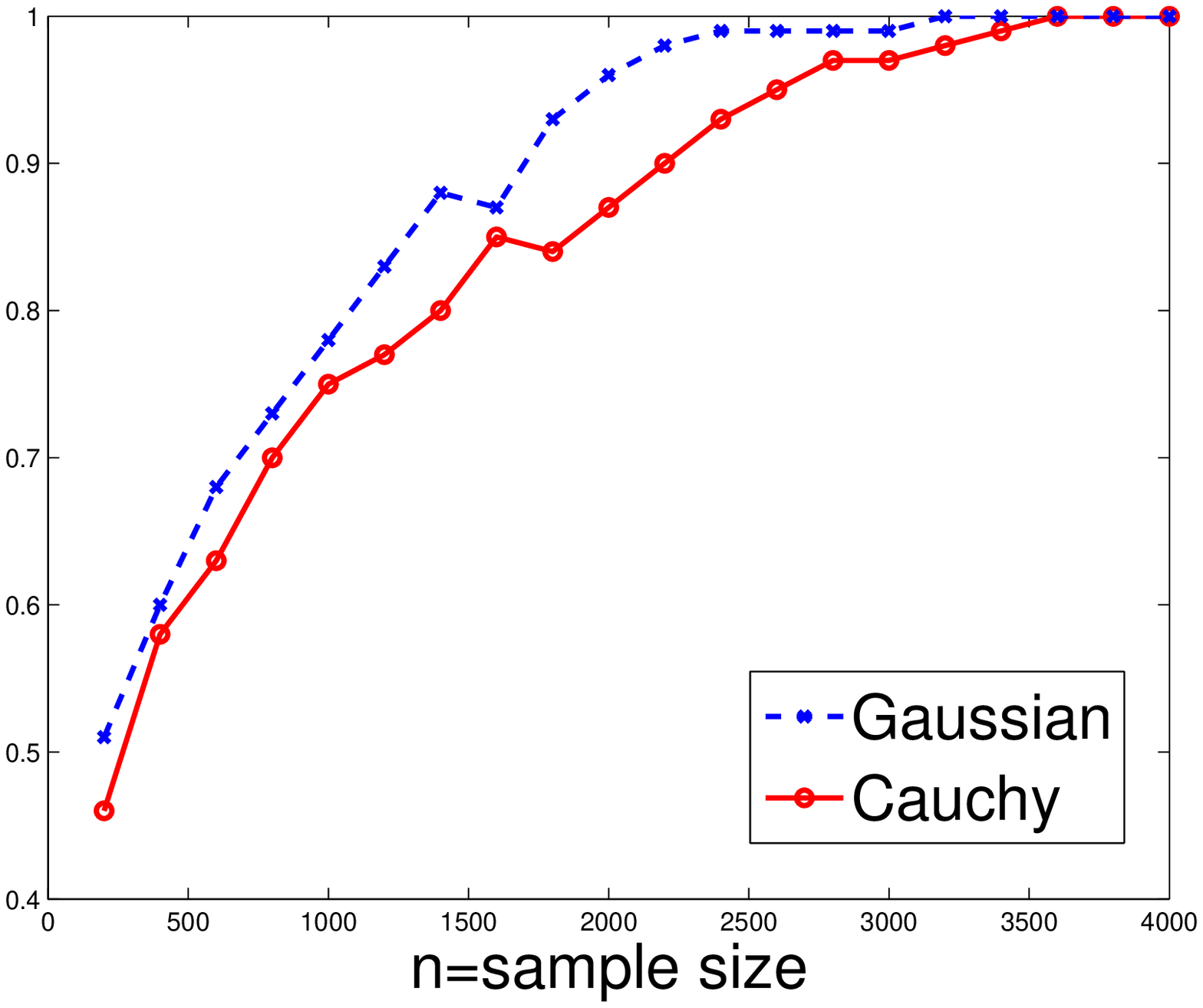}
\end{minipage}
\caption{\footnotesize{
Performance of $\widehat{G}_{n}$ in Algorithm 1 under misspecified kernel setting and $f_{0}$ is a finite mixture of $f$. Left to right:
(1) $W_{1}(\widehat{G}_{n},G_{*})$ under Gaussian case.
(2) $W_{1}(\widehat{G}_{n},G_{*})$ under Cauchy case.
(3) Percentage of time $\widehat{m}_{n}=4$ obtained from 100 runs.
}}
\label{figure_misspecify_kernel}
\end{figure*}
\paragraph{Case 4:} $\sigma_{1}, \sigma_{0}$ are chosen such that $\left\{f*K_{\sigma_{1}}\right\}=\left\{f_{0}*K_{\sigma_{0}}\right\}$. Under this case, we consider two choices of $f$ and $f_{0}$: Gaussian and Cauchy kernel with only location parameter.
\begin{itemize}
\item Case 4.1 - Gaussian distribution:
\begin{eqnarray*}
  f(x|\eta) &=& \dfrac{1}{\sqrt{2\pi}}\exp\left(-\dfrac{(x-\eta)^{2}}{2}\right), \ f_{0}(x|\eta)  =  \dfrac{1}{2\sqrt{2\pi}}\exp\left(-\dfrac{(x-\eta)^{2}}{8}\right) \\
  G_{0} &=& \dfrac{1}{3}\delta_{-1}+\dfrac{2}{3}\delta_{2}
\end{eqnarray*}
\item Case 4.2 - Cauchy distribution: $f$ is Cauchy kernel,
\begin{eqnarray*}
 f(x|\eta) &=& \dfrac{1}{\pi(1+(x-\eta)^{2})}, f_{0}(x|\eta)  =  \dfrac{4}{2\pi(4+(x-\eta)^{2})} \\
  G_{0} &=& \dfrac{1}{3}\delta_{-1}+\dfrac{2}{3}\delta_{2}
\end{eqnarray*}
\end{itemize}
To ensure that $\left\{f*K_{\sigma_{1}}\right\}=\left\{f_{0}*K_{\sigma_{0}}\right\}$,
we need to choose $\sigma_{1}^{2}+1=\sigma_{0}^{2}+4$ for both the cases of
Gaussian and Cauchy distribution when $K$ is chosen to be the standard Gaussian and
Cauchy kernel respectively. Therefore, with our simulation studies of Algorithm 1 in
this case, we choose $\sigma_{1}=2$ while $\sigma_{0}=1$. Under these choices of
bandwidths, we quickly have $G_{*}=G_{0}$. Note that, since there exists no value of $
\sigma_{0}>0$ such that $\sigma_{0}^{2}+4=1$, it implies that the estimator from WS
algorithm may not be able to estimate the true mixing measure $G_{0}$ regardless the
value of $\sigma_{0}$. Now, the procedure for choosing $K$, $n$, and $\widehat{m}_{n}
$ is similar to that of Case 1 in the well-specified kernel seting. Figure
\ref{figure_misspecify_kernel_variance} illustrates the Wasserstein distances $W_{1}
(\widehat{G}_{n},G_{0})$ and the percentage of time $\widehat{m}_{n}=2$ along with
the increasing sample size $n$ and the error bars. The simulation results under that
misspecified seting of both families fit with the standard $n^{-1/2}$ rate from Theorem
\ref{theorem:convergence_rate_mixing_measure_misspecified_strong}.
\begin{figure*}[h]
\centering
\captionsetup{justification=centering}
\begin{minipage}[b]{.25\textwidth}
\includegraphics[width=40mm,height=40mm]{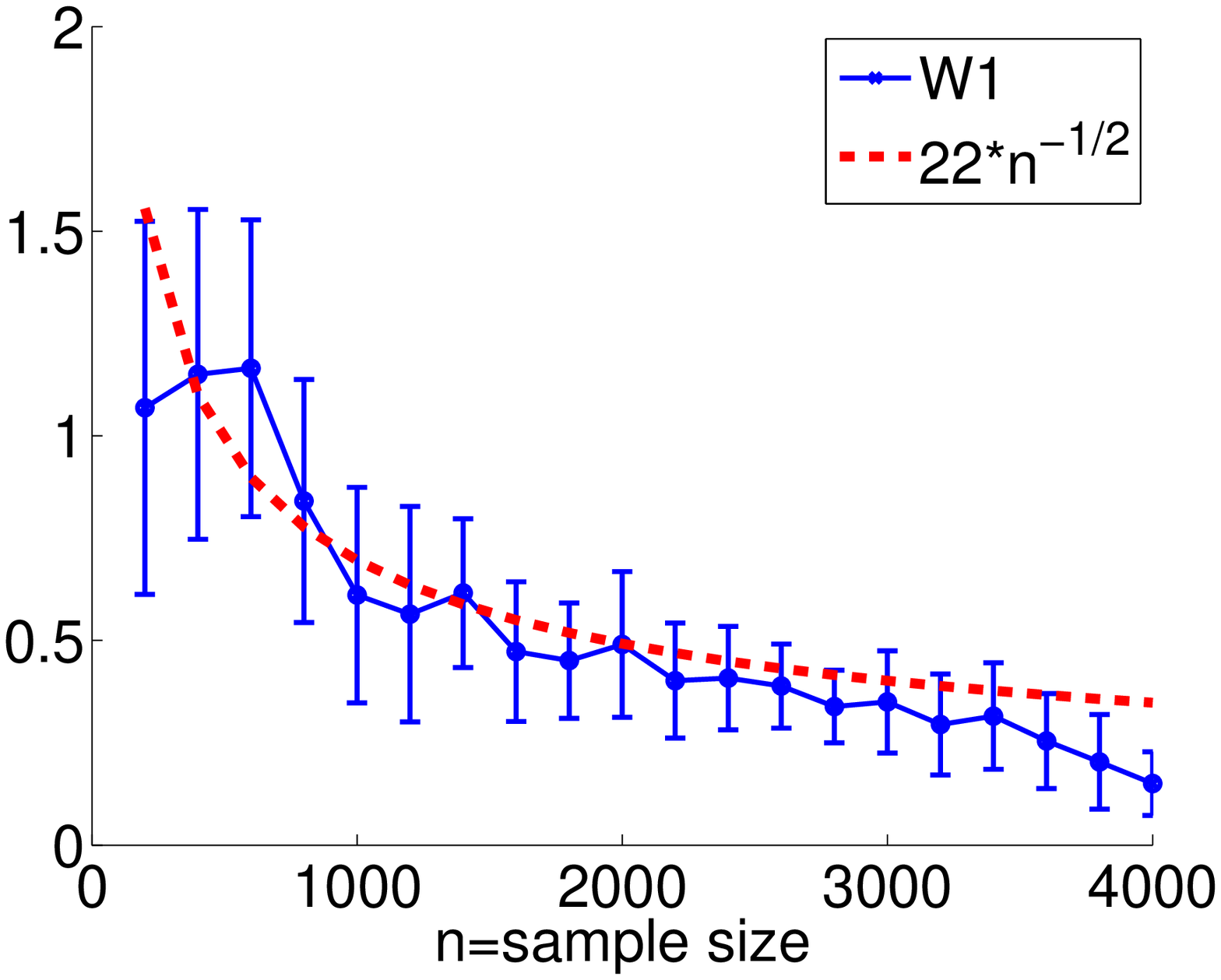}
\end{minipage}
\quad \quad
\begin{minipage}[b]{.25\textwidth}
\includegraphics[width=40mm,height=40mm]{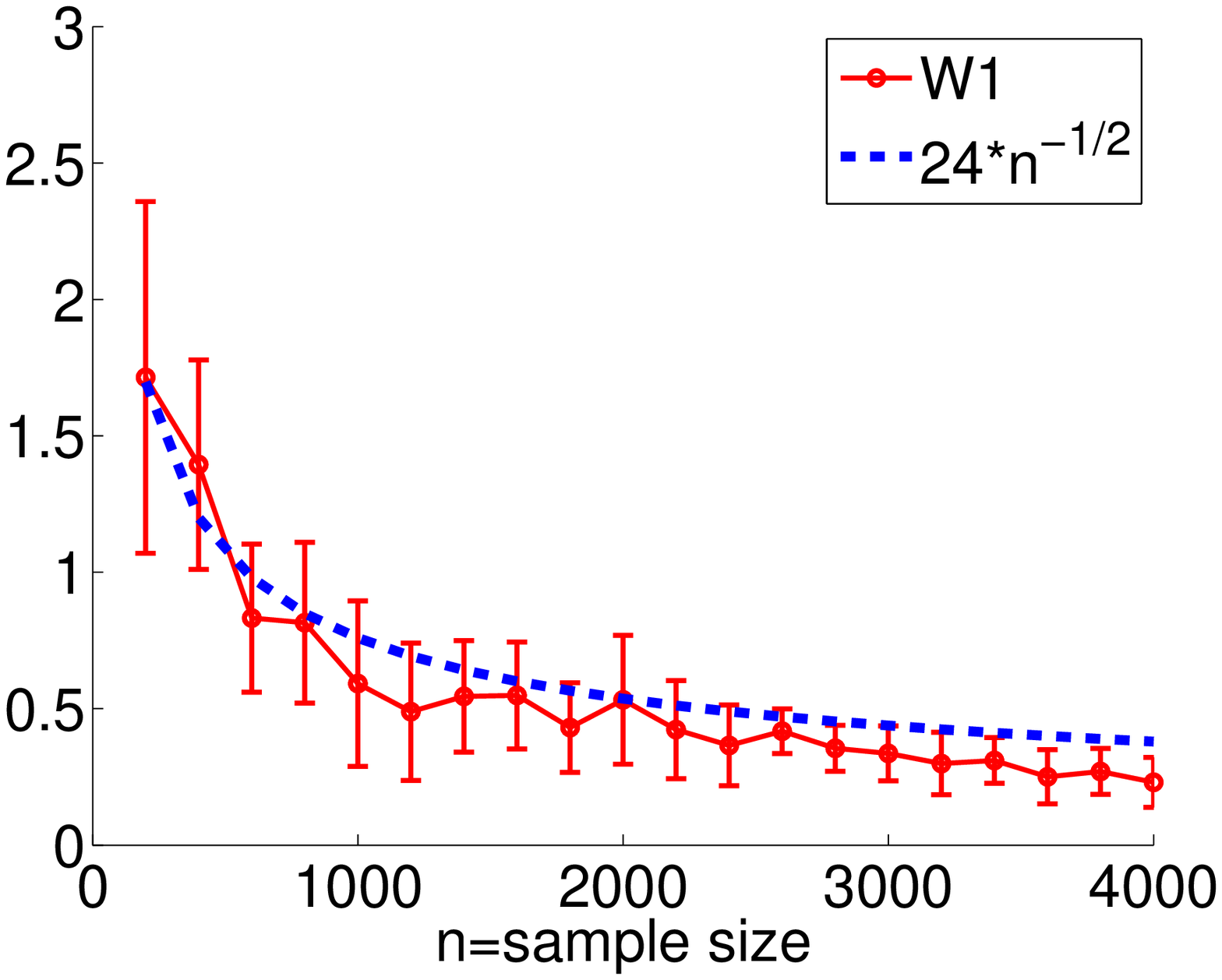}
\end{minipage}
\quad \quad
\begin{minipage}[b]{.25\textwidth}
\includegraphics[width=40mm,height=40mm]{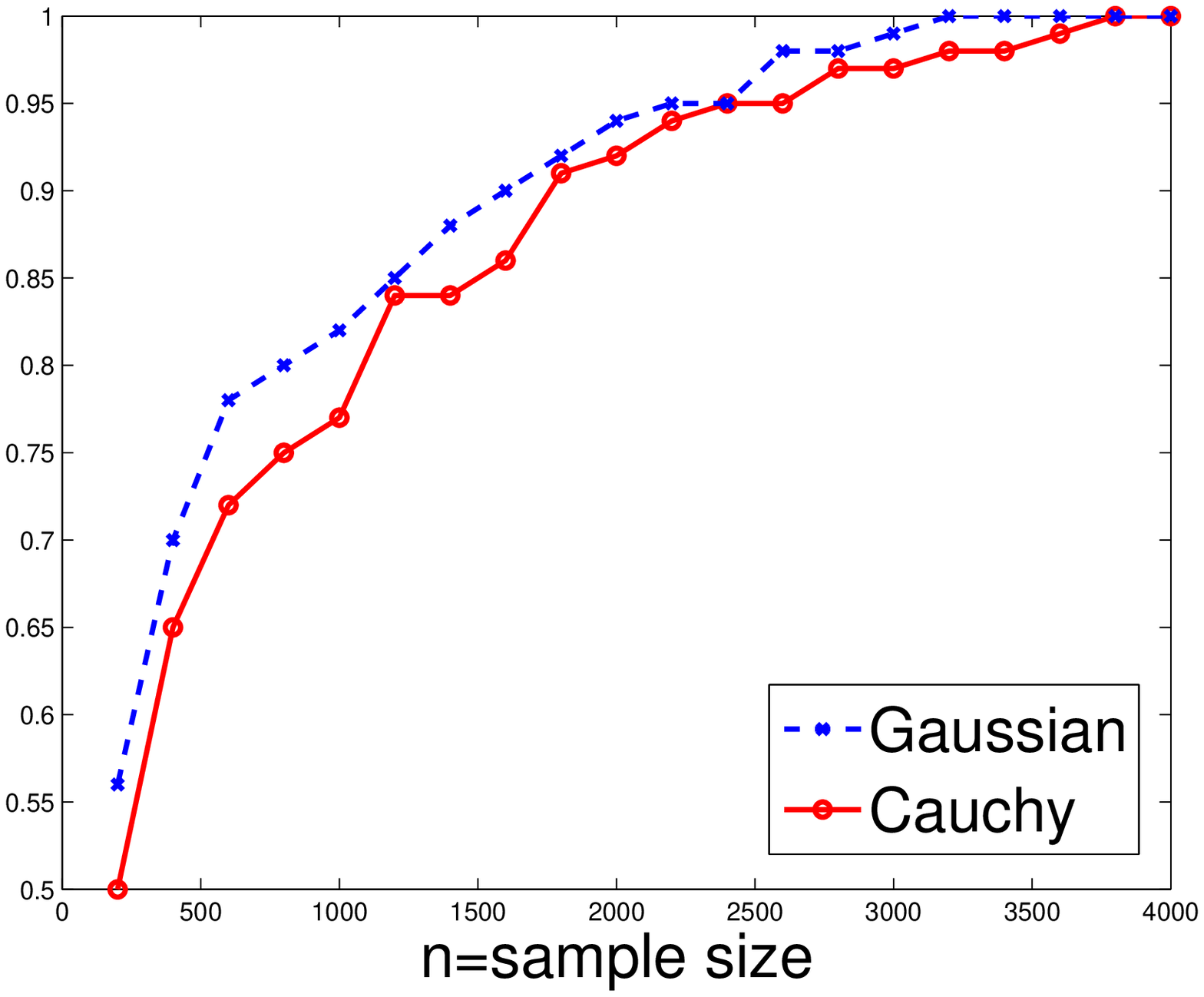}
\end{minipage}
\caption{\footnotesize{
Performance of $\widehat{G}_{n}$ in Algorithm 1 under misspecified kernel setting and $\left\{f*K_{\sigma_{1}}\right\}=\left\{f_{0}*K_{\sigma_{0}}\right\}$. Left to right:
(1) $W_{1}(\widehat{G}_{n},G_{*})$ under Gaussian case.
(2) $W_{1}(\widehat{G}_{n},G_{*})$ under Cauchy case.
(3) Percentage of time $\widehat{m}_{n}=2$ obtained from 100 runs.
}}
\label{figure_misspecify_kernel_variance}
\end{figure*}
\vspace{-10pt}
\subsection{Real Data}
We begin investigating the performance of Algorithm 1 on the well-known data set of the
Sodium-lithium countertransport (SLC) data \citep{Dudley-1991, Roeder-1994,
Ishwaran-2001}. This simple dataset includes red blood cell sodium-lithium
countertransport (SLC) activity data collected from 190 individuals. As being argued by
\cite{Roeder-1994}, the SLC activity data were believed to be derived from either mixture
of two normal distributions or mixture of three normal distributions. Therefore, we will fit
this data by using mixture of normal distributions with unknown mean and variance. We
choose the bandwidths $\sigma_{1}=\sigma_{0} = 0.05 $ and the tuning parameter
$C_{n}=\sqrt{3\log n}/\sqrt{2}$ where $n$ is the sample size. This follows BIC, which is
the criterion appropriate for modelling parameter estimation. The simulation result yields
$\widehat{m}_{n}=2$ while the values of $\widehat{G}_{n}$ are reported in Table
\ref{table_slc_activity}.

The SLC activity data was also considered in \cite{Mijawoo-2006} when the authors
achieved $\overline{m}_{n}=2$. In particular, they allowed the bandwidth $\sigma_{0}$ in
WS Algorithm to go to 0 and chose the tuning parameter $C_{n}=3/n$, which is inspired
by AIC criterion. They also obtained similar result of estimating the true number of
components when utilizing the minimum Kulback-Leibler divergence estimator (MKE) from
\citep{James-2001}. The values of parameter estimation from these two algorithms were
presented in Table 7 in \cite{Mijawoo-2006} where we will use them for the comparison
purpose with the results from Algorithm 1. Moreover, we also run the EM Algorithm to
determine the parameter estimation when we assume the data come from mixture of two
normal distributions. All the values of parameter estimation from these three algorithms
are included in Table \ref{table_slc_activity}. Finally, Figure \ref{figure_real_data}
represents the fits from parameter estimation of all the aforementioned algorithms to SLC
data. Even though the weights from Algorithm 1 are not very close to those from WS
Algorithm and EM Algorithm, the fit from Algorithm 1 is comparable to those from these
algorithms, i.e., their fits look fairly similar. As a consequence, the results from Algorithm 1
with SLC data are in agreement with those from several state-of-the-art algorithms in the
literature.
\begin{center}
\begin{table}
\captionsetup{justification=centering}
\footnotesize{
\begin{center}
\begin{tabular}{| p{2cm} | p{1cm} | p{1cm} | p{1cm} | p{1cm} | p{1cm}|p{1cm}|}
    \hline
    & $p_{1}$ & $p_{2}$ &
$\eta_{1}$ &
$\eta_{2}$ & $\tau_{1}$ & $\tau_{2}$ \\
\hline
Algorithm 1
& 0.264 & 0.736 & 0.368 & 0.231 & 0.118 & 0.065 \\
\hline
WS Algorithm
& 0.305 & 0.695 & 0.352 & 0.222 & 0.106 & 0.060 \\
\hline
MKE Algorithm
& 0.246 & 0.754 & 0.378 & 0.225 & 0.102 & 0.060 \\
\hline
EM Algorithm
& 0.328 & 0.672 & 0.363 & 0.227 & 0.115 & 0.058 \\
\hline
\end{tabular}
\caption{
\footnotesize{Summary of parameter estimates in SLC activity data from mixture of two normal distributions with Algorithm 1, WS Algorithm, MKE Algorithm, and EM Algorithm. Here, $p_{i},\eta_{i},\tau_{i}$ represents the weights, means, and variance respectively.}}
\label{table_slc_activity}
\end{center}}
\end{table}
\end{center}
\begin{figure*}[h]
\centering
\captionsetup{justification=centering}
\begin{minipage}[b]{.45\textwidth}
\includegraphics[width=60mm,height=50mm]{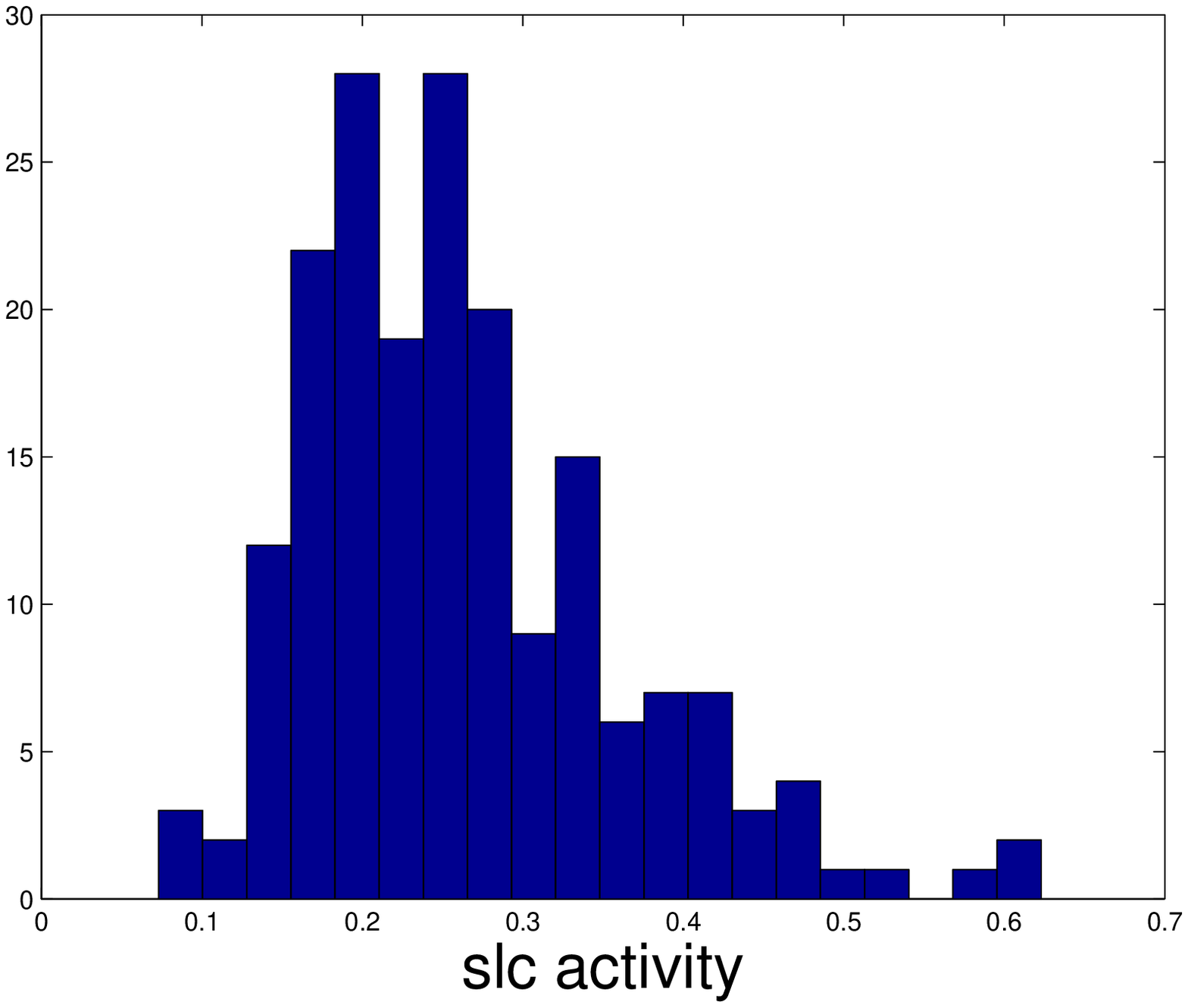}
\end{minipage}
\quad \quad
\begin{minipage}[b]{.45\textwidth}
\includegraphics[width=70mm,height=50mm]{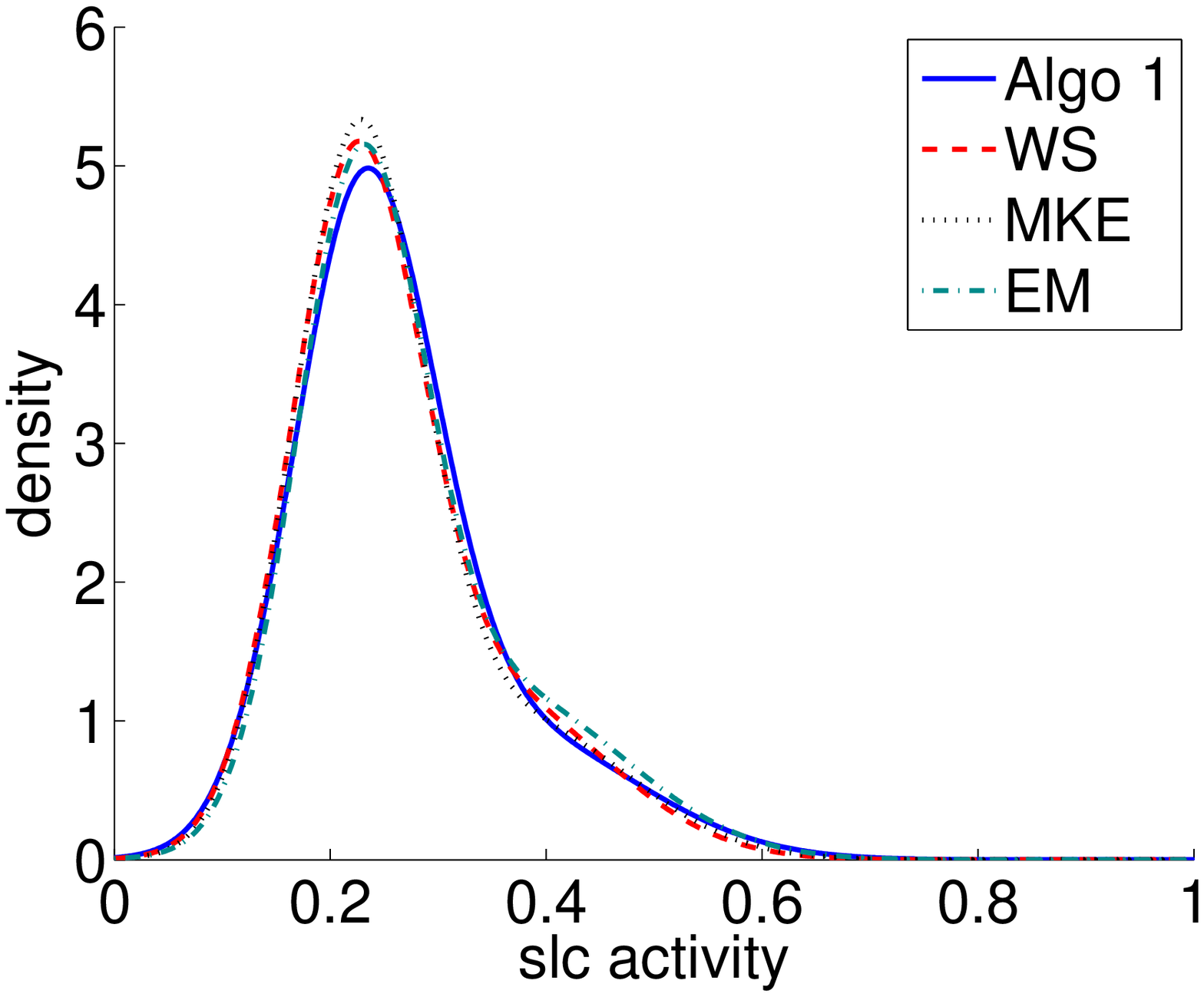}
\end{minipage}
\caption{\footnotesize{
From left to right:
(1) Histogram of SLC activity data.
(2) Density plot from mixture of two normals based on Algorithm 1, WS Algorithm, MKE Algorithm, and MLE.
}}
\label{figure_real_data}
\end{figure*}
\vspace{-10pt}
\section{Summaries and discussions}
\label{Section:discussion}
In this paper, we propose flexible robust estimators of mixing measure in finite mixture
models based on the idea of minimum Hellinger distance estimator, model selection criteria,
and super-efficiency phenomenon. Our estimators are shown to exhibit the consistency of
the number of components under both the well- and mis-specified kernel setting.
Additionally, the best possible convergence rates of parameter estimation are derived
under various settings of both kernel $f$ and $f_{0}$. Another salient feature of our
estimators is the flexible choice of bandwidths, which circumvents the subtle choices of
bandwidth from proposed estimators in the literature. However, there are still many open
questions relating to the performance or the extension of our robust estimators in the
paper. We give several examples:
\begin{itemize}
\item[•] As being mentioned in the paper, our estimator in Algorithm 1 and WS estimator
achieve the consistency of the number of components when the bandwidth $\sigma_{0}$
goes to 0 sufficiently slow. Can we determine the setting of bandwidth such that the
convergence rates of parameter estimation from these estimators are optimal, at least
under the well-specified kernel setting?
\item[•] Our analysis is based on the assumption that the parameters of $G_{0}$ belong
to the compact set $\Theta$. When $G_0$ is finitely supported, this is always the case,
but the set is unknown in advance and, in practice, we often do not know the range of the
true parameters. Therefore, it would be interesting to see whether our estimators in
Algorithm 1 and Algorithm 2 still achieve both the consistency of the number of
components and best possible convergence rates of parameter estimation when $\Theta=
\mathbb{R}^{d_{1}}$.
\item[•] Bayesian robust inference of mixing measure in finite mixture models has been of
interest recently, see for example \citep{Miller-2015}. Whether the idea of minimum
Hellinger distance estimator can be adapted to that setting is also an interesting direction
to consider in the future.
\end{itemize}

\section{Proofs of key results}
In this section, we provide the proofs of Theorem \ref{theorem:convergence_rate_mixing_measure}
and Theorem \ref{theorem:convergence_rate_mixing_measure_misspecified_strong} in
Section \ref{Section:minimum_Hellinger}. The remaining proofs are given in the Appendices.
\label{Section:proof}
\paragraph{PROOF OF THEOREM \ref{theorem:convergence_rate_mixing_measure}}
We divide the main argument into three key steps:
\paragraph{Step 1:} $\widehat{m}_{n} \to k_{0}$ almost surely. The proof of this step follows the argument from \citep{Leroux-1992}. In fact, for any positive integer $m$ we denote
\begin{eqnarray}
G_{0,m}=\mathop {\arg \min}\limits_{G \in \mathcal{O}_{m}}{h(p_{G,f_{0}} * K_{\sigma_{0}}, p_{G_{0},f_{0}}*K_{\sigma_{0}})}. \nonumber
\end{eqnarray}
Now, as $n \to \infty$ we have almost surely that
\begin{eqnarray}
h(p_{\widehat{G}_{n,m},f_{0}}*K_{\sigma_{0}},P_{n}*K_{\sigma_{0}}) - h(p_{\widehat{G}_{n,m+1},f_{0}}*K_{\sigma_{0}},P_{n}*K_{\sigma_{0}}) \to d_{m},\nonumber
\end{eqnarray}
where $d_{m}=h(p_{G_{0,m},f_{0}}*K_{\sigma_{0}},p_{G_{0},f_{0}}*K_{\sigma_{0}})-
h(p_{G_{0,m+1},f_{0}}*K_{\sigma_{0}},p_{G_{0},f_{0}}*K_{\sigma_{0}})$ and the limit
is due to the fact that $h(P_{n} * K_{\sigma_{0}}, p_{G_{0},f_{0}} *K_{\sigma_{0}}) \to
0$ almost surely for all $\sigma_{0}>0$. 
From the formulation of Step 2 in Algorithm 1 and the fact that $C_{n}n^{-1/2} \to 0$ as $n \to \infty$, we obtain
\begin{eqnarray}
h(p_{\widehat{G}_{n,\widehat{m}_{n}},f_{0}}*K_{\sigma_{0}},P_{n}*K_{\sigma_{0}}) - h(p_{\widehat{G}_{n,\widehat{m}_{n}+1},f_{0}}*K_{\sigma_{0}},P_{n}*K_{\sigma_{0}}) \leq C_{n}n^{-1/2} \to 0. \nonumber
\end{eqnarray}
Therefore, to demonstrate that $\widehat{m}_{n} \to
k_{0}$ almost surely, it is sufficient to prove that $d_{m}=0$ as $m \geq k_{0}$ and
$d_{m}>0$ as $m < k_{0}$. In fact, as $m \geq k_{0}$, we have $\inf \limits_{G \in
\mathcal{O}_{m}}{h(p_{G,f_{0}} *K_{\sigma_{0}} ,p_{G_{0},f_{0}}*K_{\sigma_{0}})}
=0$. Therefore, $d_{m}=0$ as $m \geq k_{0}$.

When $m<k_{0}$, we assume that $d_{m}=0$, i.e., $h(p_{G_{0,m},f_{0}}*K_{\sigma_{0}},p_{G_{0},f_{0}}*K_{\sigma_{0}})=h(p_{G_{0,m+1},f_{0}}*K_{\sigma_{0}},p_{G_{0},f_{0}}*K_{\sigma_{0}})$. It implies that
\begin{eqnarray}
h(p_{G_{0,m},f_{0}}*K_{\sigma_{0}},p_{G_{0},f_{0}}*K_{\sigma_{0}}) \leq h(p_{G,f_{0}} * K_{\sigma_{0}},p_{G_{0},f_{0}}*K_{\sigma_{0}}) \ \ \forall \ G \in \mathcal{O}_{m+1}. \nonumber
\end{eqnarray}
For any $\epsilon>0$, we choose $G=(1-\epsilon) G_{0,m}+ \epsilon\delta_{\theta}$ where $\theta \in \Theta$ is some component. The inequality in the above display implies that
\begin{eqnarray}
\int {(p_{G_{0},f_{0}}*K_{\sigma_{0}}(x))^{1/2}\biggr(\biggr[(1-\epsilon) p_{G_{0,m},f_{0}}*K_{\sigma_{0}}(x)+\epsilon f_0*K_{\sigma_{0}}(x|\theta) \biggr]^{1/2}} \nonumber \\
-(p_{G_{0,m},f_{0}}*K_{\sigma_{0}}(x))^{1/2}\biggr)\textrm{d}x \leq 0. \nonumber
\end{eqnarray}
As $\epsilon \to 0$, the above inequality divided by $\epsilon$ becomes
\begin{eqnarray}
\int {(p_{G_{0},f_{0}}*K_{\sigma_{0}}(x))^{1/2}(p_{G_{0,m},f_{0}}*K_{\sigma_{0}}(x))^{1/2}}\textrm{d}x \geq \nonumber \\
\int {(p_{G_{0},f_{0}}*K_{\sigma_{0}}(x))^{1/2}f_{0}*K_{\sigma_{0}}(x|\theta) (p_{G_{0,m},f_{0}}*K_{\sigma_{0}}(x))^{-1/2}}\textrm{d}x. \nonumber
\end{eqnarray}
Now, by choosing $\theta=\theta_{i}^{0}$ for all $1 \leq i \leq k_{0}$, as we sum up the right hand side of the above inequality, we obtain
\begin{eqnarray}
& & \int {(p_{G_{0},f_{0}}*K_{\sigma_{0}}(x))^{1/2}(p_{G_{0,m},f_{0}}*K_{\sigma_{0}}(x))^{1/2}}\textrm{d}x \nonumber \\
& & \hspace{5 em} \geq  \int (p_{G_{0},f_{0}}*K_{\sigma_{0}}(x))^{1/2}\biggr(\sum \limits_{i=1}^{k_{0}}{p_{i}^{0}f_{0}*K_{\sigma_{0}}(x|\theta_{i}^{0})}\biggr)(p_{G_{0,m},f_{0}}*K_{\sigma_{0}}(x))^{-1/2} \nonumber \\
& & \hspace{5 em} \geq \int (p_{G_{0},f_{0}}*K_{\sigma_{0}}(x))^{3/2}(p_{G_{0,m},f_{0}}*K_{\sigma_{0}}(x))^{-1/2} \textrm{d}x
 \geq  1 \nonumber
\end{eqnarray}
where the final inequality is due to the inequality $\int q_{1}^{3/2}(x)q_{2}^{-1/2}(x)dx \geq 1$ for any two density functions $q_{1}(x)$ and $q_{2}(x)$. Therefore, we have $h(p_{G_{0,m},f_{0}}*K_{\sigma_{0}},p_{G_{0},f_{0}}
*K_{\sigma_{0}})=0$. Due to the identifiability assumption of $f_{0}*K_{\sigma_{0}}$,
the previous equation implies that $G_{0,m} \equiv G_{0}$, which is a contradiction as $m
<k_{0}$. Thus, we have $d_{m}>0$ for any $m<k_{0}$. We achieve the conclusion that $
\widehat{m}_{n} \to k_{0}$ almost surely.
\paragraph{Step 2:} $h(P_{n}*K_{\sigma_{0}},p_{G_{0},f_{0}}*K_{\sigma_{0}})=O_{p}
\biggr(\sqrt{\dfrac{\Psi(G_{0},\sigma_{0})}{n}}\biggr)$. Indeed, by means of Taylor
expansion up to the first order, we have
\begin{eqnarray}
& & h^{2}(P_{n}*K_{\sigma_{0}},p_{G_{0},f_{0}}*K_{\sigma_{0}}) \nonumber \\
& & \hspace{3 em} = \int {\biggr(1-\sqrt{1+\dfrac{P_{n}*K_{\sigma_{0}}(x)-p_{G_{0},f_{0}}*K_{\sigma_{0}}(x)}{p_{G_{0},f_{0}}*K_{\sigma_{0}}(x)}}\biggr)^{2}p_{G_{0},f_{0}}*K_{\sigma_{0}}(x)}\textrm{d}x \nonumber \\
& & \hspace{3 em} \simeq  \dfrac{1}{4}\int {\dfrac{(P_{n}*K_{\sigma_{0}}(x)-p_{G_{0},f_{0}}*K_{\sigma_{0}}(x))^{2}}{p_{G_{0},f_{0}}*K_{\sigma_{0}}(x)}}\textrm{d}x. \nonumber
\end{eqnarray}
Notice that,
\begin{eqnarray}
E\biggr(\int {\dfrac{(P_{n}*K_{\sigma_{0}}(x)-p_{G_{0},f_{0}}*K_{\sigma_{0}}(x))^{2}}{p_{G_{0},f_{0}}*K_{\sigma_{0}}(x)}}\textrm{d}x\biggr)=\int {\dfrac{\text{Var}(P_{n}*K_{\sigma_{0}}(x))}{p_{G_{0},f_{0}}*K_{\sigma_{0}}(x)}}\textrm{d}x, \nonumber
\end{eqnarray}
From assumption (P.2), we obtain ${\displaystyle \int {\dfrac{\text{Var}(P_{n}
*K_{\sigma_{0}}(x))}{p_{G_{0},f_{0}}*K_{\sigma_{0}}(x)}}\textrm{d}x=O
\biggr(\dfrac{\Psi(G_{0},\sigma_{0})}{n}\biggr)}$. It follows that
\begin{eqnarray}
E\biggr(\int {\dfrac{(P_{n}*K_{\sigma_{0}}(x)-p_{G_{0},f_{0}}*K_{\sigma_{0}}(x))^{2}}{p_{G_{0},f_{0}}*K_{\sigma_{0}}(x)}}\textrm{d}x\biggr)=
O\biggr(\dfrac{\Psi(G_{0},\sigma_{0})}{n}\biggr). \nonumber
\end{eqnarray}
Therefore, we achieve $h(P_{n}*K_{\sigma_{0}},p_{G_{0},f_{0}}
*K_{\sigma_{0}})=O_{p}\biggr(\sqrt{\dfrac{\Psi(G_{0},\sigma_{0})}{n}}\biggr)$. It
implies that for any $\epsilon>0$, we can find $M_{\epsilon}>0$ and the index $N_{1}
(\epsilon) \geq 1$ such that
\begin{eqnarray}
P\biggr(h(P_{n}*K_{\sigma_{0}},p_{G_{0},f_{0}}*K_{\sigma_{0}})>M_{\epsilon}\sqrt{\dfrac{\Psi(G_{0},\sigma_{0})}{n}}\biggr)<\epsilon/2 \label{eqn:well_specified_first_inequality}
\end{eqnarray}
for all $n \geq N_{1}(\epsilon)$.
\paragraph{Step 3:} Now, denote the event $A=\left\{\widehat{m}_{n} \to k_{0} \
\text{as} \ n \to \infty \right\}$. Under this event, for each $\omega \in A$, we can find
$N(\omega)$ such that as $n \geq N(\omega)$, we have $\widehat{m}_{n}=k_{0}$. It
suggests that $\widehat{G}_{n} \in \mathcal{O}_{k_{0}}$ as $n \geq N(\omega)$.
Define $A_{m}=\left\{\omega \in A: \ \forall \ n \geq m \ \text{we have} \ \widehat{m}
_{n}=k_{0}\right\}$. From this definition, we obtain $A_{1} \subset A_{2} \ldots \subset
A_{m} \subset \ldots$ and $\bigcup \limits_{m=1}^{\infty}{A_{m}}=A$. Therefore, $\lim
\limits_{m \to \infty}{P(A_{m})}=P(A)=1$. Therefore, for any $\epsilon>0$ we can find
the corresponding index $N_{2}(\epsilon)$ such that $P(A_{N_{2}(\epsilon)})>1-
\epsilon/2$.

Now, for any $\omega \in A_{N_{2}(\epsilon)}$, we have $\widehat{m}_{n}=k_{0}$ as
$n \geq N_{2}(\epsilon)$. From assumptions (P.1) and the definition of $C_{1}
(\sigma_{0})$ in Theorem \ref{theorem:convergence_rate_mixing_measure}, we obtain
\begin{eqnarray}
C_{1}(\sigma)W_{1}(\widehat{G}_{n},G_{0}) & \leq & h(p_{\widehat{G}_{n},f_{0}}*K_{\sigma_{0}},p_{G_{0},f_{0}}*K_{\sigma_{0}}) \nonumber \\
& \leq & h(p_{\widehat{G}_{n},f_{0}}*K_{\sigma_{0}},P_{n}*K_{\sigma_{0}})+h(P_{n}*K_{\sigma_{0}},p_{G_{0},f_{0}}*K_{\sigma_{0}}) \nonumber \\
& \leq & 2 h(P_{n}*K_{\sigma_{0}},p_{G_{0},f_{0}}*K_{\sigma_{0}}). \label{eqn:well_specified_second_inequality}
\end{eqnarray}
Using the inequalities from \eqref{eqn:well_specified_first_inequality} and \eqref{eqn:well_specified_second_inequality}, we have
\begin{eqnarray}
P\biggr(W_{1}(\widehat{G}_{n},G_{0})>2M_{\epsilon}\sqrt{\dfrac{\Psi(G_{0},\sigma_{0})}{C_{1}^{2}(\sigma_{0})n}}\biggr)  =  P\biggr(\biggr(W_{1}(\widehat{G}_{n},G_{0})>2M_{\epsilon}\sqrt{\dfrac{\Psi(G_{0},\sigma_{0})}{C_{1}^{2}(\sigma_{0})n}}\biggr)\mathbbm{1}_{A_{N_{2}(\epsilon_{2})}^{c}}\biggr) \nonumber \\
 +  P\biggr(\biggr(W_{1}(\widehat{G}_{n},G_{0})>2M_{\epsilon}\sqrt{\dfrac{\Psi(G_{0},\sigma_{0})}{C_{1}^{2}(\sigma_{0})n}}\biggr)\mathbbm{1}_{A_{N_{2}(\epsilon)}}\biggr) \nonumber \\
 \leq  \epsilon/2 + P\biggr(\biggr(W_{1}(\widehat{G}_{n},G_{0})>2M_{\epsilon}\sqrt{\dfrac{\Psi(G_{0},\sigma_{0})}{C_{1}^{2}(\sigma_{0})n}}\biggr)\mathbbm{1}_{A_{N_{2}(\epsilon)}}\biggr) < \epsilon \nonumber
\end{eqnarray}
for all $n \geq \mathop {\max }{\left\{N_{1}(\epsilon),N_{2}(\epsilon)\right\}}$. We achieve the conclusion of the theorem.
\paragraph{PROOF OF THEOREM \ref{theorem:convergence_rate_mixing_measure_misspecified_strong}}
We divide our argument in the proof of this theorem into two key steps.

\paragraph{Step 1} $\widehat{m}_{n} \to k_{*}$ almost surely. Indeed, by carrying out
the same argument as that of Step 1 in the proof of Theorem
\ref{theorem:convergence_rate_mixing_measure} (here, we replace $f_{0}$ by $f$ and
$G_{0,m}$ by $G_{*,m}$, as $m<k_{*}$), we eventually obtain the following inequality
\begin{eqnarray}
\int {(p_{G_{0},f_{0}}*K_{\sigma_{0}}(x))^{1/2}(p_{G_{*,m},f}*K_{\sigma_{1}}(x))^{1/2}}\textrm{d}x \geq \nonumber \\
\int {(p_{G_{0},f_{0}}*K_{\sigma_{0}}(x))^{1/2}f*K_{\sigma_{1}}(x|\theta)(p_{G_{*,m},f}*K_{\sigma_{1}}(x))^{-1/2}}\textrm{d}x. \nonumber
\end{eqnarray}
for any $\theta \in \Theta$. By choosing $\theta \in \text{supp}(G_{*})$, which is the set of all support points of $G_{*}$, and sum over all of these components, we achieve
\begin{eqnarray}
\int {(p_{G_{0},f_{0}}*K_{\sigma_{0}}(x))^{1/2}(p_{G_{*,m},f}*K_{\sigma_{1}}(x))^{1/2}}\textrm{d}x \geq \nonumber \\
\int {(p_{G_{0},f_{0}}*K_{\sigma_{0}}(x))^{1/2}p_{G_{*},f}*K_{\sigma_{1}}(x)(p_{G_{*,m},f}*K_{\sigma_{1}}(x))^{-1/2}}\textrm{d}x. \nonumber
\end{eqnarray}
From the above inequality, we have
\begin{eqnarray}
\int {\biggr(\sqrt{p_{G_{*,m},f}*K_{\sigma_{1}}(x)}-\sqrt{p_{G_{*},f}*K_{\sigma_{1}}(x)}\biggr)^{2}\sqrt{\dfrac{p_{G_{0},f_{0}}*K_{\sigma_{0}}(x)}{p_{G_{*,m},f}*K_{\sigma_{1}}(x)}}}\textrm{d}x \nonumber \\
\leq 2\biggr(\int {\sqrt{p_{G_{0},f_{0}}*K_{\sigma_{0}}(x)}\sqrt{p_{G_{*,m},f}*K_{\sigma_{1}}(x)}}\textrm{d}x \nonumber \\
- \int {\sqrt{p_{G_{0},f_{0}}*K_{\sigma_{0}}(x)}\sqrt{p_{G_{*},f}*K_{\sigma_{1}}(x)}}\textrm{d}x\biggr) \leq 0, \nonumber
\end{eqnarray}
where the second inequality is due to the fact that $G_{*}$ minimizes $h(p_{G,f}
*K_{\sigma_{1}},p_{G_{0},f_{0}}*K_{\sigma_{0}})$ among all $G \in 
\overline{\mathcal{G}}$. The above inequality implies that $p_{G_{*,m},f}
*K_{\sigma_{1}}(x)=p_{G_{*},f}*K_{\sigma_{1}}(x)$ for almost surely $x \in \mathcal{X}
$. Due to the identifiability of $f*K_{\sigma_{1}}$, we obtain $G_{*,m} \equiv G_{*}$, 
which is a contradiction to the fact that $m<k_{*}$. Therefore, we achieve $\widehat{m}
_{n} \to k_{*}$ almost surely.
\paragraph{Step 2} Now, since $\widehat{m}_{n} \to k_{*}$ almost surely, using the 
same argument as Step 3 in the proof of Theorem 
\ref{theorem:convergence_rate_mixing_measure}, we can find $N(\epsilon)$ such that $
\widehat{m}_{n} = k_{*}$ for any $n \geq N(\epsilon)$ and such that 
$P(A_{N(\epsilon)})>1-\epsilon/2$ for any $\epsilon>0$. Additionally, since $\widehat{G}
_{n}=\widehat{G}_{n,\widehat{m}_{n}}$ minimizes $h(p_{G,f}*K_{\sigma_{1}}, P_{n}
*K_{\sigma_{0}})$ among all $G \in \mathcal{O}_{\widehat{m}_{n}}$, it follows that
\begin{eqnarray}
\int{\sqrt{p_{\widehat{G}_{n},f}*K_{\sigma_{1}}(x)}\sqrt{P_{n}*K_{\sigma_{0}}(x)}}\textrm{d}x \geq \int{\sqrt{p_{G_{*},f}*K_{\sigma_{1}}(x)}\sqrt{P_{n}*K_{\sigma_{0}}(x)}}\textrm{d}x. \nonumber
\end{eqnarray}
when $n \geq N(\epsilon)$. From this inequality, we obtain
\begin{eqnarray}
& & \hspace{-2em} \int{\biggr(\sqrt{p_{\widehat{G}_{n},f}*K_{\sigma_{1}}(x)}-\sqrt{p_{G_{*},f}*K_{\sigma_{1}}(x)}\biggr)\biggr(\sqrt{P_{n}*K_{\sigma_{0}}(x)}-\sqrt{p_{G_{0},f_{0}}*K_{\sigma_{0}}(x)}\biggr)}\textrm{d}x \geq  \label{eqn:misspecified_equation_one} \\
& & \hspace{-1em} \int{\sqrt{p_{G_{0},f_{0}}*K_{\sigma_{0}}(x)}\sqrt{p_{G_{*},f}*K_{\sigma_{1}}(x)}}\textrm{d}x -\int{\sqrt{p_{G_{0},f_{0}}*K_{\sigma_{0}}(x)}\sqrt{p_{\widehat{G}_{n},f}*K_{\sigma_{1}}(x)}}\textrm{d}x := B. \nonumber
\end{eqnarray}
By means of the key inequality in Lemma \ref{lemma:key_inequality_misspecified_setting}, we have
\begin{eqnarray}
B & \geq & \int{\sqrt{p_{\widehat{G}_{n},f}*K_{\sigma_{1}}(x)}\sqrt{\dfrac{p_{G_{0},f_{0}}*K_{\sigma_{0}}(x)}{p_{G_{*},f}*K_{\sigma_{1}}(x)}}}\textrm{d}x-\int{\sqrt{p_{G_{0},f_{0}}*K_{\sigma_{0}}(x)}\sqrt{p_{\widehat{G}_{n},f}*K_{\sigma_{1}}(x)}}\textrm{d}x \nonumber \\
& = & 2\biggr(h^{*}(p_{\widehat{G}_{n},f}*K_{\sigma_{1}},p_{G_{*},f}*K_{\sigma_{1}})\biggr)^{2} - B. \nonumber
\end{eqnarray}
It implies that $B \geq \biggr(h^{*}(p_{\widehat{G}_{n},f}*K_{\sigma_{1}},p_{G_{*},f}*K_{\sigma_{1}})\biggr)^{2}$. Plugging this inequality to \eqref{eqn:misspecified_equation_one} leads to
\begin{eqnarray}
C & := & \int{\biggr(\sqrt{p_{\widehat{G}_{n},f}*K_{\sigma_{1}}(x)}-\sqrt{p_{G_{*},f}*K_{\sigma_{1}}(x)}\biggr)\biggr(\sqrt{P_{n}*K_{\sigma_{0}}(x)}-\sqrt{p_{G_{0},f_{0}}*K_{\sigma_{0}}(x)}\biggr)}\textrm{d}x \nonumber \\
& \geq & \biggr(h^{*}(p_{\widehat{G}_{n},f}*K_{\sigma_{1}},p_{G_{*},f}*K_{\sigma_{1}})\biggr)^{2}. \label{eqn:misspecified_equation_second}
\end{eqnarray}
For the left hand side (LHS) of \eqref{eqn:misspecified_equation_second}, we have the following inequality
\begin{eqnarray}
C  & \leq &  \biggr\|(p_{\widehat{G}_{n},f}*K_{\sigma_{1}})^{1/4}-(p_{G_{*},f}*K_{\sigma_{1}})^{1/4}\biggr\|_{\infty} \int \biggr((p_{\widehat{G}_{n},f}*K_{\sigma_{1}}(x))^{1/4} \nonumber \\
&&\hspace{7 em}+(p_{G_{*},f}*K_{\sigma_{1}}(x))^{1/4}\biggr)\biggr|\sqrt{P_{n}*K_{\sigma_{0}}(x)}-\sqrt{p_{G_{0},f_{0}}*K_{\sigma_{0}}(x)}\biggr|\textrm{d}x \nonumber \\
& \leq &  \biggr\|(p_{\widehat{G}_{n},f}*K_{\sigma_{1}})^{1/4}-(p_{G_{*},f}*K_{\sigma_{1}})^{1/4}\biggr\|_{\infty} \biggr\|(p_{\widehat{G}_{n},f}*K_{\sigma_{1}})^{1/4}+(p_{G_{*},f}*K_{\sigma_{1}})^{1/4}\biggr\|_{2} \nonumber \\
& & \hspace{18 em} \times \sqrt{2}h(P_{n}*K_{\sigma_{0}},p_{G_{0},f_{0}}*K_{\sigma_{0}}) \label{eqn:misspecified_equation_second_one}
\end{eqnarray}
where the last inequality is due to Holder's inequality. Now, our next argument will be divided into two small key steps.
\paragraph{Step 2.1} With assumption (M.3), we will show that
\begin{eqnarray}
D = \biggr\|(p_{G,f}*K_{\sigma_{1}})^{1/4}-(p_{G_{*},f}*K_{\sigma_{1}})^{1/4}\biggr\|_{\infty} \leq M_{3}(\sigma_{1})W_{1}(G,G_{0}) \label{eqn:misspecified_equation_second_second}
\end{eqnarray}
for any $G \in \mathcal{O}_{k_{*}}$ where $M_{3}(\sigma_{1})$ is some positive constant.

In fact, denote $G=\mathop {\sum }\limits_{i=1}^{k}{p_{i}\delta_{\theta_{i}}}$ where $k 
\leq k_{*}$ and $G_{*}=\mathop {\sum }\limits_{i=1}^{k_{*}}{p_{i}^{*}
\delta_{\theta_{i}^{*}}}$. Using the same proof argument as that of  
\eqref{eqn:algorithm_2_simple_wellspecified_second} in the proof of Proposition 
\ref{proposition:sufficient_condition_wellspecified_setting}, there exists a positive number 
$\epsilon_{0}$ depending only $G_{*}$ such that as long as $W_{1}(G,G_{*}) \leq 
\epsilon_{0}$, $G$ will have exactly $k_{*}$ components, i.e., $k=k_{*}$. Additionally, up 
to the relabelling of the components of $G$, we also obtain that $|p_{i}-p_{i}^{*}| \leq 
c_{0}W_{1}(G,G_{*})$ where $c_{0}$ is some positive constant depending only on 
$G_{*}$. Therefore, by choosing $G$ such that $W_{1}(G,G_{*}) \leq C_{0}=\min {\left\{\epsilon_{0}, \min \limits_{1 \leq i \leq k_{*}}{\dfrac{p_{i}^{*}}{2c_{0}}}\right\}}$, we 
achieve $|p_{i}-p_{i}^{*}| \leq \min \limits_{1 \leq i \leq k_{*}}{p_{i}^{*}}/2$. Hence, 
$p_{i} \geq \min \limits_{1 \leq i \leq k_{*}}{p_{i}^{*}}/2$ for all $1 \leq i \leq k_{*}$. 
Under this setting of $G$, for any coupling $\vec{q}$ of $\vec{p}=(p_{1},\ldots,p_{k})$ 
and $\vec{p^{*}}=(p_{1}^{*},\ldots,p_{k_{*}}^{*})$, by means of triangle inequality we 
obtain
\begin{eqnarray}
D & = & \biggr\|\dfrac{p_{G,f}*K_{\sigma_{1}}-p_{G_{*},f}*K_{\sigma_{1}}}{\left\{(p_{G,f}*K_{\sigma_{1}})^{1/4}+(p_{G_{*},f}*K_{\sigma_{1}})^{1/4}\right\}\left\{(p_{G,f}*K_{\sigma_{1}})^{1/2}+(p_{G_{*},f}*K_{\sigma_{1}})^{1/2}\right\}}\biggr\|_{\infty} \nonumber \\
& & \hspace{-3 em} \leq  \sum \limits_{i,j} {q_{ij}\biggr\|\dfrac{f*K_{\sigma_{1}}(x|\theta_{i})-f*K_{\sigma_{1}}(x|\theta_{j}^{*})}{\left\{(p_{G,f}*K_{\sigma_{1}})^{1/4}+(p_{G_{*},f}*K_{\sigma_{1}})^{1/4}\right\}\left\{(p_{G,f}*K_{\sigma_{1}})^{1/2}+(p_{G_{*},f}*K_{\sigma_{1}})^{1/2}\right\}}\biggr\|_{\infty}}. \nonumber
\end{eqnarray}
where the ranges of $i,j$ in the above sum satisfy $1 \leq i, j \leq k_{*}$. It is clear that for any $
\alpha \in \left\{1/2,1/4\right\}$
\begin{eqnarray}
& & (p_{G,f}*K_{\sigma_{1}}(x))^{\alpha}+(p_{G_{*},f}*K_{\sigma_{1}}(x))^{\alpha} \nonumber \\
& & \hspace{10em} > \min {\left\{p_{i}^{\alpha},(p_{j}^{*})^{\alpha}\right\}}\left\{(f*K_{\sigma_{1}}(x|\theta_{i}))^{\alpha}+(f*K_{\sigma_{1}}(x|\theta_{j}^{*}))^{\alpha}\right\} \nonumber \\
& & \hspace{10em} > \min \limits_{1 \leq i \leq k_{*}}{\left(\dfrac{p_{i}^{*}}{2}\right)^{\alpha}}\left\{f*K_{\sigma_{1}}(x|\theta_{i}))^{\alpha}+(f*K_{\sigma_{1}}(x|\theta_{j}^{*}))^{\alpha}\right\}. \nonumber
\end{eqnarray}
Therefore, we eventually achieve that
\begin{eqnarray}
D \lesssim \sum \limits_{i,j}{q_{ij}\biggr\|(f*K_{\sigma_{1}}(x|\theta_{i}))^{1/4}-(f*K_{\sigma_{1}}(x|\theta_{j}^{*}))^{1/4}\biggr\|_{\infty}}. \nonumber
\end{eqnarray}
Now, due to assumption (M.3) and mean value theorem, we achieve for any $x \in \mathcal{X}$ that
\begin{eqnarray}
\left|f*K_{\sigma_{1}}(x|\theta_{i}))^{1/4}-(f*K_{\sigma_{1}}(x|\theta_{j}^{*}))^{1/4}\right| \leq M_{2}(\sigma_{1})||\theta_{i}-\theta_{j}^{*}||. \nonumber
\end{eqnarray}
Thus, for any coupling $\vec{q}$ of $\vec{p}$ and $\vec{p^{*}}$
\begin{eqnarray}
D \lesssim \sum \limits_{i,j}{q_{ij}||\theta_{i}-\theta_{j}^{*}||}. \nonumber
\end{eqnarray}
As a consequence, we eventually have
\begin{eqnarray}
D \lesssim \inf \limits_{\vec{q} \in \mathcal{Q}(\vec{p},\vec{p^{*}})}{\sum \limits_{i,j}{q_{ij}||\theta_{i}-\theta_{j}^{*}||}} = W_{1}(G,G_{*}) \nonumber
\end{eqnarray}
for any $G \in \mathcal{O}_{k_{*}}$ such that $W_{1}(G,G_{*}) \leq C_{0}$. Now, for any $G \in \mathcal{O}_{k_{*}}$ such that $W_{1}(G,G_{*}) > C_{0}$, as $D$ is bounded, it is clear that $D \lesssim W_{1}(G,G_{*})$. In sum, we achieve inequality \eqref{eqn:misspecified_equation_second_second}.
\paragraph{Step 2.2} Due to assumption (M.2), we also can quickly verify that
\begin{eqnarray}
\biggr\|(p_{\widehat{G}_{n},f}*K_{\sigma_{1}})^{1/4}+(p_{G_{*},f}*K_{\sigma_{1}})^{1/4}\biggr\|_{2} \leq 2\sqrt{k_{*}M_{1}(\sigma_{1})}. \label{eqn:misspecified_equation_second_third}
\end{eqnarray}
Combining \eqref{eqn:misspecified_equation_second_one}, \eqref{eqn:misspecified_equation_second_second}, \eqref{eqn:misspecified_equation_second_third}, we ultimately achieve that
\begin{eqnarray}
\biggr(h^{*}(p_{\widehat{G}_{n},f}*K_{\sigma_{1}},p_{G_{*},f}*K_{\sigma_{1}})\biggr)^{2} \leq  M(\sigma_{1}) W_{1}(\widehat{G}_{n},G_{*}) h(P_{n}*K_{\sigma_{0}},p_{G_{0},f_{0}}*K_{\sigma_{0}}) \nonumber
\end{eqnarray}
where $M(\sigma_{1})$ is some positive constant. Due to assumption (M.1), from the result of Lemma \ref{lemma:lower_bound_modified_Hellinger_Wasserstein} and the definition of $C_{*,1}(\sigma_{1})$, we have
\begin{eqnarray}
h^{*}(p_{\widehat{G}_{n},f}*K_{\sigma_{1}},p_{G_{*},f}*K_{\sigma_{1}}) \gtrsim C_{*,1}(\sigma_{1})W_{1}(\widehat{G}_{n},G_{*}). \nonumber
\end{eqnarray}
Combining the above results with the bound $h(P_{n}*K_{\sigma_{0}},p_{G_{0},f_{0}}*K_{\sigma_{0}})=O_{p}\biggr(\sqrt{\dfrac{\Psi(G_{0},\sigma_{0})}{n}}\biggr)$ from Step 2 in the proof of Theorem \ref{theorem:convergence_rate_mixing_measure}, we quickly obtain the conclusion of the theorem.
\bibliography{Nhat,NPB,Nguyen}
\newpage

\appendix
\section*{Appendix A}
In this Appendix, we provide the proofs of several key results in Section \ref{Section:minimum_Hellinger} and Section \ref{Section:Another_approach}.
\label{appendix_A}
\paragraph{PROOF OF LEMMA \ref{lemma:first_order_convolution}}
The proof of this lemma is a straightforward application of the Fourier transform. In fact, for any finite $k$ different elements
$\theta_{1}, \ldots , \theta_{k} \in \Theta$,
assume that we have $\alpha_{i} \in \mathbb{R},\beta_{i} \in \mathbb{R}^{d_{1}}$ (for all $i=1,\ldots,k$)
such that for almost all $x$
\begin{eqnarray}
\mathop {\sum }\limits_{i=1}^{k}{\alpha_{i}f_{0}*K_{\sigma_{0}}(x|\theta_{i})+\beta_{i}^{T}\dfrac{\partial{f_{0}*K_{\sigma_{0}}}}{\partial{\theta}}(x|\theta_{i})}=0, \nonumber
\end{eqnarray}
By means of the Fourier transformation on both sides of the above equation, we obtain for all $t \in \mathbb{R}^{d}$ that
\begin{eqnarray}
\widehat{K_{\sigma_{0}}}(t)\biggr(\sum \limits_{i=1}^{k}{\alpha_{i}\widehat{f_{0}}(t|\theta_{i})+\beta_{i}^{T}\dfrac{\partial \widehat{f_{0}}}{\partial{\theta}}(t|\theta_{i})}\biggr)=0. \nonumber
\end{eqnarray}
Since $\widehat{K_{\sigma_{0}}}(t)=\widehat{K}(\sigma_{0} t) \neq 0$ for almost all $t \in \mathbb{R}^{d}$ and $f$ is identifiable in the first order, we obtain that $\alpha_{i}=0, \beta_{i}=\vec{0} \in \mathbb{R}^{d_{1}}$ for all $1 \leq i \leq k$. We achieve the conclusion of this lemma.

\paragraph{PROOF OF LEMMA \ref{lemma:lower_bound_modified_Hellinger_Wasserstein}}
We denote the following weighted version of the total variation distance
\begin{eqnarray}
V^{*}(p_{G_{1},f}*K_{\sigma_{1}},p_{G_{2},f}*K_{\sigma_{1}})=\dfrac{1}{2}\int \left| p_{G_{1},f}*K_{\sigma_{1}}(x)-p_{G_{2},f}*K_{\sigma_{1}}(x)\right| \nonumber \\
\times \biggr(\dfrac{p_{G_{0},f_{0}}*K_{\sigma_{0}}(x)}{p_{G_{*},f}*K_{\sigma_{1}}(x)}\biggr)^{1/4}\textrm{d}x. \nonumber
\end{eqnarray}
for any two mixing measures $G_{1},G_{2} \in \overline{\mathcal{G}}$. By Holder's inequality, we have
\begin{eqnarray}
& & \hspace{-2.5 em}V^{*}(p_{G,f}*K_{\sigma_{1}},p_{G_{*},f}*K_{\sigma_{1}})  \leq  \dfrac{1}{\sqrt{2}}h^{*}(p_{G,f}*K_{\sigma_{1}},p_{G_{*},f}*K_{\sigma_{1}}) \nonumber \\
& & \hspace{10 em} \times \biggr(\int{\biggr(\sqrt{p_{G,f}*K_{\sigma_{1}}(x)}+\sqrt{p_{G_{*},f}*K_{\sigma_{1}}(x)}\biggr)^{2}}\textrm{d}x\biggr)^{1/2} \nonumber \\
& & \hspace{10 em}  \leq \sqrt{2}h^{*}(p_{G,f}*K_{\sigma_{1}},p_{G_{*},f}*K_{\sigma_{1}}). \label{eqn:misspecify_kernel_inequality_one}
\end{eqnarray}
Therefore, in order to obtain the conclusion of the lemma it suffices to demonstrate that
\begin{eqnarray}
\mathop {\inf }\limits_{G \in \mathcal{O}_{k_{*}}}{V^{*}(p_{G,f}*K_{\sigma_{1}},p_{G_{*},f}*K_{\sigma_{1}})/W_{1}(G,G_{*})} > 0. \label{eqn:misspecify_kernel_inequality_two}
\end{eqnarray}
Firstly, we will show that
\begin{eqnarray}
\mathop {\lim }\limits_{\epsilon \to 0}{\mathop {\inf }\limits_{G \in \mathcal{O}_{k_{*}}}{\biggr\{\dfrac{V^{*}(p_{G,f}*K_{\sigma_{1}},p_{G_{*},f}*K_{\sigma_{1}})}{W_{1}(G,G_{*})}: \ W_{1}(G,G_{*}) \leq \epsilon \biggr\}}} > 0. \nonumber
\end{eqnarray}
Assume that the above inequality does not hold. There exists a sequence $G_{n} \in \mathcal{O}_{k_{*}}$ such that $W_{1}(G_{n},G_{*}) \to 0$ and
$V^{*}(p_{G_{n},f}*K_{\sigma_{1}},p_{G_{*},f}*K_{\sigma_{1}})/W_{1}(G_{n},G_{*}) \to 0$. By means of Fatou's lemma, we obtain
\begin{eqnarray}
0 & = & \liminf \limits_{n \to \infty} \dfrac{V^{*}(p_{G_{n},f}*K_{\sigma_{1}},p_{G_{*},f}*K_{\sigma_{1}})}{W_{1}(G_{n},G_{*})} \nonumber \\
& \geq & \dfrac{1}{2}\int \liminf \limits_{n \to \infty}\dfrac{\left|p_{G_{n},f}*K_{\sigma_{1}}(x)-p_{G_{*},f}*K_{\sigma_{1}}(x)\right|\biggr(\dfrac{p_{G_{0},f_{0}}*K_{\sigma_{0}}(x)}{p_{G_{*},f}*K_{\sigma_{1}}(x)}\biggr)^{1/4}}{W_{1}(G_{n},G_{*})} \textrm{d}x. \nonumber
\end{eqnarray}
Therefore, for almost surely $x \in \mathcal{X}$, we have
\begin{eqnarray}
\liminf \limits_{n \to \infty}\dfrac{\left|p_{G_{n},f}*K_{\sigma_{1}}(x)-p_{G_{*},f}*K_{\sigma_{1}}(x)\right|\biggr(\dfrac{p_{G_{0},f_{0}}*K_{\sigma_{0}}(x)}{p_{G_{*},f}*K_{\sigma_{1}}(x)}\biggr)^{1/4}}{W_{1}(G_{n},G_{*})}=0. \label{eqn:lemma_misspecified_kernel_equality_one}
\end{eqnarray}
Since $W_{1}(G_{n},G_{*}) \to 0$ and $G_{n} \in \mathcal{O}_{k_{*}}$, we can find a 
subsequence of $k_{n}$ such that $k_{n}=k_{*}$. Without loss of generality, we replace 
that subsequence of $k_{n}$ by its whole sequence. Then, $G_{n}$ will have exactly 
$k_{*}$ components for all $n \geq 1$. From here, by using the same argument as that in 
the proof of Theorem 3.1 in \cite{Ho-Nguyen-EJS-16}, equality 
\eqref{eqn:lemma_misspecified_kernel_equality_one} cannot happen -- a contradiction.

Therefore, we can find a positive constant number $\epsilon_{0}$ such that $V^{*}
(p_{G,f}*K_{\sigma_{1}},p_{G_{*},f}*K_{\sigma_{1}}) \gtrsim W_{1}(G,G_{*})$ for any 
$W_{1}(G,G_{*}) \leq \epsilon_{0}$. Now, to obtain the conclusion of 
\eqref{eqn:misspecify_kernel_inequality_two}, we only need to verify that
\begin{eqnarray}
\mathop {\inf }\limits_{G \in \mathcal{O}_{k_{*}}:W_{1}(G,G_{*})>\epsilon_{0}}{V^{*}(p_{G,f}*K_{\sigma_{1}},p_{G_{*},f}*K_{\sigma_{1}})/W_{1}(G,G_{*})} > 0. \nonumber
\end{eqnarray}
In fact, if the above statement does not hold, we can find a sequence $G_{n}' \in 
\mathcal{O}_{k_{*}}$ such that $W_{1}(G_{n},G_{*})>\epsilon_{0}$ and $V^{*}
(p_{G_{n}',f}*K_{\sigma_{1}},p_{G_{*},f}*K_{\sigma_{1}})/W_{1}(G_{n}',G_{*}) \to 0$. 
Since $\Theta$ is a closed bounded set, we can find $G' \in \mathcal{O}_{k_{*}}$ such 
that a subsequence of $G_{n}'$ satisfies $W_{1}(G_{n}',G') \to 0$ and $W_{1}
(G',G_{*})>\epsilon_{0}$. Without loss of generality, we replace that subsequence of 
$G_{n}'$ by its whole sequence. Therefore, $V^{*}(p_{G_{n}',f}
*K_{\sigma_{1}},p_{G_{*},f}*K_{\sigma_{1}}) \to 0$ as $n \to \infty$. Since $W_{1}
(G_{n}',G') \to 0$, due to the first order uniform Liptschitz property of $f*K_{\sigma_{1}}
$ we obtain $p_{G_{n}',f}*K_{\sigma_{1}}(x) \to p_{G',f}*K_{\sigma_{1}}(x)$ for 
almost all $x \in \mathcal{X}$ when $n \to \infty$. Now, by means of Fatou's lemma
\begin{eqnarray}
0 & = &\lim \limits_{n \to \infty} V^{*}(p_{G_{n}',f}*K_{\sigma_{1}},p_{G_{*},f}*K_{\sigma_{1}}) \nonumber \\
& \geq & \int \liminf \limits_{n \to \infty} \left|p_{G_{n}',f}*K_{\sigma_{1}}(x)-p_{G_{*},f}*K_{\sigma_{1}}(x)\right|\biggr(\dfrac{p_{G_{0},f_{0}}*K_{\sigma_{0}}(x)}{p_{G_{*},f}*K_{\sigma_{1}}(x)}\biggr)^{1/4} \textrm{d}x \nonumber \\
& = & V^{*}(p_{G',f}*K_{\sigma_{1}},p_{G_{*},f}*K_{\sigma_{1}}), \nonumber
\end{eqnarray}
which only happens when $p_{G',f}*K_{\sigma_{1}}(x)=p_{G_{*},f}*K_{\sigma_{1}}(x)$ 
for almost surely $x$. Due to the identifiability of $f*K_{\sigma_{1}}$, the former equality 
implies that $G' \equiv G_{*}$, which is a contradiction to the condition that $W_{1}
(G',G_{*})>\epsilon_{0}$. As a consequence, we achieve the conclusion of the lemma.
\paragraph{PROOF OF PROPOSITION \ref{proposition:consistency_components_well_specified}}
For the simplicity of presentation we implicitly denote $\widehat{G}_{n,m}=\widehat{G}_{n,m}(\sigma_{0,n})$ and $G_{0,m} = G_{0,m}(\sigma_{0,n}) = \mathop {\arg \min}\limits_{G \in \mathcal{O}_{m}}{h(p_{G,f_{0}} * K_{\sigma_{0,n}}, p_{G_{0},f_{0}}*K_{\sigma_{0,n}})}$ for any element $\sigma_{0,n}>0$, i.e., both $\widehat{G}_{n,m}$ and $G_{0,m}$ strictly depend on $\sigma_{0,n}$ and will vary as long as $\sigma_{0,n} \to 0$. Now, as $n \to \infty$, we will prove for almost surely that
\begin{eqnarray}
h(p_{\widehat{G}_{n,m},f_{0}}*K_{\sigma_{0,n}},P_{n}*K_{\sigma_{0,n}}) - h(p_{\widehat{G}_{n,m+1},f_{0}}*K_{\sigma_{0,n}},P_{n}*K_{\sigma_{0,n}}) \to d_{m}',\label{eqn:changed_bandwidth_zero}
\end{eqnarray}
where $d_{m}'=h(p_{\widetilde{G}_{0,m},f_{0}},p_{G_{0},f_{0}})-h(p_{\widetilde{G}_{0,m+1},f_{0}},p_{G_{0},f_{0}})$ where $\widetilde{G}_{0,m}=\mathop {\arg \min}\limits_{G \in \mathcal{O}_{m}}{h(p_{G,f_{0}}, p_{G_{0},f_{0}})}$. To achieve this result, we start with the following lemma
\begin{lemma}\label{lemma:convergence_bandwidth}
For any sequence $G_{n}$ and $\sigma_{n} \to 0$, we have as $n \to \infty$ that
\begin{eqnarray}
h(p_{G_{n},f_{0}}*K_{\sigma_{n}},p_{G_{n},f_{0}}) \to 0. \nonumber
\end{eqnarray}
\end{lemma}
The proof of this lemma is deferred to Appendix B. Now, applying the result of Lemma \ref{lemma:convergence_bandwidth} to the sequences $G_{0,m}$ and $\sigma_{0,n}$, we have
\begin{eqnarray}
\mathop {\lim }\limits_{n \to \infty}{h(p_{G_{0,m},f_{0}}*K_{\sigma_{0,n}},p_{G_{0},f_{0}}*K_{\sigma_{0,n}})} & = &\mathop {\lim }\limits_{n \to \infty}{h(p_{G_{0,m},f_{0}},p_{G_{0},f_{0}})} \nonumber \\
& \geq & h(p_{\widetilde{G}_{0,m},f_{0}},p_{G_{0},f_{0}}). \label{eqn:changed_bandwidth_one}
\end{eqnarray}
On the other hand, from the definition of $G_{0,m}$, we have
\begin{eqnarray}
h(p_{G_{0,m},f_{0}}*K_{\sigma_{0,n}},p_{G_{0},f_{0}}*K_{\sigma_{0,n}}) \leq h(p_{\widetilde{G}_{0,m},f_{0}}*K_{\sigma_{0,n}},p_{G_{0},f_{0}}*K_{\sigma_{0,n}}). \nonumber
\end{eqnarray}
Therefore,
\begin{eqnarray}
\mathop {\lim }\limits_{n \to \infty}{h(p_{G_{0,m},f_{0}}*K_{\sigma_{0,n}},p_{G_{0},f_{0}}*K_{\sigma_{0,n}})} & \leq & \mathop {\lim }\limits_{n \to \infty}{h(p_{\widetilde{G}_{0,m},f_{0}}*K_{\sigma_{0,n}},p_{G_{0},f_{0}}*K_{\sigma_{0,n}})}\nonumber \\
& = & h(p_{\widetilde{G}_{0,m},f_{0}},p_{G_{0},f_{0}}). \label{eqn:changed_bandwith_second}
\end{eqnarray}
Combining the results from \eqref{eqn:changed_bandwidth_one} and \eqref{eqn:changed_bandwith_second}, we have
\begin{eqnarray}
\mathop {\lim }\limits_{n \to \infty}{h(p_{G_{0,m},f_{0}}*K_{\sigma_{0,n}},p_{G_{0},f_{0}}*K_{\sigma_{0,n}})}=h(p_{\widetilde{G}_{0,m},f_{0}},p_{G_{0},f_{0}}). \label{eqn:changed_bandwith_third}
\end{eqnarray}
Now, we will demonstrate that
\begin{eqnarray}
\mathop {\lim }\limits_{n \to \infty}{h(p_{\widehat{G}_{n,m},f_{0}}*K_{\sigma_{0,n}},P_{n}*K_{\sigma_{0,n}})}=\mathop {\lim }\limits_{n \to \infty}{h(p_{G_{0,m},f_{0}}*K_{\sigma_{0,n}},p_{G_{0},f_{0}}*K_{\sigma_{0,n}})}. \label{eqn:changed_bandwith_fourth}
\end{eqnarray}
In fact, from the definition of $\widehat{G}_{n,m}$ we quickly obtain that
\begin{eqnarray}
\mathop {\lim }\limits_{n \to \infty}{h(p_{\widehat{G}_{n,m},f_{0}}*K_{\sigma_{0,n}},P_{n}*K_{\sigma_{0,n}})} & \leq & \mathop {\lim }\limits_{n \to \infty}{h(p_{G_{0,m},f_{0}}*K_{\sigma_{0,n}},P_{n}*K_{\sigma_{0,n}})} \nonumber \\
& = & \mathop {\lim }\limits_{n \to \infty}{h(p_{G_{0,m},f_{0}}*K_{\sigma_{0,n}},p_{G_{0},f_{0}}*K_{\sigma_{0,n}})} \label{eqn:changed_bandwith_fifth}
\end{eqnarray}
where the last equality is due to the fact that $h(P_{n}*K_{\sigma_{0,n}},p_{G_{0},f_{0}}) \to 0$ almost surely as $n \to \infty$, $\sigma_{0,n} \to 0$ and $n\sigma_{0,n}^{d} \to \infty$. On the other hand, from the formulation of $G_{0,m}$ we have
\begin{eqnarray}
\mathop {\lim }\limits_{n \to \infty}{h(p_{G_{0,m},f_{0}}*K_{\sigma_{0,n}},p_{G_{0},f_{0}}*K_{\sigma_{0,n}})} & \leq & \mathop {\lim }\limits_{n \to \infty}{h(p_{\widehat{G}_{n,m},f_{0}}*K_{\sigma_{0,n}},p_{G_{0},f_{0}}*K_{\sigma_{0,n}})} \nonumber \\
& = & \mathop {\lim }\limits_{n \to \infty}{h(p_{\widehat{G}_{n,m},f_{0}}*K_{\sigma_{0,n}},P_{n}*K_{\sigma_{0,n}})} \label{eqn:changed_bandwith_sixth}
\end{eqnarray}
Combining \eqref{eqn:changed_bandwith_fifth} and \eqref{eqn:changed_bandwith_sixth}, we obtain equality \eqref{eqn:changed_bandwith_fourth}. Now, the combination of \eqref{eqn:changed_bandwith_third} and \eqref{eqn:changed_bandwith_fourth} leads to
\begin{eqnarray}
\mathop {\lim }\limits_{n \to \infty}{h(p_{\widehat{G}_{n,m},f_{0}}*K_{\sigma_{0,n}},P_{n}*K_{\sigma_{0,n}})}=h(p_{\widetilde{G}_{0,m},f_{0}},p_{G_{0},f_{0}}). \nonumber
\end{eqnarray}
Therefore, we obtain the conclusion of \eqref{eqn:changed_bandwidth_zero}. From here, by using the same argument as Step 1 of Theorem \ref{theorem:convergence_rate_mixing_measure}, we ultimately get $d_{m}'=0$ as $m \geq k_{0}$ and $d_{m}'>0$ as $m<k_{0}$. As a consequence, $\widehat{m}_{n} \to k_{0}$ almost surely as $n \to \infty$. The conclusion of the proposition follows.
\paragraph{PROOF OF LEMMA \ref{lemma:key_inequality_misspecified_setting}}
The proof proceeds by using the idea from Leroux's argument \citep{Leroux-1992}. In fact, from the definition of $G_{*}$, we have $h(p_{G_{*},f}*K_{\sigma_{1}},p_{G_{0},f_{0}}*K_{\sigma_{0}}) \leq h(p_{G,f}*K_{\sigma_{1}},p_{G_{0},f_{0}}*K_{\sigma_{0}})$ for any $G \in \overline{\mathcal{G}}$. Now, for any $\theta \in \Theta$, by choosing $G=(1-\epsilon)G_{*}+\epsilon\delta_{\theta}$ and letting $\epsilon \to 0$ as in Step 1 of the proof of Theorem \ref{theorem:convergence_rate_mixing_measure}, we eventually obtain
\eqsplit{
&\int {(p_{G_{0},f_{0}}*K_{\sigma_{0}}(x))^{1/2}(p_{G_{*},f}*K_{\sigma_{1}}(x))^{1/2}}\textrm{d}x
\\
&\hspace{3em}\geq\int {(p_{G_{0},f_{0}}*K_{\sigma_{0}}(x))^{1/2}f*K_{\sigma_{1}}(x|\theta)(p_{G_{*},f}*K_{\sigma_{1}}(x))^{-1/2}}\textrm{d}x. }
By choosing $\theta \in \text{supp}(G_{*})$ and summing over all of these components, we readily obtain inequality \eqref{eqn:property_modified_hellinger_distance}, which concludes the result of the lemma.
\paragraph{PROOF OF LEMMA \ref{lemma:key_property_non_unique}}
By means of Holder inequality, we obtain
\begin{eqnarray}
\biggr(\int {\sqrt{p_{G_{1,*},f}*K_{\sigma_{1}}(x)}
\sqrt{p_{G_{0},f_{0}}*K_{\sigma_{0}}(x)}}\textrm{d}x\biggr)^{2} \leq \int {p_{G_{1,*},f}*K_{\sigma_{1}}(x)\sqrt{\dfrac{p_{G_{0},f_{0}}*K_{\sigma_{0}}(x)}{p_{G_{2,*},f}*K_{\sigma_{1}}(x)}}}\textrm{d}x \nonumber \\
\hspace{-2em} \times \int {\sqrt{p_{G_{2,*},f}*K_{\sigma_{1}}(x)}
\sqrt{p_{G_{0},f_{0}}*K_{\sigma_{0}}(x)}}\textrm{d}x. \nonumber
\end{eqnarray}
From the definition of $G_{1,*}$ and $G_{2,*}$, we have
\begin{eqnarray}
\int {\sqrt{p_{G_{1,*},f}*K_{\sigma_{1}}(x)}
\sqrt{p_{G_{0},f_{0}}*K_{\sigma_{0}}(x)}}\textrm{d}x=\int {\sqrt{p_{G_{2,*},f}*K_{\sigma_{1}}(x)}
\sqrt{p_{G_{0},f_{0}}*K_{\sigma_{0}}(x)}}\textrm{d}x. \nonumber
\end{eqnarray}
Therefore, the above inequality along with innequality \eqref{eqn:property_modified_hellinger_distance} in Lemma \ref{lemma:key_inequality_misspecified_setting} lead to
\begin{eqnarray}
\int {p_{G_{1,*},f}*K_{\sigma_{1}}(x)\sqrt{\dfrac{p_{G_{0},f_{0}}*K_{\sigma_{0}}(x)}{p_{G_{2,*},f}*K_{\sigma_{1}}(x)}}}\textrm{d}x = \int {\sqrt{p_{G_{2,*},f}*K_{\sigma_{1}}(x)}
\sqrt{p_{G_{0},f_{0}}*K_{\sigma_{0}}(x)}}\textrm{d}x. \nonumber
\end{eqnarray}
It eventually implies that
\begin{eqnarray}
\int {\biggr(\sqrt{p_{G_{1,*},f}*K_{\sigma_{1}}(x)}-\sqrt{p_{G_{2,*},f}*K_{\sigma_{1}}(x)}\biggr)^{2}\sqrt{\dfrac{p_{G_{0},f_{0}}*K_{\sigma_{0}}(x)}{p_{G_{2,*},f}*K_{\sigma_{1}}(x)}}}\textrm{d}x=0. \nonumber
\end{eqnarray}
Therefore, $p_{G_{1,*},f}*K_{\sigma_{1}}(x) = p_{G_{2,*},f}*K_{\sigma_{1}}(x)$ almost surely $x \in \mathcal{X}$. As a consequence, we obtain the conclusion of the lemma.
\paragraph{PROOF OF THEOREM \ref{theorem:sufficent_necessary_condition}}
Here, we only provide the proof for part (b) as it is the generalization of part (a). The proof is similar to that in Step 1 of Theorem \ref{theorem:convergence_rate_mixing_measure_misspecified_strong}. In fact, as $n \to \infty$ we have for almost surely that
\begin{eqnarray}
h(p_{\widehat{G}_{n,m},f}*K_{\sigma_{1}},P_{n}*K_{\sigma_{0}}) \to h(p_{G_{*,m},f}*K_{\sigma_{1}},p_{G_{0},f_{0}}*K_{\sigma_{0}})
\end{eqnarray}
where $G_{*,m}=\mathop {\arg \min}\limits_{G \in \mathcal{O}_{m}}{h(p_{G,f}*K_{\sigma_{1}},p_{G_{0},f_{0}}*K_{\sigma_{0}})}$. From the argument of Step 1 in the proof of Theorem \ref{theorem:convergence_rate_mixing_measure_misspecified_strong}, we have
\begin{eqnarray}
h(p_{G_{*,m+1},f}*K_{\sigma_{1}},p_{G_{0},f_{0}}*K_{\sigma_{0}})< h(p_{G_{*,m},f}*K_{\sigma_{1}},p_{G_{0},f_{0}}*K_{\sigma_{0}})
\end{eqnarray}
for any $1 \leq m \leq k_{*}-1$. It implies that $G_{*,m} \in \mathcal{E}_{m}$ for all $1 \leq m \leq k_{*}$. Now,  if we would like to have $\widetilde{m}_{n} \to k_{*}$ as $n \to \infty$, the sufficient and necessary condition is
\begin{eqnarray}
h(p_{G_{*},f}*K_{\sigma_{1}},p_{G_{0},f_{0}}*K_{\sigma_{0}}) \leq \epsilon< h(p_{G_{*,k_{*}-1},f}*K_{\sigma_{1}},p_{G_{0},f_{0}}*K_{\sigma_{0}}), \nonumber
\end{eqnarray}
which is precisely the conclusion of the theorem.
\paragraph{PROOF OF PROPOSITION \ref{proposition:sufficient_condition_wellspecified_setting}}
Using the argument from Step 1 in the proof of Theorem 
\ref{theorem:convergence_rate_mixing_measure}, we obtain that $G_{0,k_{0}-1}$ has 
exactly $k_{0}-1$ elements. Now, since $f_{0}*K_{\sigma_{0}}$ is uniformly Lipschitz up 
to the first order and identifiable, we obtain
\begin{eqnarray}
\mathop {\inf }\limits_{G \in \mathcal{E}_{k_{0}-1}}{h(p_{G,f_{0}}*K_{\sigma_{0}},p_{G_{0},f_{0}}*K_{\sigma_{0}})/W_{1}(G,G_{0})}=C' \nonumber
\end{eqnarray}
where $C'$ is some positive constant depending only on $f_{0}, G_{0}, \Theta$, and $\sigma_{0}$. Therefore, we get
\begin{eqnarray}
h(p_{G_{0,k_{0}-1},f_{0}}*K_{\sigma_{0}},p_{G_{0},f_{0}}*K_{\sigma_{0}}) \geq C' W_{1}(G_{0},G_{0,k_{0}-1}) \geq C' \mathop {\inf }\limits_{G \in \mathcal{E}_{k_{0}-1}}{W_{1}(G,G_{0})}. \label{eqn:algorithm_2_simple_wellspecified_first}
\end{eqnarray}
Now, for any $G=\sum \limits_{i=1}^{k_{0}-1}{p_{i}\delta_{\theta_{i}}} \in \mathcal{E}_{k_0-1}$, we can find the index $j^* \in [1,k_{0}]$ such that
\begin{eqnarray}
\|\theta_{i}-\theta_{j^*}^{0}\| \geq \min \limits_{1 \leq j \neq j^* \leq k_{0}}{||\theta_{i}-\theta_{j}^{0}||} \nonumber
\end{eqnarray}
for any $1 \leq i \leq k_{0}-1$. Therefore, we obtain
\begin{eqnarray}
2||\theta_{i}-\theta_{j^*}^{0}||  \geq  ||\theta_{i}-\theta_{j^{*}}^{0}||+\min \limits_{1 \leq j \neq j^* \leq k_{0}}{||\theta_{i}-\theta_{j}^{0}||} \geq \mathop {\min }\limits_{1 \leq u \neq v \leq k_{0}}{||\theta_{u}^{0}-\theta_{v}^{0}||} \nonumber
\end{eqnarray}
for any $1 \leq i \leq k_{0}-1$. From the definition of $W_{1}(G,G_{0})$, we can find the 
optimal coupling $\vec{q} \in \mathcal{Q}(\vec{p},\vec{p^0})$ such that $W_{1}
(G,G_{0})  = \mathop {\sum }\limits_{i,j} {q_{ij}\|\theta_{i}-\theta_{j}^0\|}$. Hence, we get
\begin{eqnarray}
W_{1}(G,G_{0}) & \geq & \sum \limits_{i=1}^{k_{0}}{q_{ij^*}\|\theta_{i}-\theta_{j^*}^0\|} \geq  p_{j^*}^0 \min \limits_{1 \leq i \leq k_{0}-1}{\|\theta_{i}-\theta_{j^*}^0\|} \nonumber \\
& \geq & \biggr(\mathop {\min }\limits_{1 \leq i \leq k_{0}}{p_{i}^{0}} \times \mathop {\min }\limits_{1 \leq i \neq j \leq k_{0}}{||\theta_{i}^{0}-\theta_{j}^{0}||}\biggr)/2 \nonumber
\end{eqnarray}
for all $G \in \mathcal{E}_{k_{0}-1}$. It implies that
\begin{eqnarray}
\mathop {\inf }\limits_{G \in \mathcal{E}_{k_{0}-1}}{W_{1}(G,G_{0})} \geq \biggr(\mathop {\min }\limits_{1 \leq i \leq k_{0}}{p_{i}^{0}} \times \mathop {\min }\limits_{1 \leq i \neq j \leq k_{0}}{||\theta_{i}^{0}-\theta_{j}^{0}||}\biggr)/2. \label{eqn:algorithm_2_simple_wellspecified_second}
\end{eqnarray}
By combining \eqref{eqn:algorithm_2_simple_wellspecified_first} and \eqref{eqn:algorithm_2_simple_wellspecified_second}, if we choose $ \mathop {\min }\limits_{1 \leq i \leq k_{0}}{p_{i}^{0}}\mathop {\min }\limits_{1 \leq i \neq j \leq k_{0}}{||\theta_{i}^{0}-\theta_{j}^{0}||} \geq 2\epsilon/C'$, then $\epsilon< h(p_{G_{0,k_{0}-1},f_{0}}*K_{\sigma_{0}},p_{G_{0},f_{0}}*K_{\sigma_{0}})$. As a consequence, by defining $C=1/C'$ we obtain the conclusion of the lemma.
\paragraph{PROOF OF PROPOSITION \ref{proposition:lower_bound_varying_G0}}
The proof of this proposition is a straightforward combination of Fatou's lemma and the 
argument from Theorem 4.6 in \citep{Jonas-2016}. In fact, for any $\epsilon>0$, as 
$W_{1}(G_{0}^{n},\widetilde{G}_{0}) \to 0$ we can find $M(\epsilon) \in \mathbb{N}$ 
such that $W_{1}(G_{0}^{n},\widetilde{G}_{0})<\epsilon$ for any $n \geq M(\epsilon)$. 
Additionally, as $\mathop {\lim\sup}\limits_{n \to \infty}{k_{0}^{n}}=k$, we can find 
$T(\epsilon) \in \mathbb{N}$ such that $k_{0}^{n} \leq k$ for all $n \geq T(\epsilon)$. 
Denote $N(\epsilon)=\mathop {\max }{\left\{M(\epsilon),T(\epsilon)\right\}}$ for any $
\epsilon>0$. Now, assume that for any $\epsilon>0$, we have
\begin{eqnarray}
\mathop {\inf }\limits_{G \in \mathcal{O}_{k_{0}^{n}}: W_{1}(G,\widetilde{G}_{0})<\epsilon} {h(p_{G,f_{0}}*K_{\sigma_{0}},p_{G_{0}^{n},f_{0}}*K_{\sigma_{0}})/W_{1}^{2k-2\widetilde{k}_{0}+1}(G,G_{0}^{n})}=0. \nonumber
\end{eqnarray}
as long as $n \geq N(\epsilon)$. For each $n \geq N(\epsilon)$, it implies that we have a 
sequence $G_{m}^{n} \in \mathcal{O}_{k_{0}^{n}} \subset \mathcal{O}_{k}$ such that 
$W_{1}(G_{m}^{n},\widetilde{G}_{0})<\epsilon$ for all $m \geq 1$ and
\begin{eqnarray}
h(p_{G_{m}^{n},f_{0}}*K_{\sigma_{0}},p_{G_{0}^{n},f_{0}}*K_{\sigma_{0}})/W_{1}^{2k-2\widetilde{k}_{0}+1}(G_{m}^{n},G_{0}^{n}) \to 0 \nonumber
\end{eqnarray}
as $m \to \infty$. By means of Fatou's lemma, we eventually have
\begin{eqnarray}
\mathop {\lim \inf }\limits_{m \to \infty}{\left(p_{G_{m}^{n},f_{0}}*K_{\sigma_{0}}(x)-p_{G_{0}^{n},f_{0}}*K_{\sigma_{0}}(x)\right)/W_{1}^{2k-2\widetilde{k}_{0}+1}(G_{m}^{n},G_{0}^{n})} \to 0 \nonumber
\end{eqnarray}
almost surely $x \in \mathcal{X}$. However, from the argument of Theorem 4.6 in 
\cite{Jonas-2016}, we can find $\epsilon_{0}>0$ such that for all $G_{m}^{n},G_{0}^{n} 
\in \mathcal{O}_{k}$ where $W_{1}(G_{m}^{n},\widetilde{G}_{0}) \vee W_{1}(G_{0}
^{n},\widetilde{G}_{0})<\epsilon_{0}$, not almost surely $x \in \mathcal{X}$ that $
\left(p_{G_{m}^{n},f_{0}}*K_{\sigma_{0}}(x)-p_{G_{0}^{n},f_{0}}*K_{\sigma_{0}}(x)
\right)/W_{1}^{2k-2\widetilde{k}_{0}+1}(G_{m}^{n},G_{0}^{n}) \to 0$ for each $n \geq 
N(\epsilon_{0})$, which is a contradiction. Therefore, we achieve the conclusion of the 
proposition.
\newpage
\section*{Appendix B}
\label{appendix_B}
This appendix contains remaining proofs of the main results in the paper.
\paragraph{PROOF OF PROPOSITION \ref{proposition:first_order_identifiability}}
A careful investigation of the proof of Theorem 3.1 in \cite{Ho-Nguyen-EJS-16} implies that
\begin{eqnarray}
h(p_{G,f},p_{G_{0},f}) \gtrsim W_{1}(G,G_{0}), \label{eqn:first_order_identifiability}
\end{eqnarray}
for any $G \in \mathcal{O}_{k_{0}}$ such that $W_{1}(G,G_{0})$ is sufficiently small. 
The latter restriction means that the result in \eqref{eqn:first_order_identifiability} is of a 
local nature.  We also would like to extend this lower bound of $h(p_{G,f},p_{G_{0},f})$ 
for any $G \in \mathcal{O}_{k_{0}}$. It appears that the first order Lipschitz continuity of 
$f$ is sufficient to extend \eqref{eqn:first_order_identifiability}
for any $G \in \mathcal{O}_{k_{0}}$. In fact, by the result in \eqref{eqn:first_order_identifiability}, we can find a positive constant $\epsilon_{0}$ such that
\begin{eqnarray}
\mathop {\inf }\limits_{G \in \mathcal{O}_{k_{0}}: W_{1}(G,G_{0}) \leq \epsilon_{0}}{h(p_{G,f},p_{G_{0},f})/W_{1}(G,G_{0})}>0 \nonumber
\end{eqnarray}
Therefore, to extend \eqref{eqn:first_order_identifiability} for any $G \in \mathcal{O}_{k_{0}}$, it is sufficient to demonstrate that
\begin{eqnarray}
\mathop {\inf }\limits_{G \in \mathcal{O}_{k_{0}}: W_{1}(G,G_{0}) > \epsilon_{0}}{h(p_{G,f},p_{G_{0},f})/W_{1}(G,G_{0})}>0 \nonumber
\end{eqnarray}
Assume by the contrary that the above result does not hold. It implies that we can find a 
sequence $G_{n} \in \mathcal{O}_{k_{0}}$ such that $W_{1}(G_{n},G_{0})>\epsilon_{0}
$ and $h(p_{G_{n},f},p_{G_{0},f})/W_{1}(G_{n},G_{0}) \to 0$ as $n \to \infty$. Since $
\Theta$ is a compact set, we can find $G' \in \mathcal{O}_{k_{0}}$ such that a 
subsequence of $G_{n}$ satisfies $W_{1}(G_{n},G') \to 0$ and $W_{1}(G',G_{0})>
\epsilon_{0}$. Without loss of generality, we replace that subsequence by its whole 
sequence. Therefore, $h(p_{G_{n},f},p_{G_{0},f}) \to 0$ as $n \to \infty$. Due to the 
first order Lipschitz continuity of $f$, we obtain $p_{G_{n},f}(x) \to p_{G',f}(x)$ for any 
$x \in \mathcal{X}$ when $n \to \infty$. Now, by means of Fatou's lemma, we have
\begin{eqnarray}
0=\lim \limits_{n \to \infty}{h(p_{G_{n},f},p_{G_{0},f})} \geq \int \mathop {\lim \inf }\limits_{n \to \infty}{\biggr(\sqrt{p_{G_{n},f}(x)}-\sqrt{p_{G_{0},f}(x)}\biggr)^{2}}\textrm{d}x = h(p_{G',f},p_{G_{0},f}). \nonumber
\end{eqnarray}
Since $f$ is identifiable, the above result implies that $G' \equiv G_{0}$, which contradicts 
the assumption that $W_{1}(G',G_{0}) >\epsilon_{0}$. As a consequence, we can extend 
inequality \eqref{eqn:first_order_identifiability} for any $G \in \mathcal{O}_{k_{0}}$ when 
$f$ is uniformly Lipschitz up to the first order.
\paragraph{PROOF OF LEMMA  \ref{lemma:convergence_bandwidth}}
The proof idea of this lemma is similar to that of Theorem 1 in Chapter 2 of 
\citep{Devroye-1985}. However, it is slightly more complex than the proof of this theorem 
as we allow $G_{n}$ to vary when $\sigma_{n}$ vary. Here, we provide the proof of this 
lemma for the completeness. Since the Hellinger distance is upper bound by the total 
variation distance, it is sufficient to demonstrate that $V(p_{G_{n},f_{0}}
*K_{\sigma_{n}},p_{G_{n},f_{0}}) \to 0$ as $n \to \infty$. Firstly, assume that $f_{0}(x|
\theta)$ is continuous and vanishes outside a compact set which is independent of $
\theta$ and $\Sigma$. For any large number $M$, we split $K=K'+K''$ where $K'=K 1_{\|
x|\ \leq M}$ and $K''=K 1_{\|x|\ >M}$.  Now, by using Young's inequality we obtain
\begin{eqnarray}
& & \int {|p_{G_{n},f_{0}}*K_{\sigma_{n}}(x)-p_{G_{n},f_{0}}(x)|}\textrm{d}x \nonumber \\
& & \hspace{8 em} \leq \int {\biggr|p_{G_{n},f_{0}}*K_{\sigma_{n}}'(x)-p_{G_{n},f_{0}}(x)\int{K'_{\sigma_{n}}(y)}dy\biggr|}\textrm{d}x \nonumber \\
& & \hspace{8 em} + \int {|p_{G_{n},f_{0}}*K''_{\sigma_{n}}(x)|}\textrm{d}x + \int {p_{G_{n},f_{0}}(x)}\textrm{d}x\int{K''_{\sigma_{n}}(x)}\textrm{d}x \nonumber \\
& & \hspace{8 em}  \leq \int \limits_{A}{\biggr|p_{G_{n},f_{0}}*K_{\sigma_{n}}(x)-p_{G_{n},f_{0}}(x)\int{K'_{\sigma_{n}}(y)}dy\biggr|}\textrm{d}x +2 \int{K''_{\sigma_{n}}(x)}\textrm{d}x. \nonumber
\end{eqnarray}
for some compact set $A$. It is clear that for any $\epsilon>0$, we can choose 
$M(\epsilon)$ such that as $M>M(\epsilon)$, ${\displaystyle \int{K''_{\sigma_{n}}(x)}
\textrm{d}x=\int{K''(x)}\textrm{d}x}<\epsilon$. Regarding the first term in the right hand side of the above 
display, by denoting $G_{n}=\sum \limits_{i=1}^{m_{n}}{p_{i}^{n}\delta_{\theta_{i}
^{n}}}$ we obtain
\begin{eqnarray}
& & \int \limits_{A}{\biggr|p_{G_{n},f_{0}}*K_{\sigma_{n}}(x)-p_{G_{n},f_{0}}(x)\int{K'_{\sigma_{n}}(y)}dy\biggr|}\textrm{d}x \nonumber \\
& & \hspace{6 em} \leq \int \limits_{A}{\int {|p_{G_{n},f_{0}}(x-y)-p_{G_{n},f_{0}}(x)| K'_{\sigma_{n}}(y)|}}dy\textrm{d}x \nonumber \\
& & \hspace{6 em} \leq \int \limits_{A}{\int {\sum \limits_{i=1}^{m_{n}}{p_{i}^{n}|f_{0}(x-y|\theta_{i}^{n})-f_{0}(x|\theta_{i}^{n})|}K'_{\sigma_{n}}(y)|}}dy\textrm{d}x  \nonumber \\
& & \hspace{6 em} \leq \omega(M\sigma_{n})\int \limits_{A}{\int {|K'_{\sigma_{n}}(y)|}}dy\textrm{d}x \leq \omega(M\sigma_{n})\mu(A) \to 0 \nonumber
\end{eqnarray}
where $\omega(t)=\mathop {\sup }\limits_{||x-y|| \leq t}{|f_{0}(x|\theta)-f_{0}(y|
\theta)|}$ denotes the modulus of continuity of $f_{0}$ and $\mu$ denotes the Lebesgue 
measure. Therefore, the conclusion of this lemma holds for that setting of $f_{0}(x|\theta)$.

Regarding the general setting of $f_{0}(x|\theta)$, for any $\epsilon>0$ since $\Theta$ 
is a bounded set, we can find a continous function $g(x|\theta)$ being supported on a 
compact set $B(\epsilon)$ that is independent of $\theta \in \Theta$ such that $
{\displaystyle \int {\left|f_{0}(x|\theta)-g(x|\theta)\right|}\textrm{d}x <\epsilon}$. Hence, we obtain
\begin{eqnarray}
\int {|p_{G_{n},f_{0}}*K_{\sigma_{n}}(x)-p_{G_{n},f_{0}}(x)|}\textrm{d}x & \leq & \int {\biggr|\biggr(p_{G_{n},f_{0}}-p_{G_{n},g}\biggr)*K_{\sigma_{n}}(x)\biggr|}\textrm{d}x \nonumber \\
& & \hspace{-8 em} + \int {|p_{G_{n},f_{0}}(x)-p_{G_{n},g}(x)|}\textrm{d}x+\int {|p_{G_{n},g}*K_{\sigma_{n}}(x)-p_{G_{n},g}(x)|}\textrm{d}x \nonumber \\
& \leq & 2\epsilon+\int {|p_{G_{n},g}*K_{\sigma_{n}}(x)-p_{G_{n},g}(x)|}\textrm{d}x \nonumber
\end{eqnarray}
where ${\displaystyle \int {|p_{G_{n},g}*K_{\sigma_{n}}(x)-p_{G_{n},g}(x)|}\textrm{d}x \to 0}$ as $n \to \infty$. We achieve the conclusion of the lemma.
\begin{lemma} \label{lemma:bandwith_vary_hellinger_bound}
Assume that $f_{0}$ and $K$ satisfy condition (P.1) in Theorem 
\ref{theorem:convergence_rate_mixing_measure}. Furthermore, $K$ has an integrable 
radial majorant $\Psi \in L_{1}(\mu)$
where $\Psi(x)=\mathop {\sup }\limits_{||y|| \geq ||x||}{|K(y)|}$. Then, we can find a positive constant $\epsilon_{1}^{0}$ such that as $\sigma \leq \epsilon_{1}^{0}$, for any $G \in \mathcal{O}_{k_{0}}$ we have
\begin{eqnarray}
h(p_{G,f_{0}}*K_{\sigma_{0}},p_{G_{0},f_{0}}*K_{\sigma_{0}}) \geq V(p_{G,f_{0}}*K_{\sigma_{0}},p_{G_{0},f_{0}}*K_{\sigma_{0}}) \gtrsim W_{1}(G,G_{0}). \nonumber
\end{eqnarray}
, i.e., $C_{1}(\sigma_{0}) \geq C_{1}$ as $\sigma_{0} \to 0$ where $C_{1}$ only depends on $G_{0}$.
\end{lemma}
\begin{proof}
We divide the proof of this lemma into two key steps
\paragraph{Step 1:} We firstly demontrate the following result
\begin{eqnarray}
\mathop {\lim }\limits_{\epsilon \to 0}{\mathop {\inf }\limits_{G \in \mathcal{O}_{k_{0}},\sigma_{0} >0}{\left\{\dfrac{V(p_{G,f_{0}}*K_{\sigma_{0}},p_{G_{0},f_{0}}*K_{\sigma_{0}})}{W_{1}(G,G_{0})}: \ W_{1}(G,G_{0}) \vee \sigma_{0} \leq \epsilon \right\}}} >0. \label{eqn:appendix_one}
\end{eqnarray}
The proof idea of the above inequality is essentially similar to that from the proof of 
Theorem 3.1 in \citep{Ho-Nguyen-EJS-16}. Here, we provide such proof for the 
completeness. Assume that the
conclusion of inequality \eqref{eqn:appendix_one} does not hold. Therefore, we can find two sequences $\left\{G_{n}\right\}$ and $\left\{\sigma_{0,n}\right\}$ such that $V(p_{G_{n},f_{0}}
*K_{\sigma_{0,n}},p_{G_{0},f_{0}}*K_{\sigma_{0,n}})/W_{1}(G_{n},G_{0}) \to 0$ where
$W_{1}(G_{n},G_{0}) \to 0$ and $\sigma_{0,n} \to 0$ as $n \to \infty$. As $G_{n} \in
\mathcal{O}_{k_{0}}$, it implies that there exists a subsequence $\left\{G_{n_{m}}\right\}$ of $\left\{G_{n}\right\}$ such that $G_{n_{m}}$ has exactly $k_{0}$ elements for all $m
$. Without loss of generality, we replace this subsequence by the whole sequence $\left\{G_{n}\right\}$. Now, we can represent $G_{n}$ as $G_{n}=\mathop {\sum }\limits_{i=1}
^{k_{0}}{p_{i}^{n}\delta_{\theta_{i}^{n}}}$ such that $(p_{i}^{n},
\theta_{i}^{n}) \to (p_{i}^{0},\theta_{i}^{0})$. Similar to the argument in Step
1 from the proof of Theorem 3.1 in \cite{Ho-Nguyen-EJS-16}, we have $W_{1}(G_{n},G_{0})
\lesssim d(G_{n},G_{0})$ where $d(G_{n},G_{0})=\sum \limits_{i=1}^{k_{0}}{p_{i}^{n}\|
\Delta \theta_{i}^{n}\|+|\Delta p_{i}^{n}|}$ and $\Delta p_{i}^{n}
=p_{i}^{n}-p_{i}^{0}, \Delta \theta_{i}^{n}=\theta_{i}^{n}-\theta_{i}^{0}$ for all $1 \leq i \leq k_{0}$. It implies that
$V(p_{G_{n},f_{0}}*K_{\sigma_{0,n}},p_{G_{0},f_{0}}*K_{\sigma_{0,n}})/
d(G_{n},G_{0}) \to 0$.

Now, we denote ${\displaystyle g_{n}(x|\theta)=\int {f_{0}(x-y|\theta)K_{\sigma_{0,n}}(y)}dy}$ for all $\theta \in \Theta$. Similar to Step 2 from the proof of Theorem 3.1 in \cite{Ho-Nguyen-EJS-16}, by means of Taylor expansion up to the first order we can represent
\begin{eqnarray}
\dfrac{p_{G_{n},f_{0}}*K_{\sigma_{0,n}}(x)-p_{G_{0},f_{0}}*K_{\sigma_{0,n}}(x)}{d(G_{n},G_{0})} \asymp \dfrac{1}{d(G_{n},G_{0})}\biggr(\sum \limits_{i=1}^{k_{0}}{\Delta p_{i}^{n}g_{n}(x|\theta_{i}^{0})+p_{i}^{n}\dfrac{\partial{g_{n}}}{\partial{\theta}}(x|\theta_{i}^{0})}\biggr) \nonumber
\end{eqnarray}
which is the linear combinations of the elements of $g_{n}(x|\theta_i^0), 
\dfrac{\partial{g_{n}}}{\partial{\theta}}(x|\theta_i^0)$ for $1 \leq i \leq k_0$. Denote 
$m_{n}$ to be the maximum of the absolute values of these coefficients. We can argue 
that $m_{n} \not \to 0$ as $n \to \infty$. Additionally, since $K$ has an integrable radial 
majorant $\Psi \in L_{1}(\mu)$, from Theorem 3 in Chapter 2 of \cite{Devroye-1985}, we 
have $g_{n}(x|\theta) \to f_{0}(x|\theta)$ and $\dfrac{\partial{g_{n}}}{\partial{\theta}}
(x|\theta) \to \dfrac{\partial{f_{0}}}{\partial{\theta}}(x|\theta)$ for almost surely $x$ 
and for any $\theta \in \Theta$. Therefore, we obtain
\begin{eqnarray}
\dfrac{1}{m_{n}}\dfrac{d_{n}\biggr(p_{G_{n},f_{0}}*K_{\sigma_{0,n}}(x)-p_{G_{0},f_{0}}*K_{\sigma_{0,n}}(x)\biggr)}{d(G_{n},G_{0})} \to \sum \limits_{i=1}^{k_{0}}{\alpha_{i}f_{0}(x|\theta_{i}^{0})+\beta_{i}^{T}\dfrac{\partial{f_{0}}}{\partial{\theta}}(x|\theta_{i}^{0})} \nonumber
\end{eqnarray}
where not all the elements of $\alpha_{i},\beta_{i}$ equal to 0. Due to the first order 
identifiability of $f_{0}$ and the Fatou's lemma, $V(p_{G_{n},f_{0}}
*K_{\sigma_{0,n}},p_{G_{0},f_{0}}*K_{\sigma_{0,n}})/d(G_{n},G_{0}) \to 0$ will lead 
to $\alpha_{i}=0,\beta_{i}=\vec{0} \in \mathbb{R}^{d_{1}}$ for all $1 \leq i \leq k_{0}$, 
which is a contradiction. We achieve the conclusion of \eqref{eqn:appendix_one}.
\paragraph{Step 2:} The result of \eqref{eqn:appendix_one} implies that we can find a 
positive number $\epsilon_{1}^{0}$ such that as $W_{1}(G,G_{0}) \vee \sigma_{0} \leq \epsilon_{1}^{0}$, we have
\begin{eqnarray}
h(p_{G,f_{0}}*K_{\sigma_{0}},p_{G_{0},f_{0}}*K_{\sigma_{0}}) \geq V(p_{G,f_{0}}*K_{\sigma_{0}},p_{G_{0},f_{0}}*K_{\sigma_{0}}) \gtrsim W_{1}(G,G_{0}).
\end{eqnarray}
In order to extend the above inequality to any $G \in \mathcal{O}_{k_{0}}$, it is sufficient to demonstrate that
\begin{eqnarray}
\mathop {\inf }\limits_{\sigma_{0}<\epsilon_{1}^{0}, W_{1}(G,G_{0})>\epsilon_{1}^{0}}{\dfrac{h(p_{G,f_{0}}*K_{\sigma_{0}},p_{G_{0},f_{0}}*K_{\sigma_{0}})}{W_{1}(G,G_{0})}}>0. \nonumber
\end{eqnarray}
In fact, if the above result does not hold, we can find two sequences $G_{n}' \in \mathcal{O}
_{k_{0}}$ and $\sigma_{0,n}'$ such that $W_{1}(G_{n}',G_{0})>\epsilon_{1}^{0}$, $
\sigma_{0,n}' \leq \epsilon_{1}^{0}$ and $h(p_{G_{n}',f_{0}}*K_{\sigma_{0,n}'},p_{G_{0},f_{0}}*K_{\sigma_{0,n}'})/W_{1}(G_{n}',G_{0}) \to 0$ as $n \to \infty$. Since $\Theta$ is closed bounded set, we can find two subsequences $\left\{G_{n_{m}}'\right\}$ and $\left\{\sigma_{0,n_{m}}'\right\}$ of $\left\{G_{n}'\right\}$ and
$\left\{\sigma_{n}'\right\}$ respectively such that $W_{1}(G_{n_{m}}',G') \to 0$ and $|
\sigma_{0,n_{m}}'-\sigma'| \to 0$ as $m \to \infty$ where $G' \in \mathcal{O}_{k_{0}}$
and $\sigma' \in [0,\epsilon_{1}^{0}]$.

Due to the first order Lipschitz continuity of $f*K_{\sigma_{0,n_{m}}'}$ for any $m
\geq 1$, we achieve that
\begin{eqnarray}
p_{G_{n_{m}',f_{0}}}*K_{\sigma_{0,n_{m}}'}(x) \to p_{G',f_{0}}
*K_{\sigma'}(x) \nonumber
\end{eqnarray}
for any $x \in \mathcal{X}$. Here, $p_{G',f_{0}}*K_{\sigma'}=
p_{G',f_{0}}$ when $\sigma'=0$. Therefore, by utilizing the Fatou's argument, we obtain
$h(p_{G',f_{0}}*K_{\sigma'},p_{G_{0},f_{0}}*K_{\sigma'})=0$, which implies that $G'
\equiv G_{0}$, a contradiction. As a consequence, when $\sigma_{0} \leq \epsilon_{1}^{0}$,
for any $G \in \mathcal{O}_{k_{0}}$ we have
\begin{eqnarray}
h(p_{G,f_{0}}*K_{\sigma_{0}},p_{G_{0},f_{0}}*K_{\sigma_{0}}) \geq V(p_{G,f_{0}}*K_{\sigma_{0}},p_{G_{0},f_{0}}*K_{\sigma_{0}}) \gtrsim W_{1}(G,G_{0}). \nonumber
\end{eqnarray}
We achieve the conclusion of the lemma.
\end{proof}
\begin{lemma} \label{lemma:singularity_level}
Assume that $\widehat{K}(t) \neq 0$ for almost all $t \in \mathbb{R}^{d}$ where $
\widehat{K}(t)$ is the Fourier transform of kernel function $K$. If $I(G_{0},f_{0})$ 
has $r$-th singularity level for some $r \geq 0$, then $I(G_{0},f_{0}*K_{\sigma_{0}})$ 
also has $r$-th singularity level for any $\sigma_{0}>0$.
\end{lemma}
\begin{proof}
Remind that, $I(G_{0},f_{0})$ has $r$-th singularity level is equivalent to the fact that 
$G_{0}$ is $r$-singular relative to the ambient space $\mathcal{O}_{k_{0}}$ (cf. 
Definition 3.1 and Definition 3.2 in \citep{Ho-Nguyen-Ann-17}). Now, for any $\rho \in 
\mathbb{N}$, given any sequence $G_{n}=\sum \limits_{i=1}^{k_{n}}{p_{i}^{n}
\delta_{\theta_{i}^{n}}} \in \mathcal{O}_{k_{0}}$ such that $G_{n} \to G_{0}$ in 
$W_{\rho}$ metric. We can find a subsequence of $G_{n}$ such that $k_{n}=k_{0}$ and 
each atoms of $G_{0}$ will have exactly one component of $G_{n}$ converges to. 
Without loss of generality, we replace the subsequence of $G_{n}$ by its whole sequence 
and relabel the atoms of $G_{n}$ such that $(p_{i}^{n},\theta_{i}^{n}) \to (p_{i}^{0},
\theta_{i}^{0})$ for all $1 \leq i \leq k_{0}$. Denote $\Delta \theta_{i}^{n}=\theta_{i}
^{n}-\theta_{i}^{0}$ and $\Delta p_{i}^{n}=p_{i}^{n}-p_{i}^{0}$ for all $1 \leq i \leq 
k_{0}$. From Definition 3.1 in \cite{Ho-Nguyen-Ann-17}, a $\rho$-minimal form of $G_{n}
$ from Taylor expansion up to the order $\rho$ satisfies
\begin{eqnarray}
\dfrac{p_{G_{n},f_{0}}(x)-p_{G_{0},f_{0}}(x)}{W_{\rho}^{\rho}(G_{n},G_{0})}=\sum_{l=1}^{T_\rho}
\biggr (\frac{\xi_{l}^{(\rho)}(G_{n})}{W_\rho^\rho(G_n,G_{0})} \biggr ) H_{l}^{(\rho)}(x) + o(1), \label{eqn:lemma_singularity_level_first}
\end{eqnarray}
for all $x$. Here, $H_{l}^{(\rho)}(x)$are linearly independent functions of $x$ for all $l$, and coefficients $\xi_{l}^{(\rho)}(G)$ are polynomials of
the components of $\Delta \theta_{i}$ and $\Delta p_{i}$ for $l$ ranges from 1 to a finite $T_\rho$. From the above representation, we achieve
\begin{eqnarray}
\dfrac{p_{G_{n},f_{0}*K_{\sigma_{0}}}(x)-p_{G_{0},f_{0}*K_{\sigma_{0}}}(x)}{W_{\rho}^{\rho}(G_{n},G_{0})}=\sum_{l=1}^{T_\rho}
\biggr (\frac{\xi_{l}^{(\rho)}(G_{n})}{W_\rho^\rho(G_n,G_{0})} \biggr ) H_{l}^{(\rho)}*K_{\sigma_{0}}(x) + o(1),
\label{eqn:lemma_singularity_level_second}
\end{eqnarray}
where $H_{l}^{(\rho)}*K_{\sigma_{0}}(x)={\displaystyle \int {H_{l}^{(\rho)}(x-
y)K_{\sigma_{0}}(y)}dy}$ for all $1 \leq l \leq T_{\rho}$. We will show that $H_{l}
^{(\rho)}*K_{\sigma_{0}}(x)$ are linearly independent functions of $x$ for all $1 \leq l 
\leq T_{\rho}$. In fact, assume that we can find the coefficients $\alpha_{l} \in 
\mathbb{R}$ such that
\begin{eqnarray}
\sum \limits_{l=1}^{T_{\rho}}{\alpha_{l}H_{l}^{(\rho)}*K_{\sigma_{0}}(x)}=0 \nonumber
\end{eqnarray}
for all $x$. By means of Fourier transformation in both sides of the above equation, we obtain
\begin{eqnarray}
\widehat{K}_{\sigma_{0}}(t)\biggr(\sum \limits_{l=1}^{T_{\rho}}{\alpha_{l}\widehat{H_{l}}^{(\rho)}(t)}\biggr)=0 \nonumber
\end{eqnarray}
for all $t \in \mathbb{R}^{d}$. As $\widehat{K}_{\sigma_{0}}(t)=\widehat{K}
(\sigma_{0} t) \neq 0$ for all $t \in \mathbb{R}^{d}$ and $H_{l}^{(\rho)}(x)$ for all $l$ 
are linearly independent functions of $x$ for all $1 \leq l \leq T_{\rho}$, the above 
equation implies that $\alpha_{l}=0$ for all $1 \leq l \leq T_{\rho}$. Therefore, $H_{l}
^{(\rho)}*K_{\sigma_{0}}(x)$ are linearly independent functions of $x$ for all $1 \leq l 
\leq T_{\rho}$ and $\rho \in \mathbb{N}$.

According to Definition 3.2 in \cite{Ho-Nguyen-Ann-17}, since $I(G_{0},f_{0})$ has $r$-th 
singularity level, it implies that for any sequence $G_{n} \in \mathcal{O}_{k_{0}}$ such 
that $W_{r+1}^{r+1}(G_{n},G_{0}) \to 0$, we do not have all the ratios $\xi_{l}^{(r+1)}
(G_{n})/W_{r+1}^{r+1}(G_0,G_{n})$ in \eqref{eqn:lemma_singularity_level_first} go to 0 
for $1 \leq l \leq T_{r+1}$. It means that not all the ratios $\xi_{l}^{(r+1)}(G_{n})/W_{r
+1}^{r+1}(G_0,G_{n})$ in \eqref{eqn:lemma_singularity_level_second} go to 0. 
Additionally, as $I(G_{0},f_{0})$ has $r$-th singularity level, we can find a sequence 
$G_{n}' \in \mathcal{O}_{k_{0}}$ such that $W_{r}^{r}(G_{n}',G_{0}) \to 0$ and $
\xi_{l}^{(r)}(G_{n}')/W_{r}^{r}(G_0,G_{n}')$ in \eqref{eqn:lemma_singularity_level_first} 
go to 0 for $1 \leq l \leq T_{r}$. It in turns also means that all the ratios $\xi_{l}^{(r)}
(G_{n}')/W_r^r(G_0,G_{n}')$ in \eqref{eqn:lemma_singularity_level_second} go to 0. As 
a consequence, from Definition 3.3 in \cite{Ho-Nguyen-Ann-17}, we achieve the conclusion 
of the lemma.
\end{proof}
\end{document}